\documentclass[11pt]{article}
\usepackage{amsmath,amsthm,amssymb,amsxtra, graphicx}
\usepackage{amsmath}
\usepackage{graphicx}
\usepackage{amsfonts}
\usepackage{amssymb}
\usepackage{color}
\usepackage{soul}

\date{}
\providecommand{\U}[1]{\protect\rule{.1in}{.1in}}
\theoremstyle{plain}

\newtheorem{corollary}{Corollary}

\newtheorem{lemma}{Lemma}

\newtheorem{proposition}{Proposition}
\newtheorem{remark}{Remark}

\newtheorem{theorem}{Theorem}
\numberwithin{equation}{section}

\def\func#1{\mathop{\rm #1}\nolimits}

\title{\bf Interpolation of Fredholm operators}
\author{I.~Asekritova, N.~Kruglyak and M.~Masty{\l}o}
\begin{document}
\maketitle
\renewcommand{\thefootnote}{\fnsymbol{footnote}}

\footnotetext{The third named author was supported by the
Foundation for Polish Science (FNP).}

\maketitle

\begin{flushright}
\emph{To Yuri Brudnyi on his 80th birthday}
\end{flushright}

\vspace{5 mm}

\begin{abstract}
\noindent We prove novel results on interpolation of Fredholm
operators including an abstract factorization theorem. The main
result of this paper provides sufficient conditions on the
parameters $\theta \in (0,1)$ and $q\in \lbrack 1,\infty ]$ under
which an operator $A$ is a~Fredholm operator from the real
interpolation space $(X_{0},X_{1})_{\theta ,q}$ to
$(Y_{0},Y_{1})_{\theta ,q} $ for a~given operator $A\colon
(X_{0},X_{1})\rightarrow (Y_{0},Y_{1})$ between compatible pairs
of Banach spaces such that its restrictions to the endpoint spaces
are Fredholm operators. These conditions are expressed in terms of
the corresponding indices generated by the $K$-functional of
elements from the kernel of the operator $A$ in the interpolation sum $%
X_{0}+X_{1}$.~If in addition $q\in \lbrack 1,\infty )$ and $A$ is
invertible operator on endpoint spaces, then these conditions are
also necessary.~We apply these results to obtain and present an
affirmative solution of the famous Lions-Magenes problem on the
real interpolation of closed subspaces. We also discuss some
applications to the spectral theory of operators as well as to
perturbation of the Hardy operator by identity on weighted
$L_{p}$-spaces.
\end{abstract}

\maketitle

\vspace{5 mm}

\section{Introduction}

Interpolation techniques that have been developed in recent years
are very powerful in studying various topics in Banach space as
well as operator space theory.~From the point of view of the
theory of operators it is useful to identify important properties
of operators which are stable for various interpolation
methods.~There are many properties of operators between Banach
spaces that are stable for the complex as well as for the real
method of interpolation.

In this paper we investigate the stability of Fredholm properties
on interpolation spaces. We notice that the class of Fredholm
operators is very important for several areas of mathematics,
including spectral theory of operators on Banach spaces and PDE's
(see, e.g., \cite{Ai, Vo}).

One of the fundamental questions in modern analysis is to
determine whether, for a~given Banach space operator $A\colon X\to
Y$, the range of $A$ is a~closed subspace of $Y$. For example, the
famous Fredholm alternative is based on the fact that if $X=Y$ and
$A = I-K$, where $K\colon X\to X$ is a~compact operator, then $A$
is a~Fredholm operator and so the range of $A$ is closed. In the
interpolation theory of operators this question leads to the
following challenging and still open problem posed by
J.-\,L.~Lions, whether the interpolation of closed subspaces of
a~compatible couple of Banach spaces leads to a~closed subspace of
an interpolation space? We note that from the remarkable result of
M.~Cwikel \cite{Cw} on one-sided interpolation of the compactness
property it follows that the Fredholm alternative is stable for
the real method of interpolation. This motivates the investigation
of stability of Fredholm properties on interpolation spaces. In
this paper we focus on real interpolation spaces, since for the
complex method even interpolation of the compactness property is
still unsolved despite considerable efforts (see, e.g., \cite{CK}
and \cite{CKM}).

It should be noted that there is no (even two-sided) interpolation
variant of Cwikel's result for Fredholm operators between real
interpolation spaces. In fact, we will see in Section $8$ that the
situation in the case of Fredholm operators even for couples of
Hilbert spaces is much more complicated than for compact
operators.

Generally notation will be introduced when needed. Nevertheless,
in this section we remind the reader of some standard notation and
state the conventions and definitions that will be used throughout
the paper.

Let $M$ and $N$ be linear subspaces of a~vector space $Y,$ as
usual $M+N$ will denote the sum of $M$ and $N$. If $M\cap N=\{0\}$
and $Y=M+N$, then $Y$ is called
the direct sum of $M$ and $N$; for this case we will use the notation $%
Y=M\oplus N$ and the space $N$ will be called an algebraic
complement of $M$. If $Y$ is a~normed space and $M$, $N$ are
closed subspaces of $Y$ such that $Y=M\oplus N$, then $Y$ is the
topological direct sum of $M$ and $N$ and $N$ will be called
a~topological complement of $M$. In this paper we will make
continual use of the fact that in a~Banach space every finite
dimensional subspace has a topological complement.

All operators between Banach spaces are assumed to be continuous and
linear.~Let $T\colon X\to Y$ be an operator between Banach spaces.~We say
that $T$ is \emph{upper semi-Fredholm} if the kernel $\ker _{X}T$ is finite
dimensional and the range $T(X)$ is closed. The operator $T\colon X\to Y$ is
called to be \emph{lower semi-Fredholm} if $T$ has finite codimensional
image (i.e., $\mathrm{codim}\,T:= \func{dim}Y/T(X)<\infty $). If $T$ is
upper and lower semi-Fredholm, then $T$ is called Fredholm operator; in this
case the integer $\mathrm{ind}(T)=\mathrm{ind}(T\colon X\to Y):=\mathrm{dim}(%
\mathrm{ker}\,T)-\mathrm{codim}\,T$ is called the \emph{index} of $T$.

We will often use that if $Y=T(X)\oplus N$, where $T\in L(X,Y)$
and $N$ is a~closed subspace of $Y$, then $T(X)$ is closed in $Y$
(see, e.g., \cite[Theorem 2.16, p.~76]{AB}).

For basic results and notation from interpolation theory we refer to \cite%
{BL} and \cite{BK}. We recall that a mapping ${\mathcal{F}}$ from the
category of compatible couples of Banach spaces into the category of Banach
spaces is said to be an \emph{interpolation functor} if, for any couple $%
\vec{X}=(X_{0},X_{1})$, ${\mathcal{F}}(\vec{X})$ is a Banach space
intermediate with respect to $\vec{X}$ (i.e., $X_{0}\cap X_{1}\subset {%
\mathcal{F}}(\vec{X})\subset X_{0}+X_{1}$), and $A\colon {\mathcal{F}}(\vec{X%
})\rightarrow {\mathcal{F}}(\vec{Y})$ for all Banach couples $\vec{X}$, $%
\vec{Y}$ and every $A\colon \vec{X}\rightarrow \vec{Y}$. Here as usual we
use the notation $A\colon \vec{X}\rightarrow \vec{Y}$ to mean that $A\colon
X_{0}+X_{1}\rightarrow Y_{0}+Y_{1}$ is a~linear map such that the
restrictions $A|_{X_{j}}\in L(X_{j},Y_{j})$ $(j=0,1)$. We let $L(\vec{X},%
\vec{Y})$ denote the Banach space of all operators $A\colon \vec{X}%
\rightarrow \vec{Y}$, equipped with the norm $\Vert A\Vert _{\vec{X}%
\rightarrow \vec{Y}}:=\max_{j=0,1}\,\Vert A|_{X_{j}}\Vert _{X_{j}\rightarrow
Y_{j}}$.

Suppose we are given an operator $A\colon (X_{0},X_{1})\rightarrow
(Y_{0},Y_{1})$. We say that $A$ is \emph{invertible} (resp., \emph{%
surjective, Fredholm, lower semi-Fredholm}, \emph{upper semi-Fredholm}) if
the restrictions $A|_{X_{0}}\colon X_{0}\rightarrow Y_{0}$ and $%
A|_{X_{1}}\colon X_{1}\rightarrow Y_{1}$ on endpoint spaces are
invertible (resp., surjective, Fredholm, lower semi-Fredholm,
upper semi-Fredholm). It will cause no confusion to write for
short $A\colon (X_{0},X_{1})\rightarrow (Y_{0},Y_{1})$ is
invertible (resp., surjective, Fredholm, lower semi-Fredholm,
upper semi-Fredholm).

In the present paper we work with the \emph{real method} of
interpolation. We recall that for any Banach couple $\vec
{X}=(X_0,X_1)$ and $t>0$ the \emph{Peetre $K$-functional} is
defined by:
\begin{equation*}
K(t,x;\vec {X})=K(t,x; X_0, X_1):= \func{inf}_{x=x_0+x_1}\{\|x_0\|_{X_0}+t%
\|x_1\|_{X_1}\},\quad \,x\in X_{0}+X_{1}.
\end{equation*}
If $0<\theta <1$ and $1\leq q\leq \infty $, then the real interpolation
space $\vec {X}_{\theta ,q}$ consists of all $x\in X_0+X_1$ such that
\begin{equation*}
\|x\|_{\vec {X}_{\theta ,q}}=\bigg (\int _{0}^{\infty
}\big(t^{-\theta }K(t,x; \vec {X})\big)^{q}\frac{dt}{t}\bigg
)^{1/q}<\infty
\end{equation*}
with obvious modification for $q=\infty $. We will use the well
known fact that the space $\vec {X}_{\theta ,q}$ (with equivalent
norm) can be defined in terms of $J$-\emph{functional} given by
\begin{equation*}
J(t,x;\vec {X})=J(t,x;X_0,X_1):=\max \{\|x\|_{X_0},t\|x\|_{X_1}\},\quad
\,x\in X_{0}\cap X_{1}.
\end{equation*}
Specifically, $\vec {X}_{\theta ,q}$ consists of all $x\in X_0+X_1$
representable in the form
\begin{equation*}
x=\sum _{k=-\infty }^{\infty }x_k\,\quad \,{\text {(convergence in $X_0+X_1)$%
}}
\end{equation*}
with the norm
\begin{equation*}
\|x\|= \func{inf}\bigg (\sum _{k=-\infty }^{\infty }\big(2^{-k\theta }J(2^k,x_k;\vec {X}%
)\big)^{q}\bigg)^{1/q}.
\end{equation*}


\section{Statement of the main results}

This section comprises our main results on interpolation of Fredholm operators.
We~note that the problem of invertibility of operators between interpolation
spaces is studied in the paper \cite{AK1}; the main result proved in this paper
gives a~description of parameters $\theta \in (0,1)$ and $q\in
\lbrack 0,\infty ]$ for which $A\colon (X_{0},X_{1})_{\theta
,q}\rightarrow (Y_{0},Y_{1})_{\theta ,q}$ is invertible provided
$A\colon (X_{0},X_{1})\rightarrow (Y_{0},Y_{1})$ is invertible on
endpoint spaces.

The main aim of this paper is to study interpolation of Fredholm operators.
Throughout we usually try to present results in as general form as possible.
We are concerned in the following~challenging problem: \emph{Given Banach couples $%
\vec{X}=(X_{0},X_{1})$ and $\vec{Y}=(Y_{0},Y_{1})$ and an operator
$A\colon (X_{0},X_{1})\rightarrow (Y_{0},Y_{1})$, which is a Fredholm operator
on endpoint spaces. Find necessary and sufficient conditions on parameters
$\theta \in (0,1)$, $q\in \lbrack 1,\infty ]$ such that the operator $A\colon
(X_{0},X_{1})_{\theta ,q}\rightarrow (Y_{0},Y_{1})_{\theta ,q}$ is
Fredholm.}

The study of the stability of Fredholm properties on complex
interpolation scales was initiated by \v{S}ne\u{\i}berg
\cite{Sh} and continued, e.g., in \cite{CSS}, \cite{CS} and
\cite{KaM}. We wish to mention here that Herrero \cite{He},
extending the result due to Barnes \cite{Barnes} for ordered
couples of $L_{p}$-spaces, proved the following stability result
on interpolation of Fredholm operators in the case of ordered
regular Banach couple $(X_{0},X_{1})$, i.e., $X_{0}\hookrightarrow
X_{1}$ with $X_{0}$ dense in $X_{1}$, which states that if
$T\colon (X_{0},X_{1})\rightarrow (X_{0},X_{1})$ is a Fredholm
operator with $\mathrm{ind}(T\colon
X_{0}\rightarrow X_{0})=\mathrm{ind}(T\colon X_{1}\rightarrow X_{1})$, then $%
T$ is also Fredholm on any interpolation space $X$ between $X_{0}$ and $%
X_{1} $ and $\mathrm{ind}(T\colon X\rightarrow
X)=\mathrm{ind}(T\colon X_{0}\rightarrow X_{0})$. However, these
assumptions are too restrictive. We would like to remark that even
for an operator which corresponds to the integral equation of the
second kind on weighted $L^{2}$-spaces the situation is complicated;
it is shown in the last section of the paper that there exists an
operator $A$ which is invertible on endpoint spaces $L^{2}(\omega _{0})$, $%
L^{2}(\omega _{1})$, however for a~given $\theta \in (0,1)$ the operator $A$
restricted to the space $\left( L^{2}(\omega _{0}),L^{2}(\omega _{1})\right)
_{\theta ,2}=L^{2}(\omega _{0}^{1-\theta }\omega _{1}^{\theta })$ is either
Fredholm or not Fredholm.

Thus it is of interest to have information how the property of being
a~Fredholm operator varies with parameters of a~given interpolation
method. It should be pointed out that in the current literature there are no general
abstract results  related to this problem.

Let us now summarize some of the main results of the paper. In Section 3,
we show that the problem on interpolation of Fredholm operators between Banach couples
can be reduced to surjective operators on corresponding Banach couples. More precisely,
we show (see Lemma \ref{reduction}) that if an operator $A\colon
(X_{0},X_{1})\rightarrow (Y_{0},Y_{1})$ between Banach couples is
Fredholm,
then it is possible to construct in natural way a Banach couple $(\widetilde{%
X}_{0},\widetilde{X}_{1})$ and an operator $\tilde{A}\colon (\widetilde{X}%
_{0},\widetilde{X}_{1})\rightarrow (Y_{0},Y_{1})$ such that $\tilde{A}\colon
\widetilde{X}_{i}\rightarrow Y_{i}$ is surjective for $i=0,1$. Moreover
operator $A\colon {\mathcal{F}}(X_{0},X_{1})\rightarrow {\mathcal{F}}%
(Y_{0},Y_{1})$ is Fredholm operator if and only if $\tilde{A}\colon {%
\mathcal{F}}(\widetilde{X}_{0},\widetilde{X}_{1})\rightarrow {\mathcal{F}}%
(Y_{0},Y_{1})$ is Fredholm operator. Thus for simplicity of
presentation we assume that $A\colon (X_{0},X_{1})\rightarrow
(Y_{0},Y_{1})$ is a~surjective Fredholm operator.

Furthermore, we provide sufficient conditions on parameters
$\theta $, $q$ under which the operator $A\colon
(X_{0},X_{1})_{\theta ,q}\rightarrow (Y_{0},Y_{1})_{\theta ,q}$ is
Fredholm. We also prove that if $q\in \lbrack 1,\infty )$, then
these conditions are necessary in the case when $A\colon
(X_{0},X_{1})\rightarrow (Y_{0},Y_{1})$ is invertible on endpoint
spaces.

Let us fix parameters $\theta \in (0,1)$ and $q\in \left[1,\infty
\right]$. In the description of the main results the key role is played by
the kernel of the operator $A\colon X_{0}+X_{1}\rightarrow
Y_{0}+Y_{1}$,
\begin{align*}
{\text{ker}}_{X_{0}+X_{1}}A=\left\{ x\in
X_{0}+X_{1};\,Ax=0\right\},
\end{align*}%
and its two linear subspaces  defined as
\begin{align}
V_{\theta ,q}^{0}& :=\left\{ x\in \ker
_{X_{0}+X_{1}}A;\,\,\int_{0}^{1} \big(t^{-\theta}
K(t,x; X_{0},X_{1})\big)^{q}\frac{dt}{t}<\infty \right\},  \label{L2} \\
V_{\theta ,q}^{1}& :=\left\{ x\in \ker _{X_{0}+X_{1}}A;\,\,\int_{1}^{\infty
}\big(t^{-\theta} K(t,x; X_{0},X_{1})\big)^{q}\frac{dt}{t}%
<\infty \right\}, \label{L3}
\end{align}%
with obvious modification for $q=\infty $.~Note that the kernel of
$A$ in the space $(X_{0},X_{1})_{\theta ,q}$ is equal to
$V_{\theta ,q}^{0}\cap V_{\theta ,q}^{1}$ and moreover it is
finite-dimensional provided $A\colon (X_{0},X_{1})_{\theta
,q}\rightarrow (Y_{0},Y_{1})_{\theta ,q}$ is a~Fredholm operator.

In Section 3 we prove that if $A\colon (X_{0},X_{1})\rightarrow
(Y_{0},Y_{1})$ is a~Fredholm operator such that
\begin{align*}
\ker_{X_{0}+X_{1}}A=(V_{\theta ,q}^{0}+V_{\theta ,q}^{1})\oplus
\widetilde{V},
\end{align*}%
then $\widetilde{V}$ is finite-dimensional and its dimension is
equal to the codimension of $A\left( (X_{0},X_{1})_{\theta
,q}\right) $ in $(Y_{0},Y_{1})_{\theta ,q}$ provided $%
1\leq q<\infty $.

We point out that finite dimension of the spaces $V_{\theta
,q}^{0}\cap V_{\theta ,q}^{1}$ and $\widetilde{V}$ are necessary
but not sufficient conditions for $A\colon
(X_{0},X_{1})_{\theta ,q}\rightarrow (Y_{0},Y_{1})_{\theta ,q}$ to
be a~Fredholm operator.

To discuss further results more precisely we need to define
special operators. Let $A_{1}$ be a~quotient operator
\begin{equation*}
A_{1}\colon X_{0}+X_{1}\rightarrow (X_{0}+X_{1})/(V_{\theta ,q}^{0}\cap
V_{\theta ,q}^{1}),
\end{equation*}%
which maps the couple $(X_{0},X_{1})$ to the quotient couple $%
(Z_{0},Z_{1})=\left( A_{1}(X_{0}\right) ,A_{1}(X_{1}))$, and let
$A_{2}$ be a~quotient operator
\begin{equation*}
A_{2}\colon A_{1}(X_{0})+A_{1}(X_{1})\rightarrow
(A_{1}(X_{0})+A_{1}(X_{1}))/A_{1}(\widetilde{V}),
\end{equation*}%
which maps couple $\left( Z_{0},Z_{1}\right) $ to $\left( W_{0},W_{1}\right)
=\left( A_{2}(Z_{0}\right) ,A_{2}(Z_{1}))=\left( A_{2}A_{1}(X_{0}\right)
,A_{2}A_{1}(X_{1}))$.

It is not difficult to show that
\begin{align}
A=A_{3}A_{2}A_{1}{\text {,}}  \label{L1}
\end{align}
where $A_{3}$ is uniquely defined operator from the couple $%
\left
(W_{0},W_{1}\right )$ to the couple $(Y_{0},Y_{1})$ with \ the kernel $%
A_{2}A_{1}(\ker _{X_{0}+X_{1}}A)$.

Observe that if $A\colon (X_{0},X_{1})\rightarrow (Y_{0},Y_{1})$
is a~Fredholm and surjective operator on endpoint spaces, then the
constructed operators $A_{1}$, $A_{2}$ and $A_{3}$ are also
Fredholm and surjective operators on endpoint spaces. Moreover, if
$A\colon (X_{0},X_{1})_{\theta ,q}\rightarrow
(Y_{0},Y_{1})_{\theta ,q}$ is a~Fredholm operator, then it
can be shown (see Theorem \ref{TN2} on factorization) that the operators $%
A_{1}$, $A_{2}$ and $A_{3}$ constructed above have the following properties:

a)~The operator $A_{1}\colon (X_{0},X_{1})_{\theta ,q}\rightarrow
(Z_{0},Z_{1})_{\theta ,q}$ is surjective and the kernel $\ker
_{X_{0}+X_{1}}A_{1}$ is finite dimensional and is contained in $%
(X_{0},X_{1})_{\theta ,q}$;

b)~The operator $A_{2}\colon (Z_{0},Z_{1})_{\theta ,q}\rightarrow
(W_{0},W_{1})_{\theta ,q}$ is injective and $A_{2}((Z_{0},Z_{1})_{\theta
,q}) $ is a~closed subspace of $(W_{0},W_{1})_{\theta ,q}$ of finite
codimension, which is equal to the dimension of its kernel in the sum, i.e., $%
\ker _{Z_{0}+Z_{1}}A_{2}$ (it is also equal to codimension of $%
A((X_{0},X_{1})_{\theta ,q})$ in $(Y_{0},Y_{1})_{\theta ,q}$);

c)~The operator $A_3\colon (W_0,W_1)_{\theta ,q}\to (Y_0,Y_1)_{\theta ,q}$
is invertible.


The decomposition $A=A_{3}A_{2}A_{1}$ and properties a)-c) of operators $A_{1}$,
$A_{2}$, $A_{3}$ lead to the investigations of the following three classes
of Fredholm operators denoted by ${\mathbb{F}}_{\theta ,q}^{1}$, ${\mathbb{F}%
}_{\theta ,q}^{2}$ and ${\mathbb{F}}_{\theta ,q}^{3}$.

The class ${\mathbb{F}}_{\theta ,q}^{1}$ consists of all
surjective operators $A\colon (X_{0},X_{1})\rightarrow
(Y_{0},Y_{1})$ such that \linebreak $\dim (\ker
_{X_{0}+X_{1}}A)<\infty $, $A\colon (X_{0},X_{1})_{\theta
,q}\rightarrow (Y_{0},Y_{1})_{\theta ,q}$ is surjective and $\ker
_{X_{0}+X_{1}}A\subset (X_{0},X_{1})_{\theta ,q}$.

The class ${\mathbb{F}}_{\theta ,q}^{2}$ consists of all
surjective operators $A\colon (X_{0},X_{1})\!\rightarrow\!
(Y_{0},Y_{1})$ such that \linebreak $\dim (\ker
_{X_{0}+X_{1}}A)<\infty $, $A\colon (X_{0},X_{1})_{\theta
,q}\rightarrow (Y_{0},Y_{1})_{\theta ,q}$ is injective and
\[
{\text{dim}}\big((Y_{0},Y_{1})_{\theta ,q}/A((X_{0},X_{1})_{\theta
,q})\big) =\text{dim}\,(\ker _{X_{0}+X_{1}}A).
\]
The class ${\mathbb{F}}_{\theta ,q}^{3}$ consists of all
surjective Fredholm operators $A\colon (X_{0},X_{1})\rightarrow
(Y_{0},Y_{1})$ such that $A\colon (X_{0},X_{1})_{\theta
,q}\rightarrow (Y_{0},Y_{1})_{\theta ,q}$ is invertible. For this
class we do not require that $\ker _{X_{0}+X_{1}}A$ is finite
dimensional.

Note that if $A\colon (X_{0},X_{1})_{\theta ,q}\rightarrow
(Y_{0},Y_{1})_{\theta ,q}$ is a~Fredholm operator, then $A_1\in
{\mathbb{F}}_{\theta
,q}^{1}$, $A_2\in {\mathbb{F}}_{\theta ,q}^{2}$ and $A_3\in {\mathbb{F}}%
_{\theta ,q}^{3}$.~In order to describe results we need to recall the
definition of generalized dilation indices introduced in \cite{AK1}. Let $%
\vec {X}=(X_0,X_1)$ be a~Banach couple and $\Omega $ be a~subset of $X_0+X_1$%
.We denote by $\beta (\Omega )$ (resp., $\beta _0(\Omega )$, and
$\beta _{\infty }(\Omega )$) the infimum of all $\theta \in \left
[0,1\right ]$ for which there exists $\gamma =\gamma (\theta
,\Omega )>0$ such that
\begin{align}
s^{-\theta }\,K(s,x;\vec {X})\geq \gamma \,t^{-\theta }\,K(t,x;\vec {X})
\label{L7}
\end{align}
for all $x\in \Omega $ and all $0<s\leq t$ (resp., all $0<s\leq t\leq 1$,
and $1\leq s\leq t<\infty $).

We also define $\alpha (\Omega )$ (resp., $\alpha _0(\Omega )$, and $\alpha
_{\infty }(\Omega )$) to be the supremum of all $\theta \in
\left
[0,1\right
]$ for which there exists $\gamma =\gamma (\theta ,\Omega
)>0$ such that
\begin{align}
s^{-\theta }\,K(s,x;\vec {X})\leq \gamma \,t^{-\theta }\,K(t,x;\vec {X})
\label{L7_second}
\end{align}
for all $x\in \Omega $ and all $0<s\leq t$ (resp., all $0<s<t\leq 1$, and
all $1\leq s<t<\infty $). If the set $\Omega $ consists of one element,
i.e., $\Omega =\left \{x\right \}$, then the indices $\alpha (\Omega ),\beta
(\Omega )$ will be denoted by $\alpha (x),\beta (x)$ (similarly we write $%
\alpha _{0}(x),\beta _{0}(x),\alpha _{\infty }(x),\beta _{\infty }(x)$).

The following theorem provides sufficient conditions under which an operator
$A\colon (X_0,X_1)\to (Y_{0},Y_{1})$ belongs to one of the classes ${\mathbb{%
F}}_{\theta ,q}^{1},{\mathbb{F}}_{\theta ,q}^{2}$, ${\mathbb{F}}_{\theta
,q}^{3}$.


\begin{theorem}
\label{TN3} Suppose that an operator $A\colon
(X_{0},X_{1})\rightarrow (Y_{0},Y_{1})$ is surjective and Fredholm
on the endpoint spaces. Then the following statements are
true{\rm:}
\end{theorem}

\textrm{a)} \emph{If}
\begin{equation}
\beta _{\infty }(\mathrm{ker}_{X_{0}+X_{1}}A)<\theta <\alpha _{0}(\mathrm{ker%
}_{X_{0}+X_{1}}A),  \label{K1}
\end{equation}
then $\mathrm{ker}_{X_{0}+X_{1}}A\subset (X_{0},X_{1})_{\theta ,q}$ and $%
A\colon (X_{0},X_{1})_{\theta ,q}\rightarrow (Y_{0},Y_{1})_{\theta
,q}$ is a~surjective operator, i.e., $(\ref{K1})$ and
$\func{dim}\,(\ker_{X_{0}+X_{1}}A)<\infty$ imply $A\in
{\mathbb{F}}_{\theta ,q}^{1}$.

\vspace{1.5 mm}

\textrm{b)} \emph{If }$\func{dim}\,(\ker _{X_{0}+X_{1}}A)<\infty $\emph{\ and%
}
\begin{equation}
\beta _{0}(\mathrm{ker}_{X_{0}+X_{1}}\,A)<\theta <\alpha _{\infty }(\mathrm{%
ker}_{X_{0}+X_{1}}\,A),  \label{K2}
\end{equation}%
\emph{then $A\colon (X_{0},X_{1})_{\theta ,q}\rightarrow
(Y_{0},Y_{1})_{\theta ,q}$ is injective and codimension of $%
A((X_{0},X_{1})_{\theta ,q})$ in $(Y_{0},Y_{1})_{\theta ,q}$ is equal to the
dimension of $\mathrm{ker}_{X_{0}+X_{1}}\,A$, i.e., $A\in {\mathbb{F}}%
_{\theta ,q}^{2}$.}

\textrm{c)} \emph{If $\mathrm{ker}_{X_{0}+X_{1}}\,A=V_{\theta
,q}^{0}+V_{\theta ,q}^{1}$ and}
\begin{equation}
\beta (V_{\theta ,q}^{1})<\theta <\alpha \left( V_{\theta ,q}^{0}\right) ,
\label{K3}
\end{equation}%
\emph{then the operator $A\colon (X_{0},X_{1})_{\theta ,q}\rightarrow
(Y_{0},Y_{1})_{\theta ,q}$ is invertible, i.e., $A\in {\mathbb{F}}_{\theta
,q}^{3}$.}

\vspace{2mm}

In connection with this result an important question arises whether the
conditions (\ref{K1})-(\ref{K3}) are also necessary. The next theorem
shows that it is so if we suppose that the operator $A\colon
(X_{0},X_{1})\rightarrow (Y_{0},Y_{1})$\ is invertible on endpoint
spaces.


\begin{theorem}
\label{TN4} Suppose that an operator $A\colon
(X_{0},X_{1})\rightarrow (Y_{0},Y_{1})$ is invertible on endpoint
spaces. Then conditions $({\ref{K1}})$, $({\ref{K2}})$ and
$({\ref{K3}})$ are necessary conditions for $A$ to belong to the
classes ${\mathbb{F}}_{\theta ,q}^{1}$, ${\mathbb{F}}_{\theta
,q}^{2}$ and ${\mathbb{F}}_{\theta ,q}^{3}$, respectively.
\end{theorem}

Note that Theorems \ref{TN3}-\ref{TN4} are true for all $q\in \left[
1,\infty \right] $. If we combine these theorems with factorization Theorem %
\ref{TN2}, then we immediately obtain necessary and sufficient
conditions for an operator $A\colon (X_{0},X_{1})\rightarrow
(Y_{0},Y_{1})$ to be a~Fredholm operator from
$(X_{0},X_{1})_{\theta ,q}$ to $(Y_{0},Y_{1})_{\theta ,q}$.


\begin{corollary}
\label{TN5} Suppose that an operator $A\colon
(X_{0},X_{1})\rightarrow (Y_{0},Y_{1})$ is invertible on endpoint
spaces and let $1\leq q<\infty $. Then the following conditions
are necessary and sufficient for \linebreak $A\colon
(X_{0},X_{1})_{\theta ,q}\rightarrow (Y_{0},Y_{1})_{\theta ,q}$ to
be a~Fredholm operator\textrm{:}
\end{corollary}

\textrm{a)} \emph{The spaces $V_{\theta ,q}^{0}\cap V_{\theta ,q}^{1}$ and $%
\widetilde{V}$ are finite dimensional;}

\textrm{b)} \emph{The indices of the spaces $V_{\theta ,q}^{0}\cap V_{\theta
,q}^{1}$, $A_{1}(\widetilde {V})$, $A_{2}A_{1}(V_{\theta ,q}^{0})$, $%
A_{2}A_{1}(V_{\theta ,q}^{1})$ satisfy inequalities}
\begin{equation}
\beta _{\infty }(V_{\theta ,q}^{0}\cap V_{\theta ,q}^{1})<\theta <\alpha
_{0}(V_{\theta ,q}^{0}\cap V_{\theta ,q}^{1}),  \label{L14}
\end{equation}

\begin{equation}
\beta _{0}(A_{1}(\widetilde {V}))<\theta <\alpha _{\infty }(A_{1}(%
\widetilde {V})),  \label{L15}
\end{equation}

\begin{equation*}
\beta (A_{2}A_{1}(V_{\theta ,q}^{1}))<\theta <\alpha \left(
A_{2}A_{1}(V_{\theta ,q}^{0})\right) ,
\end{equation*}%
\textit{where the indices are calculated with respect to Banach couples }$%
(X_{0},X_{1})$\textit{, }$(A_{1}(X_{0}),A_{1}(X_{1}))$\textit{\ and }$%
(A_{2}A_{1}(X_{0}),A_{2}A_{1}(X_{1}))$\textit{, respectively. }

\vspace{2.5 mm}

The necessity part of Corollary \ref{TN5} explains why the problem
of interpolation of Fredholm operators is rather difficult even
for the real method.~We notice here that in the case of the
complex method, even interpolation of compact operators is not
fully understood.

We also remark that the condition $\text{dim}(V_{\theta, q}^0 \cap
V_{\theta, q}^1)<\infty$ implies that the kernel of an operator
$A\colon (X_0, X_1)_{\theta, q} \to (Y_0, Y_1)_{\theta, q}$ is
finite dimensional and all the other conditions  in Corollary
\ref{TN5} ensure that $A((X_0, X_1)_{\theta, q})$ is a~closed
subspace of finite codimension in the space $(Y_0, Y_1)_{\theta,
q}$. \linebreak We note that condition $1\leq q<\infty$ appeared
in this corollary because it is required in the factorization
theorem.

It should be pointed out that Corollary $1$ provides an algorithm
that allows one to check for given $\theta \in (0, 1)$ and $q \in
[1, \infty)$ whether the weak analog of the long-standing problem
of interpolation of subspaces by the real method has a positive or
negative solution (see section $8$ for details).

It is also worth stressing that the problem of interpolation of
subspaces by interpolation methods was posed by J.-\,L.~Lions and
E.~Magenes in their monograph \cite{LM}. In connection with
important applications of the well-known methods in the theory of
interpolation (the method of traces, the method of means and the
complex method of interpolation), the mentioned problem was
repeated by J.-\,L.~Lions in Notices Amer.~Math.~Soc.~22~(1975),
pp.~124--126 in section related to Problems in interpolation of
operators and applications.

We conclude with the following remark that results on
interpolation of subspaces found deep applications in various
problems of modern analysis. We hope that our results will find
further important applications in many other areas including
PDE's. In the current paper we do not concentrate on specific
applications, which would require us to prove some involved technical
results. This will done in detail in a~separate paper.

\section{Some basic general results}

In this section we will prove some general results on interpolation of
Fredholm operators which will be used throughout the paper.

\subsection{Reducing the problem to the case of surjective operator}

An interpolation functor ${\mathcal{F}}$ is said to be \emph{admissible}
provided for any Banach couple $(X_{0},X_{1})$, ${\mathcal{F}}(X_{0},X_{1})$
is contained in the closure of $X_{0}\cap X_{1}$ in $X_{0}+X_{1}$. We note
that if $\varphi \colon (0,\infty )\times (0,\infty )\rightarrow (0,\infty )$
is a fundamental function of the functor (i.e., ${\mathcal{F}}(s{\mathbb{R}}%
,t{\mathbb{R}})=\varphi (s,t){\mathbb{R}}$ for all $s,t>0$), then it can be
shown that ${\mathcal{F}}$ is admissible provided $\lim_{s\rightarrow
0}\varphi (s,1)=\lim_{t\rightarrow 0}\varphi (1,t)=0$. Real $\left( \cdot
\right) _{\theta ,q}$ functors and complex $\left[ \cdot \right] _{\theta }$
functors are admissible functors.

Now we state and prove a~lemma that is interesting on its own. It shows that
interpolation of Fredholm operators between Banach couples can be reduced to
surjective operators on corresponding Banach couples. The proof is based on
some simple trick from \cite{ACK}.

\begin{lemma}
\label{reduction} Let $(X_{0},X_{1})$ and $(Y_{0},Y_{1})$ be
Banach couples and let $A\colon (X_{0},X_{1})\rightarrow
(Y_{0},Y_{1})$ be an operator such that $A(X_0)$ and $A(X_1)$ are
complemented subspaces in $Y_0$ and $Y_1$, respectively. Then
there exist a Banach couple
$(\widetilde{X}_{0},\widetilde{X}_{1})$ and an operator
$\tilde{A}\colon (\widetilde{X}_{0},\widetilde{X}_{1})\rightarrow
(Y_{0},Y_{1})$ such that $\tilde{A}\colon
\widetilde{X}_{i}\rightarrow Y_{i}$ is surjective for $i=0$, $1$.
If in addition $A\colon (X_{0},X_{1})\rightarrow (Y_{0},Y_{1})$ is
a~Fredholm operator and ${\mathcal{F}}$ is an admissible
interpolation functor, then $A\colon
{\mathcal{F}}(X_{0},X_{1})\rightarrow
{\mathcal{F}}(Y_{0},Y_{1})$ is Fredholm operator if and only if $\tilde{A}%
\colon {\mathcal{F}}(\widetilde{X}_{0},\widetilde{X}_{1})\rightarrow {%
\mathcal{F}}(Y_{0},Y_{1})$ is Fredholm operator.
\end{lemma}

\begin{proof} It follows by hypothesis that we can write $Y_{i}$ as the
topological direct sum $Y_{i}=A(X_{i})\oplus M_{i}$ for $i=0,1$. Let
\begin{equation*}
\widetilde{X}_{0}:=X_{0}\times M_{0}\times \{0\},\quad \,\widetilde{X}%
_{1}:=X_{1}\times \{0\}\times M_{1}
\end{equation*}%
be Banach spaces equipped with the norms
\begin{equation*}
\Vert (x_{0},m_{0},0)\Vert _{\widetilde{X}_{0}}=\Vert x_{0}\Vert
_{X_{0}}+\Vert m_{0}\Vert _{Y_{0}},\quad \,\Vert (x_{1},0,m_{1})\Vert _{%
\widetilde{X}_{1}}=\Vert x_{1}\Vert _{X_{1}}+\Vert m_{1}\Vert _{Y_{1}}.
\end{equation*}%
Now, we define an operator $\tilde{A}\colon (\widetilde{X}_{0},\widetilde{X}%
_{1})\rightarrow (Y_{0},Y_{1})$ by the formula
\begin{equation*}
\tilde{A}(x,m_{0},m_{1})=A(x)+m_{0}+m_{1},\quad \,(x,m_{0},m_{1})\in
(X_{0}+X_{1})\times M_{0}\times M_{1}.
\end{equation*}%
Obviously $\tilde{A}(\widetilde{X}_{i})=Y_{i}$, and the operator $\tilde{A}%
\colon (\widetilde{X}_{0},\widetilde{X}_{1})\rightarrow (Y_{0},Y_{1})$ is
a~Fredholm on endpoint spaces. To this end, notice that our hypothesis that $%
{\mathcal{F}}$ is admissible yields that
\begin{equation*}
{\mathcal{F}}(\widetilde{X}_{0},\widetilde{X}_{1})={\mathcal{F}}%
(X_{0},X_{1})\times \{0\}\times \{0\}.
\end{equation*}%
In consequence the operator $\tilde{A}\colon {\mathcal{F}}(\widetilde{X}_{0},%
\widetilde{X}_{1})\rightarrow {\mathcal{F}}(Y_{0},Y_{1})$ is Fredholm if and
only if the operator $A\colon {\mathcal{F}}(X_{0},X_{1})\rightarrow {%
\mathcal{F}}(Y_{0},Y_{1})$ is Fredholm.
\end{proof}

\subsection{The decomposition property and the space $\protect\widetilde{V}$}

Let $A$ be an intermediate space with respect to a~couple
$(A_{0},A_{1})$ of Banach spaces. We say that $A$ has the
decomposition property on the couple $(A_{0},A_{1})$ if
\begin{equation*}
A=A\cap A_{0}+A\cap A_{1}.
\end{equation*}%
Following \cite{AK1}, an interpolation functor ${\mathcal{F}}$ is said to
have the \emph{decomposition property} on the couple $\vec{X}=(X_{0},X_{1})$
if ${\mathcal{F}}(\vec{X})$ has the decomposition property on $\vec{X}$.

If ${\mathcal{F}}$ is an interpolation functor, then for a~given operator $%
A\colon \vec{X}\rightarrow \vec{Y}$ we put
\begin{equation*}
V^{0}(A)=V_{{\mathcal{F}}}^{0}(A):=\{x\in {\text{ker}}_{X_{0}+X_{1}}A;%
\,x=x_{0}+x_{1},\,x_{0}\in {\mathcal{F}}(\vec{X})\cap X_{0},\,x_{1}\in
X_{1}\},
\end{equation*}%
and
\begin{equation*}
V^{1}(A)=V_{{\mathcal{F}}}^{1}(A):=\{x\in {\text{ker}}_{X_{0}+X_{1}}A;%
\,x=x_{0}+x_{1},\,x_{0}\in X_{0},\,x_{1}\in {\mathcal{F}}(\vec{X})\cap
X_{1}\}.
\end{equation*}%
Then the space $\widetilde{V}$ is any linear subspace of $\mathrm{ker}%
_{X_{0}+X_{1}}A$ such that%
\begin{equation*}
\mathrm{ker}_{X_{0}+X_{1}}A=(V^{0}(A)+V^{1}(A))\oplus \widetilde{V}.
\end{equation*}

Let $U={\text{span}}\{e_{1},e_{2},...,e_{n}\}$ $\subset
\widetilde{V}$, where $e_{1},...,e_{n}$ are linearly independent. Let us fix the decompositions
\begin{equation*}
e_{i}=e_{i}^{(0)}+e_{i}^{(1)},\quad \,1\leq i\leq n,
\end{equation*}%
where $e_{i}^{(0)}\in X_{0}$, $e_{i}^{(1)}\in X_{1}$ for each $1\leq i\leq n$
and define operators $P_{X_{0}}$ $:U\rightarrow X_{0}$ and $%
P_{X_{1}}:U\rightarrow X_{1}$ by formulas%
\begin{equation}
P_{X_{0}}(\sum_{i=1}^{n}\lambda _{i}e_{i})=\sum_{i=1}^{n}\lambda
_{i}e_{i}^{(0)},\quad \,P_{X_{1}}(u)=\sum_{i=1}^{n}\lambda _{i}e_{i}^{(1)}.
\label{PX}
\end{equation}%
If $\dim (\widetilde{V})<\infty $, then we will define operators
$P_{X_{0}}$, $P_{X_{1}}$ on the whole space $\widetilde{V}$.

Now, we are ready to state the following result, which will be very useful
in the proof of factorization theorem. We would like to remind the reader
that the interpolation functor ${\mathcal{F}}$ is called regular if $X_{0}\cap
X_{1}$ is dense in ${\mathcal{F}}(X_{0},X_{1})$ for all couples $%
(X_{0},X_{1})$.

\begin{theorem}
\label{lowerFredholm} Suppose that an operator $A\colon
(X_{0},X_{1})\rightarrow (Y_{0},Y_{1})$ be such that $Y_{0}\cap
Y_{1}=A(X_{0})\cap A(X_{1})$. Assume that a~regular interpolation functor ${%
\mathcal{F}}$ has the decomposition property on $(X_{0},X_{1})$.
Then the following two properties are equivalent{\rm:}
\begin{itemize}
\item[\rm{(i)}] $A\colon {\mathcal{F}}(X_{0},X_{1})\rightarrow
{\mathcal{F}}(Y_{0},Y_{1})$ is a lower semi-Fredholm operator.
\item[\rm{(ii)}] $A({\mathcal{F}}(X_{0},X_{1}))$ is a closed subspace of ${\mathcal{F}}%
(Y_{0},Y_{1})$ and $\mathrm{dim}\,\widetilde{V}<\infty $.
\end{itemize}
Moreover, if one of these properties holds then
\begin{equation}
\mathrm{dim}\,{\mathcal{F}}(Y_{0},Y_{1})/A({\mathcal{F}}(X_{0},X_{1}))=%
\mathrm{dim}\,\widetilde{V} \label{NM9}
\end{equation}%
and
\begin{equation}
{\mathcal{F}}(Y_{0},Y_{1})=A({\mathcal{F}}(X_{0},X_{1}))\oplus A(P_{X_{0}}(%
\widetilde{V})).  \label{NM5}
\end{equation}
\end{theorem}

\vspace{1.5 mm}

\begin{proof} (i)$\Rightarrow $(ii). As $A\colon {\mathcal{F}}%
(X_{0},X_{1})\rightarrow {\mathcal{F}}(Y_{0},Y_{1})$ is a~lower semi-Fredholm
operator then $m=\mathrm{dim}\,{\mathcal{F}}(Y_{0},Y_{1})/A({\mathcal{F}}%
(X_{0},X_{1}))<\infty $. We first show that
${\text{dim}}\,\widetilde{V}\leq m$. Suppose that
${\text{dim}}\,\widetilde{V}\geq n$. Thus there exists a~linear
independent system $\{e_{1},...,e_{n}\}$ in $\widetilde{V}$.
Consider $U={\text{span}}\{e_{1},e_{2},...,e_{n}\}$ and operators
$P_{X_{0}}\colon U\rightarrow X_{0}$ and $P_{X_{1}}\colon
U\rightarrow X_{1}$ defined in (\ref{PX}). We show that both
$AP_{X_{0}}$ and $AP_{X_{1}}$ are injective operators on $U$. To
see this observe that if $P_{X_{0}}(u)\in \ker _{X_{0}+X_{1}}A\cap
X_{0}$ then $P_{X_{1}}(u)=u-P_{X_{0}}(u)\in X_{1}\cap
\ker _{X_{0}+X_{1}}A.$ Combining this with embeddings $X_{0}\cap {\text{ker}%
}_{X_{0}+X_{1}}A\subset V^{1}(A)$ and $X_{1}\cap {\text{ker}}%
_{X_{0}+X_{1}}A\subset V^{0}(A)$ we conclude that
\begin{equation*}
u=P_{X_{0}}(u)+P_{X_{1}}(u)\in U\cap (V^{0}(A)+V^{1}(A))=\{0\}.
\end{equation*}%
In consequence $AP_{X_{0}}$ is an injective operator on $U$.
Similarly, $AP_{X_{1}}$ is also an injective operator on $U$.

Moreover, since $u\in \ker _{X_{0}+X_{1}}A$,
\begin{equation*}
A(u)=A(P_{X_{0}}(u)+P_{X_{1}}(u))=A(P_{X_{0}}(u))+A(P_{X_{1}}(u))=0,
\end{equation*}%
thus
\begin{equation*}
A(P_{X_{0}}(u))=-A(P_{X_{1}}(u))\in Y_{0}\cap Y_{1}.
\end{equation*}%
This shows that $W:=A(P_{X_{0}}(U))=A(P_{X_{1}}(U))$ is an $n$-dimensional
subspace of $Y_{0}\cap Y_{1}$. We claim that
\begin{equation*}
W\cap A({\mathcal{F}}(X_{0},X_{1}))=\{0\}.
\end{equation*}%
To see this, suppose that there exists $u\in U$ and $x\in {\mathcal{F}}%
(X_{0},X_{1})$ such that
\begin{equation*}
A(P_{X_{0}}(u))=-A(P_{X_{1}}(u))=Ax.
\end{equation*}%
Then $P_{X_{0}}(u)-x$, $P_{X_{1}}(u)+x\in \ker _{X_{0}+X_{1}}A$. Since ${%
\mathcal{F}}$ has a decomposition property on $(X_{0},X_{1})$, $%
x=x_{0}+x_{1} $ with $x_{0}\in X_{0}\cap {\mathcal{F}}(X_{0},X_{1})$ and $%
x_{1}\in X_{1}\cap {\mathcal{F}}(X_{0},X_{1})$. Therefore,
\begin{equation*}
P_{X_{0}}(u)-x=\big (P_{X_{0}}(u)-x_{0}\big )-x_{1}\in V^{1}(A),
\end{equation*}%
\begin{equation*}
P_{X_{1}}(u)+x=x_{0}+\big (P_{X_{1}}(u)+x_{1}\big )\in V^{0}(A),
\end{equation*}%
and consequently
\begin{equation*}
u=P_{X_{0}}(u)+P_{X_{1}}(u)\in V^{0}(A)+V^{1}(A).
\end{equation*}%
Thus $u=0$ because $u\in \widetilde{V}\cap \big (V^{0}(A)+V^{1}(A)\big )%
=\{0\}$ and so the claim is proved. So $n\leq m$ and therefore $\dim (%
\widetilde{V})\leq m$. So to prove that $\dim (\widetilde{V})=m$ it is
enough to show
\begin{equation*}
{\mathcal{F}}(Y_{0},Y_{1})=A({\mathcal{F}}(X_{0},X_{1}))\oplus A(P_{X_{0}}(%
\widetilde{V})).
\end{equation*}%
Let us take $U=\widetilde{V}.$ Then the following inclusion holds:
\begin{equation*}
Y_{0}\cap Y_{1}\subset A({\mathcal{F}}(X_{0},X_{1}))\oplus W,
\end{equation*}%
where $W=AP_{X_{0}}(\widetilde{V})$. To prove this fix $y\in
Y_{0}\cap Y_{1}$. Then our hypothesis $Y_{0}\cap
Y_{1}=A(X_{0})\cap A(X_{1})$ implies that there exist $x_{0}\in
X_{0}$ and $x_{1}\in X_{1}$ such that
\begin{equation*}
Ax_{0}=Ax_{1}=y.
\end{equation*}%
Since $x_{0}-x_{1}\in \ker _{X_{0}+X_{1}}A$, there exist elements
$v_{0}\in V^{0}(A)$, $v_{1}\in V^{1}(A)$ and $\widetilde{v}\in
\widetilde{V}$ such that
\begin{equation*}
x_{0}-x_{1}=v_{0}+v_{1}+\tilde{v}.
\end{equation*}%
For $j=0,1$ let $v_{j}=v_{j}^{(0)}+v_{j}^{(1)}$, where $v_{j}^{(0)}\in X_{0}$
and $v_{j}^{(1)}\in X_{1}$. Then we have
\begin{equation*}
x_{0}-x_{1}=v_{0}^{(0)}+v_{0}^{(1)}+v_{1}^{(0)}+v_{1}^{(1)}+P_{X_{0}}(\tilde{%
v})+P_{X_{1}}(\tilde{v}).
\end{equation*}%
This implies that
\begin{equation*}
u_{0}:=x_{0}-v_{0}^{(0)}-v_{1}^{(0)}-P_{X_{0}}(\tilde{v}%
)=x_{1}+v_{0}^{(1)}+v_{1}^{(1)}+P_{X_{1}}(\tilde{v})\in X_{0}\cap X_{1}
\end{equation*}%
and thus
\begin{equation*}
x_{0}=u_{0}+v_{0}^{(0)}+v_{1}^{(0)}+P_{X_{0}}(\tilde{v}).
\end{equation*}%
Combining this we obtain
\begin{equation*}
y=A(x_{0})=A(u_{0})+A(v_{0}^{(0)})+A(v_{1}^{(0)})+A(P_{X_{0}}(\tilde{v})).
\end{equation*}%
Note that from the definition of the spaces $V^{0}(A)+V^{1}(A)$, it follows
that
\begin{equation*}
v_{j}^{(j)}\in {\mathcal{F}}(X_{0},X_{1}),\quad \,j=0,1.
\end{equation*}%
Hence
\begin{equation*}
A(v_{0}^{(0)})\in A({\mathcal{F}}(X_{0},X_{1}))\quad {\text{and $%
A(v_{1}^{(0)})=-A(v_{1}^{(1)})\in A({\mathcal{F}}(X_{0},X_{1}))$}}.
\end{equation*}%
Since $u_{0}\in X_{0}\cap X_{1}\in {\mathcal{F}}(X_{0},X_{1}))$ therefore $%
A(u_{0})+A(v_{0}^{(0)})+A(v_{1}^{0})\in A({\mathcal{F}}(X_{0},X_{1}))$ and
we have
\begin{equation*}
y\in A({\mathcal{F}}(X_{0},X_{1}))\oplus W.
\end{equation*}%
Our hypothesis that $A$ is a~lower semi-Fredholm operator implies that $A({%
\mathcal{F}}(X_{0},X_{1}))$ is a~closed subspace of ${\mathcal{F}}%
(Y_{0},Y_{1})$, and so $A((X_{0},X_{1}))\oplus W$ is a~closed subspace of ${%
\mathcal{F}}(Y_{0},Y_{1}))$ (by the fact that $W$ is finite dimensional
subspace of ${\mathcal{F}}(Y_{0},Y_{1})$). Combining this with the inclusion
$Y_{0}\cap Y_{1}\subset A({\mathcal{F}}(X_{0},X_{1}))\oplus W$ and
regularity of ${\mathcal{F}}$, we obtain the following topological direct
sum,
\begin{equation}
{\mathcal{F}}(Y_{0},Y_{1})=A({\mathcal{F}}(X_{0},X_{1}))\oplus A(P_{X_{0}}(%
\widetilde{V})).  \label{K4}
\end{equation}%
So $\dim \widetilde{V}=m$ and formula (\ref{NM5}) is also proved.

(ii)$\Rightarrow $(i) Since the space $\widetilde{V}$ is finite
dimensional the above proof shows that
$A({\mathcal{F}}(X_{0},X_{1}))$ is a~closed subspace of
${\mathcal{F}}(Y_{0},Y_{1})$ and moreover
$A(P_{X_{0}}(\widetilde{V}))$ is a~
subspace of $Y_{0}\cap Y_{1}$ with dimension equal to the dimension of $%
\widetilde{V}$, so equality (\ref{K4}) implies (i).
\end{proof}

\begin{remark}
In the case when the interpolation functor ${\mathcal{F}}$ is not regular, the
above proof gives that $\dim \widetilde{V}\leq \mathrm{dim}\,{\mathcal{F}}%
(Y_{0},Y_{1})/A({\mathcal{F}}(X_{0},X_{1}))$.
\end{remark}

\begin{remark}
\label{REM} In the proof it was shown that the operator
$AP_{X_{0}}$ is injective and maps $\widetilde{V}$ into $Y_{0}\cap
Y_{1}$.
\end{remark}

Theorem \ref{lowerFredholm} looks similar to the following result
due to Albrecht and Schindler \cite{ASch}, which extends Herrero's
result mentioned in the introduction and shows that under some
restrictive hypotheses we have stability of Fredholm properties
under \textit{any} interpolation functor.

\begin{theorem}
Let $\vec{X}=(X_{0},X_{1})$, $\vec{Y}=(Y_{0},Y_{1})$ be Banach couples such
that $X_{1}\subset X_{0}$, $Y_{1}\subset Y_{0}$, and $Y_{1}$ is dense in $%
Y_{0}$. Suppose that $T\colon \vec{X}\rightarrow \vec{Y}$ is a~Fredholm
operator with $\mathrm{ind}(T\colon X_{0}\rightarrow Y_{0})=\mathrm{ind}%
(T\colon X_{1}\rightarrow Y_{1})$, and let ${\mathcal{F}}$ be an arbitrary
interpolation functor.~Then the operator $T\colon {\mathcal{F}}(\vec{X}%
)\rightarrow {\mathcal{F}}(\vec{Y})$ is Fredholm with $\mathrm{ind}(T)=%
\mathrm{ind}(T\colon X_{0}\rightarrow Y_{0})$. Moreover $\ker _{{\mathcal{F}}%
(\vec{X})}T=\ker _{X_{0}}T=\ker _{X_{1}}T$, and there exists
a~finite dimensional subspace $H$ of $Y_{1}$ $($independent of
${\mathcal{F}})$ such that $Y_{j}=H\oplus T(X_{j})$ for $j=0,1$ and ${\mathcal{F}}(\vec{Y}%
)=H\oplus T({\mathcal{F}}(\vec{Y}))$.
\end{theorem}


We also have a~variant of the Theorem \ref{lowerFredholm} in the case of
Fredholm operators. For the sake of completeness we formulate it separately.


\begin{theorem}
\label{fredholm} Suppose that an operator $A\colon (X_{0},X_{1})\rightarrow
(Y_{0},Y_{1})$ be such that $Y_{0}\cap Y_{1}=A(X_{0})\cap A(X_{1})$. Assume
that a~regular interpolation functor ${\mathcal{F}}$ has the decomposition
property on $\vec{X}$ and $A\colon {\mathcal{F}}(X_{0},X_{1})\rightarrow {%
\mathcal{F}}(Y_{0},Y_{1})$ is a~Fredholm operator with $n(A)=\mathrm{dim}(%
\mathrm{ker}_{X_{0}+X_{1}}A\cap {\mathcal{F}}(X_{0},X_{1}))$ and $d(A)=%
\mathrm{dim}\,\big ({\mathcal{F}}(Y_{0},Y_{1})/A({\mathcal{F}}(X_{0},X_{1}))%
\big).$ Then
\begin{itemize}
\item[{\rm(i)}] The space $V^{0}(A)\cap V^{1}(A)$ is finite
dimensional with $\func{dim}(V^{0}(A)\cap V^{1}(A))=n(A).$
\item[{\rm(ii)}] The space $\widetilde{V}$ is finite dimensional
with $\mathrm{dim}\,\widetilde{V}=d(A)$ and
\begin{equation*}
{\mathcal{F}}(Y_{0},Y_{1})=A({\mathcal{F}}(X_{0},X_{1}))\oplus A(P_{X_{0}}(%
\widetilde{V})).
\end{equation*}
\end{itemize}
\end{theorem}

\textit{\ }This theorem is a direct consequence of Theorem \ref%
{lowerFredholm} and the fact that our hypothesis that
$\mathcal{F}$ has the decomposition property on $\vec{X}$ easily
gives
\begin{equation}
V^{0}(A)\cap V^{1}(A)={\mathcal{F}}((X_{0},X_{1}))\cap {\text{ker}}%
_{X_{0}+X_{1}}A.  \label{NM6}
\end{equation}

Since $(\cdot )_{\theta ,q}$ is a~regular interpolation functor for all $%
\theta \in (0,1)$ and $q\in \lbrack 1,\infty )$ the following result is an
immediate consequence of Theorem \ref{fredholm}. In this theorem we denote by $%
\widetilde{V}$ a~linear subspace of
$\mathrm{ker}_{X_{0}+X_{1}}\,A$ such that
\begin{equation}
\mathrm{ker}_{X_{0}+X_{1}}A=(V_{\theta ,q}^{0}+V_{\theta ,q}^{1})\oplus
\widetilde{V}.  \label{NM1}
\end{equation}

\begin{theorem}
\label{realfredholm} Suppose that an operator $A\colon
(X_{0},X_{1})\rightarrow (Y_{0},Y_{1})$ be surjective and $q\in \lbrack
1,\infty )$. Then $A\colon (X_{0},X_{1})_{\theta ,q}\rightarrow
(Y_{0},Y_{1})_{\theta ,q}$ is a~Fredholm operator with $n(A)=\mathrm{dim(ker}%
_{X_{0}+X_{1}}\,A\cap (X_{0},X_{1})_{\theta ,q})$ and $d(A)=\mathrm{dim}\,%
\big ((Y_{0},Y_{1})_{\theta ,q}/A((X_{0},X_{1})_{\theta ,q})\big
)$ if and only if the following conditions are satisfied{\rm:}

\begin{itemize}
\item[{\rm(i)}] The space $V_{\theta ,q}^{0}\cap V_{\theta
,q}^{1}$~is finite dimensional with $\func{dim}(V_{\theta
,q}^0\cap V_{\theta ,q}^{1})=n(A)$.

\item[{\rm(ii)}] The space $\widetilde{V}$ is finite dimensional
with $\mathrm{dim}\,\widetilde{V}=d(A)$ and
\begin{equation*}
(Y_{0},Y_{1})_{\theta ,q}=A((X_{0},X_{1})_{\theta ,q})\oplus A(P_{X_{0}}(%
\widetilde{V})).
\end{equation*}
\end{itemize}
\end{theorem}

We finish this subsection with a~general construction of
interpolation functor which has decomposition property on any
Banach couple. Let $A$ be an intermediate Banach space with
respect to the couple $\vec{A}=(A_{0},A_{1})$. Following Aronszajn
and Gagliardo \cite{AG}, for any Banach couple $\vec{X} $, we
define the so called \emph{minimal} interpolation functor
$G_{\vec{A}}^{A}$ by
\begin{equation*}
G_{\vec{A}}^{A}(\vec{X}):=\Big \{\sum_{n=1}^{\infty
}T_{n}a_{n};\,\,\,\sum_{n=1}^{\infty }\Vert T_{n}\Vert _{\vec{A}\rightarrow
\vec{X}}\Vert a_{n}\Vert _{A}<\infty \Big \}.
\end{equation*}%
The norm is given by
\begin{equation*}
\Vert x\Vert =\func{inf}\Big \{\sum_{n=1}^{\infty }\Vert T_{n}\Vert _{\vec{A}%
\rightarrow \vec{X}}\Vert a_{n}\Vert _{A};\,\,\,x=\sum_{n=1}^{\infty
}T_{n}a_{n}\Big \}.
\end{equation*}%
We note that $G_{\vec{A}}^{A}$ is the minimal interpolation functor
satisfying $A\hookrightarrow G_{\vec{A}}^{A}(\vec{A})$.

We omit the technical proof of the following result.

\begin{proposition}
If a~Banach space $A$ has the decomposition property with respect to the
couple $\vec{A}$, then the minimal interpolation functor $G_{\vec{A}}^{A}$
has the decomposition property on every Banach couple $\vec{X}$.
\end{proposition}

\subsection{Functors $\left( \cdot \right) _{\protect\theta ,\infty }$ and $%
\left( \cdot \right) _{\protect\theta, c_{0}}$}

There are difficulties connected with the functor $\left( \cdot
\right) _{\theta ,q}$ for $q=\infty $. The reason of these
difficulties is that this functor is not regular. Some of these
difficulties can be overcome by using the functor $\left( \cdot
\right) _{\theta, c_{0}}$, which is regular and the norm on $(X_0,
X_1)_{\theta, c_0}$ for any couple $\left( X_{0},X_{1}\right)$ is
the restriction of the norm of $\left( X_{0},X_{1}\right) _{\theta
,\infty }$ on the closure of $X_{0}\cap X_{1}$ in $\left(
X_{0},X_{1}\right) _{\theta,\infty}$.

It should be noted that the functor $(\cdot)_{\theta, c_0}$ has
not been studied enough in Interpolation Theory. Below we present some
results which will be used in the proof of Theorem 15.

We shall be concerned with orbital methods generated by a~single element.
Let $\vec{A}=(A_{0},A_{1})$ be a~given Banach couple. For any Banach couple $%
\vec{X}=(X_{0},X_{1})$, let ${\mathcal{I}}(\vec{A},\vec{X})\subset L(\vec{A}%
,\vec{X})$ be a~Banach space of operators with the left ideal property which
holds for all Banach couples $\vec{X}$ and $\vec{Y}$, i.e., $T\in {\mathcal{I%
}}(\vec{A},\vec{X})$ and $S\in L(\vec{X},\vec{Y})$ implies $ST\in {\mathcal{I%
}}(\vec{A},\vec{Y})$ with
\begin{equation*}
\Vert ST\Vert _{{\mathcal{I}}(\vec{A},\vec{Y})}\leq \Vert S\Vert _{L(\vec{X},%
\vec{Y})}\,\Vert T\Vert _{{\mathcal{I}}(\vec{A},\vec{X})}.
\end{equation*}%
For any $0\neq a\in A_{0}+A_{1}$, we define $Orb_{a,\,{\mathcal{I}}}^{\vec{A}%
}(\vec{X})$ the space (called the orbit of $a$ generated by ${\mathcal{I}}$
) of all elements $x\in X_{0}+X_{1}$ of the form $x=Ta$ for some $T\in {%
\mathcal{I}}(\vec{A},\vec{X})$. This construction defines an exact
interpolation functor provided $Orb_{a,\,{\mathcal{I}}}^{\vec{A}}(\vec{X})$
is equipped with the norm
\begin{equation*}
\Vert x\Vert =\func{inf}\{\Vert T\Vert _{{\mathcal{I}}(\vec{A},\vec{X}%
)};\,x=Ta,\,T\in {\mathcal{I}}(\vec{A},\vec{X})\}.
\end{equation*}%
In the case when ${\mathcal{I}}(\vec{A},\vec{X})=L(\vec{A},\vec{X})$, we
obtain the classical interpolation orbit introduced in \cite{AG}. In this
case we write $Orb_{a}^{\vec{A}}(\cdot )$ instead of $Orb_{a,\,L}^{\vec{A}%
}(\cdot )$.

Let $\vec{X}=\left( X_{0},X_{1}\right) $ be a Banach couple. Then the space $%
\left( X_{0},X_{1}\right) _{\theta ,\infty }$ is defined as the set of all $%
x\in X_{0}+X_{1}$ such that%
\begin{equation*}
\left\Vert x\right\Vert _{\theta ,\infty }=\sup_{t>0}t^{-\theta}
K(t,x; \vec{X})<\infty \text{.}
\end{equation*}%
Note that the space $(X_0, X_1)_{\theta, \infty}$ can be obtained in another way.
To see this let us consider the couple $\vec{L}%
^{1}=\left( L^{1}(\frac{dt}{t}),L^{1}(t^{-1},\frac{dt}{t}\right)
)$, which
consists of functions%
\begin{equation*}
\left\Vert f\right\Vert _{L^{1}}:=\int_{0}^{\infty }\left\vert
f(t)\right\vert \frac{dt}{t}<\infty, \text{ }\left\Vert
f\right\Vert
_{L^{1}(t^{-1})}=\int_{0}^{\infty }\left\vert f(t)\right\vert \frac{dt}{t^{2}%
}<\infty \text{.}
\end{equation*}
It is easy to check that for every $\theta \in (0, 1)$ the
elements $a_{\theta}$ given by
\begin{equation}
a_{\theta} (t) =\theta (1-\theta )t^{\theta}, \quad\, t>0
\label{AP1}
\end{equation}%
satisfy%
\begin{equation*}
K(t,a_{\theta };\vec{L}^{1})=t^{\theta }, \quad\, t>0.
\end{equation*}%

In what follows we will need the so-called orbital interpolation
space $Orb_{a_{\theta }}(\vec{X})$ equipped with the norm
\begin{equation*}
\left\Vert x\right\Vert _{Orb_{a_{\theta
}}(\vec{X})}=\inf_{Ta_{\theta }=x}\left\Vert T\right\Vert
_{L\mathbb{(}\vec{L}^{1},\vec{X})}\text{.}
\end{equation*}%

Since $K(t,a_{\theta };\vec{L}^{1})=t^{\theta }$, it follows from
equivalence theorem (see \cite{BK}) that the Banach space
$Orb_{a_{\theta }}(\vec{X})$ coincides with $\left(
X_{0},X_{1}\right)_{\theta ,\infty }$ and
\begin{equation*}
\left\Vert x\right\Vert _{\theta ,\infty }\approx \left\Vert
x\right\Vert _{Orb_{a_{\theta }}(\vec{X})}
\end{equation*}%
with constant of equivalence not more than 18.

We need a similar result for the space $\left( X_{0},X_{1}\right) _{\theta
,c_{0}}$ which consists of element $x\in X_{0}+X_{1}$ such that
\begin{equation*}
\lim_{t\rightarrow 0}\frac{K(t,x;X_{0},X_{1})}{t^{\theta }}%
=\lim_{t\rightarrow \infty }\frac{K(t,x;X_{0},X_{1})}{t^{\theta }}=0
\end{equation*}%
and%
\begin{equation*}
\left\Vert x\right\Vert _{\theta ,c_{0}}=\sup_{t>0}\frac{K(t,x;X_{0},X_{1})}{%
t^{\theta }}<\infty \text{.}
\end{equation*}%
It is known that the space $\left( X_{0},X_{1}\right) _{\theta
,c_{0}}$ coincides with the closure of the intersection $X_{0}\cap
X_{1}$ in the space $\left( X_{0},X_{1}\right) _{\theta ,\infty
}$.

We need an orbital description of the space $\left(
X_{0},X_{1}\right) _{\theta
,c_{0}}$. Let consider the set of operators $T$ $\in L\mathbb{(}\vec{L}^{1},%
\vec{X})$ which has property%
\begin{equation*}
T(L^{1}+L^{1}(t^{-1}))\subset X_{0}\cap X_{1}
\end{equation*}%
and its closure in $L\mathbb{(}\vec{L}^{1},\vec{X})$, which we define by $L%
\mathbb{(}\vec{L}^{1},\vec{X})^{0}$. Clearly this set has a left ideal
property, i.e. for any operator $A:\vec{X}\rightarrow \vec{Y}$ and any
operator $T\in L\mathbb{(}\vec{L}^{1},\vec{X})^{0}$ we have $AT\in L\mathbb{(%
}\vec{L}^{1},\vec{Y})^{0}$.

Let us consider the interpolation space $Orb_{a_{\theta
}}^{0}(\vec{X})$
defined by the norm%
\begin{equation*}
\left\Vert x\right\Vert _{Orb_{a_{\theta
}}^{0}(\vec{X})}=\inf_{Ta_{\theta }=x}\left\Vert T\right\Vert
_{L\mathbb{(}\vec{L}^{1},\vec{X})^{0}}\text{.}
\end{equation*}

The next theorem provides an orbital description of the space
$\left( X_{0},X_{1}\right) _{\theta ,c_{0}}$.

\begin{theorem}
For every $\theta \in (0, 1)$ the following formula holds{\rm:}
\begin{equation*}
\left( X_{0},X_{1}\right) _{\theta ,c_{0}} = Orb_{a_{\theta }}^{0}(\vec{X})%
\text{,}
\end{equation*}
with the norms being equivalent with constants independent of
$\theta$.
\end{theorem}

\begin{proof} Fix $x\in (X_0, X_1)_{\theta ,c_{0}}$; then
from the equivalence theorem it follows that there exists a~decomposition%
\begin{equation*}
x=\sum_{i\in \mathbb{Z}}u_{i},
\end{equation*}%
where the series converges absolutely in $X_{0}+X_{1}$ and
\begin{equation}
\left\Vert x\right\Vert _{\theta ,c_{0}}\approx \sup_{i\in \mathbb{Z}}\frac{%
J(2^{i},u_{i};X_{0},X_{1})}{2^{\theta i}}, \text{ }  \label{AP3}
\end{equation}%
\begin{equation}
\lim_{i\rightarrow -\infty }\frac{J(2^{i},u_{i};X_{0},X_{1})}{2^{\theta i}}%
=\lim_{i\rightarrow \infty }\frac{J(2^{i},u_{i};X_{0},X_{1})}{2^{\theta i}}=0%
\text{.}  \label{AP2}
\end{equation}%
For each positive integer $N$ we define an operator $T_{N}\in
L\mathbb{(}\vec{L}^{1},\vec{X})^{0}$ by
\begin{equation*}
T_{N}(f)=\sum_{k=-N}^{N-1}\frac{u_{k}}{\int_{2^{k}}^{2^{k+1}}a_{\theta }(t)%
\frac{dt}{t}}\int_{2^{k}}^{2^{k+1}}f(t)\frac{dt}{t}, \quad\, f\in
L^1 + L^1(t^{-1}).
\end{equation*}%
We claim that
\begin{equation*}
\left\Vert T_{N}\right\Vert
_{L\mathbb{(}\vec{L}^{1},\vec{X})^{0}}\leq \gamma \left\Vert
x\right\Vert _{\theta ,c_{0}}\text{.}
\end{equation*}%
To see this observe that from (\ref{AP3}) it follows that%
\begin{equation*}
\sup_{i\in \mathbb{Z}}\frac{\left\Vert u_{i}\right\Vert
_{X_{0}}}{2^{\theta i}}\leq \gamma \left\Vert x\right\Vert
_{\theta ,c_{0}}, \quad\,\,\,
\sup_{i\in \mathbb{Z}}\frac{2^{i}\left\Vert u_{i}\right\Vert _{X_{1}}}{%
2^{\theta i}}\leq \gamma \left\Vert x\right\Vert _{\theta ,c_{0}}
\end{equation*}%
and so for every $f\in L^1$,%
\begin{equation*}
\left\Vert T_{N}f\right\Vert _{X_{0}}\leq \gamma \sum_{k=-N}^{N-1}\frac{%
\left\Vert u_{k}\right\Vert _{X_{0}}}{2^{\theta k}}\int_{2^{k}}^{2^{k+1}}%
\left\vert f(t)\right\vert \frac{dt}{t}\leq \gamma \left\Vert
x\right\Vert _{\theta ,c_{0}}\left\Vert f\right\Vert _{L^{1}}
\end{equation*}%
and similarly for every $f\in L^1(t^{-1})$,%
\begin{equation*}
\left\Vert T_{N}f\right\Vert _{X_{1}}\leq \gamma \sum_{k=-N}^{N-1}\frac{%
2^{k}\left\Vert u_{k}\right\Vert _{X_{1}}}{2^{\theta k}}%
\int_{2^{k}}^{2^{k+1}}\left\vert f(t)\right\vert \frac{dt}{t^{2}}\leq \gamma
\left\Vert x\right\Vert _{\theta ,c_{0}}\left\Vert f\right\Vert
_{L^{1}(t^{-1})\text{.}}
\end{equation*}%
It is also clear that $T_{N}$ maps $L^{1}+L^{1}(t^{-1})$ into
$X_{0}\cap X_{1}$. Moreover, from (\ref{AP2}) it follows that as
$N\rightarrow \infty $, the operators $T_{N}$ converge to the operator $T\in L\mathbb{(}\vec{L}^{1},\vec{%
X})^{0}$ defined by the formula
\begin{equation*}
T(f)=\sum_{k\in \mathbb{Z}}\frac{u_{k}}{\int_{2^{k}}^{2^{k+1}}a_{\theta }(t)%
\frac{dt}{t}}\int_{2^{k}}^{2^{k+1}}f(t)\frac{dt}{t}\text{.}
\end{equation*}%

Since $Ta_{\theta }=\sum_{k\in \mathbb{Z}}u_{k}=x$, we obtain%
\begin{equation*}
\left\Vert x\right\Vert _{Orb_{a_{\theta
}}^{0}(\vec{X})}=\inf_{Ta_{\theta }=x}\left\Vert T\right\Vert
_{L\mathbb{(}\vec{L}^{1},\vec{X})^{0}}\leq \gamma \left\Vert
x\right\Vert _{\theta ,c_{0}}\text{.}
\end{equation*}%

To prove the opposite inequality we just need to notice that%
\begin{equation*}
\left\Vert x\right\Vert _{\theta ,c_{0}}=\left\Vert x\right\Vert _{\theta
,\infty }\leq \gamma \inf_{Ta_{\theta }=x}\left\Vert T\right\Vert _{L\mathbb{%
(}\vec{L}^{1},\vec{X})}\leq \inf_{Ta_{\theta }=x}\left\Vert T\right\Vert _{L%
\mathbb{(}\vec{L}^{1},\vec{X})^{0}}=\left\Vert x\right\Vert
_{Orb_{a_{\theta }}^{0}(\vec{X})}.
\end{equation*}
\end{proof}

Let $0<\theta _{0}<\theta _{1}<1$ and $a_{\theta _{0}}$,
$a_{\theta _{1}}$ be corresponding functions of the form
(\ref{AP1}). In \cite{KM} a~distance $\rho$ was defined
between functors $Orb_{a_{\theta _{0}}}$ and $%
Orb_{a_{\theta _{1}}}$ by the formula
\begin{equation*}
\rho (Orb_{a_{\theta _{0}}},
Orb_{a_{\theta_{1}}})=\sup_{\vec{X}}\Big(\sup_{\left\Vert
T\right\Vert _{L\mathbb{(}\vec{L}^{1},\vec{X})}\leq
1} \Big|\left\Vert Ta_{\theta _{0}}\right\Vert _{Orb_{a_{\theta _{0}}}(%
\vec{X})}-\left\Vert Ta_{\theta _{1}}\right\Vert_{Orb_{a_{\theta _{1}}}(%
\vec{X})}\Big|\Big)
\end{equation*}
and it was shown that
\begin{equation}
\rho (Orb_{a_{\theta _{0}}}, Orb_{a_{\theta _{1}}})\leq \gamma\,
\frac{\theta _{1}-\theta _{0}}{\min (\theta _{0},1-\theta _{1})}
\label{AP4}
\end{equation}%
with constant $\gamma >0$ independent of $\theta _{0},\theta
_{1}$. Similarly we define distance between functors
$Orb_{a_{\theta _{0}}}^{0}$, $Orb_{a_{\theta _{1}}}^{0}$ (we need
only instead of $\left\Vert T\right\Vert
_{L\mathbb{(}\vec{L}^{1},\vec{X})}\leq 1$ use $\left\Vert
T\right\Vert _{L\mathbb{(}\vec{L}^{1},\vec{X})^{0}}\leq 1$), i.e.,
\begin{equation*}
\rho (Orb_{a_{\theta_{0}}}^{0}, Orb_{a_{\theta _{1}}}^{0})=\sup_{\vec{X}%
}\Big(\sup_{\left\Vert T\right\Vert
_{L\mathbb{(}\vec{L}^{1},\vec{X})^{0}}\leq
1}\Big |\left\Vert Ta_{\theta _{0}}\right\Vert_{Orb_{a_{\theta_{0}}}(%
\vec{X})}-\left\Vert Ta_{\theta _{1}}\right\Vert_{Orb_{a_{\theta _{1}}}(%
\vec{X})}\Big|\Big).
\end{equation*}%
Since the proof of estimate (\ref{AP4}) in \cite{KM} also works
for the functors $Orb_{a_{\theta_{0}}}^{0}$, $Orb_{a_{\theta
_{1}}}^{0}$, we conclude that%
\begin{equation*}
\rho (Orb_{a_{\theta _{0}}}^{0}, Orb_{a_{\theta _{1}}}^{0})\leq \gamma \frac{%
\theta _{1}-\theta _{0}}{\min (\theta _{0},1-\theta _{1})}
\end{equation*}%
with constant $\gamma >0$ independent of $\theta _{0},\theta _{1}$. Lets $U=L%
\mathbb{(}\vec{L}^{1},\vec{X})^{0}$ and $U_{\theta }$ be a closed
subspace of $U$ which consists of operators $T\in
L\mathbb{(}\vec{L}^{1},\vec{X})^{0}$ such that $Ta_{\theta }=0$.
Then the space $Orb_{a_{\theta }}^{0}(\vec{X})$
with the norm%
\begin{equation*}
\left\Vert x\right\Vert _{orb_{a_{\theta
}}^{0}(\vec{X})}=\inf_{Ta_{\theta }=x}\left\Vert T\right\Vert
_{L\mathbb{(}\vec{L}^{1},\vec{X})^{0}}
\end{equation*}
can be interpreted as a~norm in quotient space $U/U_{\theta }$ and
applying general result from \cite{KM} (see Section 6
"invertibility theorems for quotient spaces") we obtained

\begin{theorem}
\label{KM}Let $A\colon \vec{X}\rightarrow \vec{Y}$ be a bounded
linear operator. Suppose also that the operator
\begin{equation}
A\colon Orb_{a_{\theta _{0}}}^{0}(\vec{X})\rightarrow Orb_{a_{\theta _{0}}}^{0}(%
\vec{Y})  \label{AP6}
\end{equation}%
is invertible. Then there exists $\varepsilon >0$ such that from
$\rho (Orb_{a_{\theta _{0}}}^{0}, Orb_{a_{\theta _{1}}}^{0})\leq
\varepsilon $ it follows that the operator
\begin{equation}
A\colon Orb_{a_{\theta _{1}}}^{0}(\vec{X})\rightarrow Orb_{a_{\theta _{1}}}^{0}(%
\vec{Y})  \label{AP5}
\end{equation}%
is also invertible.
\end{theorem}

\begin{remark}
\label{KMR} From the results of the Section $6$ in $\cite{KM}$ it
also follows that $\varepsilon >0$ can be estimated from
\textbf{below} by the norms of $\left\Vert A\right\Vert
_{X_{i}\rightarrow Y_{i}}$, $i=0,1$ and the norm of the inverse to
the operator $(\ref{AP6})$. Moreover, for $\varepsilon >0$ small
enough, the norm of the inverse to the operator $(\ref{AP5})$ can
be estimated from \textbf{above} by the norms of $\left\Vert
A\right\Vert _{X_{i}\rightarrow Y_{i}}$, $i=0,1$ and the norm of
the inverse to the operator $(\ref{AP6})$.
\end{remark}

\section{Factorization Theorem}

This section is devoted to an important theorem on factorization of
Fredholm operators between real interpolation spaces discussed in
the introduction. In the formulation of this theorem a~special
type of quotient operators will be used.

\subsection{Special type of quotient operators}

Let $\vec{U}=(U_{0},U_{1})$ be a~Banach couple and $U$ be a~closed
subspace of the sum $U_{0}+U_{1}$. Clearly $U\cap U_{i}$ is a~closed subspace of $%
U_{i}$ for $i=0,1$. We denote by $Q\colon U_{0}+U_{1}\rightarrow
(U_{0}+U_{1})/U$ the quotient map generated by a~subspace $U$.~Since $%
Q(U_{0})\hookrightarrow (U_{0}+U_{1})/U$ and $Q(U_{1})\hookrightarrow
(U_{0}+U_{1})/U$, $(Q(U_{0}),Q(U_{1}))$ forms a~Banach couple.~We note that
the norm $Q(U_{0})+Q(U_{1})$ is equal to the norm of the quotient space $%
(U_{0}+U_{1})/U$ and the following can be easily verified:
\begin{align*}
\Vert w\Vert _{(U_{0}+U_{1})/U}& =\func{inf}\big \{\Vert u\Vert
_{U_{0}+U_{1}};\,Qu=w\}=\func{inf}\{\Vert u_{0}\Vert _{U_{0}}+\Vert
u_{1}\Vert _{U_{1}};\,w=Qu_{0}+Qu_{1}\big \} \\
& =\func{inf}_{w=w_{0}+w_{1}}\big (\func{inf}\big \{\Vert u_{0}\Vert
_{U_{0}}+\Vert u_{1}\Vert _{U_{1}};\,Qu_{0}=w_{0},\,Qu_{0}=w_{0}\big \}\big )
\\
& =\func{inf}\big \{\Vert w_{0}\Vert _{Q(U_{0})}+\Vert w_{1}\Vert
_{Q(U_{1})};\,w=w_{0}+w_{1}\big \}=\Vert w\Vert _{Q(U_{0})+Q(U_{1})}.
\end{align*}

We need to fix a~natural notation; an operator $B\colon
(X_{0},X_{1})\rightarrow (Y_{0},Y_{1})$ is said to be an
isomorphism between Banach couples $(X_{0},X_{1})$ and
$(Y_{0},Y_{1})$ if $B$ is invertible on endpoint spaces and in
addition $B\colon X_{0}+X_{1}\rightarrow Y_{0}+Y_{1}$ is
invertible.

It is clear that if $B\colon (X_{0},X_{1})\rightarrow
(Y_{0},Y_{1})$ is an isomorphism, then $B\colon
{\mathcal{F}}(X_{0},X_{1})\rightarrow {\mathcal{F}}(Y_{0},Y_{1})$
is invertible for any interpolation functor ${\mathcal{F}}$.


\begin{proposition}
\label{PrN2} Let $A\colon (X_{0},X_{1})\rightarrow (Y_{0},Y_{1})$
be a~surjective operator between Banach couples and let $U=\ker
_{X_{0}+X_{1}}A$. Then $A$ factors as follows{\rm:}
\begin{equation*}
A\colon (X_{0},X_{1})\overset{Q}{\longrightarrow }(Q(X_{0}),Q(X_{1}))\overset%
{B}{\longrightarrow }(Y_{0},Y_{1}),
\end{equation*}%
where $Q\colon X_{0}+X_{1}\rightarrow (X_{0}+X_{1})/U$ is a~quotient
operator and $B$ is an isomorphism defined by $B\widehat{x}=Ax$, where $%
\widehat{x}=Qx\in Q(X_{0})+Q(X_{1})$.
\end{proposition}

\begin{proof} Clearly that the map $B$ is well defined and bounded from $%
Q(X_{i})$ to $Y_{i}$ for $i=0,1$. Our hypothesis on $A$ implies that $%
B\colon (Q(X_{0}),Q(X_{1}))\rightarrow (Y_{0},Y_{1})$ is surjective. Since $%
B\colon Q(X_{0})+Q(X_{1})\rightarrow Y_{0}+Y_{1}$ is injective, it follows
from Banach Theorem on inverse operator that the operator $B$ is an isomorphism.
To conclude it is enough to observe that $A=BQ$.
\end{proof}

We will need the following property of a~quotient couple.

\begin{proposition}
\label{KF} If $(X_0,X_1)$ is a~Banach couple and $Q\colon X_{0}+X_{1}\to
(X_{0}+X_{1})/U$ is the quotient map. Then for all $x\in X_0+X_1$ we have
\begin{equation*}
K(t,Qx;Q(X_{0}),Q(X_{1}))= \func{inf}_{v\in U}K(t,x+v;X_{0},X_{1}),\quad
\,t>0.
\end{equation*}
\end{proposition}

\begin{proof} Since $\left\Vert Q\right\Vert _{X_{i}\rightarrow
Q(X_{i})}=1 $, for $i=0,1$, thus for all $v\in U$ and $t>0$ we have
\begin{equation*}
K(t,Qx;Q(X_{0}),Q(X_{1}))=K(t,Q(x+v);Q(X_{0}),Q(X_{1}))\leq
K(t,x+v;X_{0},X_{1}).
\end{equation*}%
Hence $K(t,Qx;Q(X_{0}),Q(X_{1}))\leq \func{inf}_{v\in
U}K(t,x+v;X_{0},X_{1})$. To prove the opposite inequality let us
choose elements $z_{0}\in Q(X_{0})$ and $z_{1}\in Q(X_{0})$ such
that $z_{0}+z_{1}=Qx$ and
\begin{equation*}
\left\Vert z_{0}\right\Vert _{Q(X_{0})}+t\left\Vert z_{0}\right\Vert
_{Q(X_{1})}\leq (1+\varepsilon )K(t,Qx;Q(X_{0}),Q(X_{1})).
\end{equation*}%
From the definition of the norms in $Q(X_{0})$ and $Q(X_{1})$, it follows
that there exist elements $x_{0}\in X_{0}$, $x_{1}\in X_{1}$ such that $%
Qx_{0}=z_{0}$, $Qx_{1}=z_{1}$ and
\begin{equation*}
\left\Vert x_{0}\right\Vert _{X_{0}}+t\left\Vert x_{1}\right\Vert
_{X_{1}}\leq \left\Vert z_{0}\right\Vert _{Q(X_{0})}+t\left\Vert
z_{0}\right\Vert _{Q(X_{1})}+\varepsilon
\end{equation*}%
Clearly that $v:=x_{0}+x_{1}-x\in {\text{ker}}\,Q$ and so
\begin{align*}
K(t,x+v;X_{0},X_{1})& =K(t,x_{0}+x_{1};X_{0},X_{1})\leq \left( \left\Vert
x_{0}\right\Vert _{X_{0}}+t\left\Vert x_{1}\right\Vert _{X_{1}}\right) \\
& \leq (\left\Vert z_{0}\right\Vert _{Q(X_{0})}+t\left\Vert z_{0}\right\Vert
_{Q(X_{1})})+\varepsilon \\
& \leq (1+\varepsilon )K(t,Qx;Q(X_{0}),Q(X_{1}))+\varepsilon {\text{.}}
\end{align*}%
Since $\varepsilon >0$ was arbitrary, we obtain the desired inequality.
\end{proof}

\subsection{Decomposition of Fredholm operator and formulation of the theorem%
}

Everywhere below the parameters $\theta \in \left( 0,1\right) $
and $q\in \left[ 1,\infty \right) $ are fixed and $A\colon
(X_{0},X_{1})\rightarrow (Y_{0},Y_{1})$ be a~surjective, Fredholm
operator such that $A\colon (X_{0},X_{1})_{\theta ,q}\rightarrow
(Y_{0},Y_{1})_{\theta ,q}$ is Fredholm. As $q\in \left[ 1,\infty
\right) $ then it follows from Theorem \ref{realfredholm} that the
spaces $V_{\theta ,q}^{0}\cap V_{\theta ,q}^{1}$ and $%
\widetilde{V}$ (see (\ref{NM1})) are finite-dimensional. Now we
can define operators $A_{1},$ $A_{2}$, $A_{3}$ in the following
way. Since $V_{\theta ,q}^{0}\cap V_{\theta ,q}^{1}$ is a~finite
dimensional space, it is a~closed subspace of $X_{0}+X_{1}$ and so
we can define an operator $A_{1}$ as a~quotient operator by the
formula
\begin{equation*}
A_{1}\colon X_{0}+X_{1}\rightarrow (X_{0}+X_{1})/(V_{\theta
,q}^{0}\cap V_{\theta ,q}^{1})
\end{equation*}%
We put $(Z_{0},Z_{1}):=(A_{1}(X_{0}),A_{1}(X_{1}))$. Now since the
space $\widetilde{V}$ is a~finite dimensional space,
$A_{1}(\widetilde{V})$ is a~closed subspace of
$Z_{0}+Z_{1}=A_{1}(X_{0})+A_{1}(X_{1})$ and we can define operator
$A_{2}$ as a~quotient operator by the formula
\begin{equation*}
A_{2}\colon Z_{0}+Z_{1}\rightarrow (Z_{0}+Z_{1})/A_{1}(\widetilde{V}).
\end{equation*}%
and
\begin{equation*}
(W_{0},W_{1}):=(A_{2}(Z_{0}),A_{2}(Z_{1}))\text{.}
\end{equation*}%
To define the operator $A_{3}$ we note that $\ker _{X_{0}+X_{1}}A$ is a
closed subspace of $X_{0}+X_{1}$ and $\ker _{X_{0}+X_{1}}A_{2}A_{1}\subset
\ker _{X_{0}+X_{1}}A$ therefore space $A_{2}(A_{1}(\ker _{X_{0}+X_{1}}A))$
is a closed subspace of the
\begin{equation*}
A_{2}(A_{1}(X_{0}+X_{1}))=A_{2}(A_{1}(X_{0}))+A_{2}(A_{1}(X_{1}))=W_{0}+W_{1}.
\end{equation*}%
Now we define an operator $Q$ by the formula
\begin{equation*}
Q\colon W_{0}+W_{1}\rightarrow (W_{0}+W_{1})/A_{2}(A_{1}((\ker
_{X_{0}+X_{1}}A)).
\end{equation*}%
We observe that operator $QA_{2}A_{1}$ is surjective on spaces $X_{0}$, $%
X_{1}$ and its kernel in $X_{0}+X_{1}$ coincides with the kernel of $A$.
Hence we may apply Proposition \ref{PrN2} to deduce that there exists an
operator $B\colon (Q(W_{0}),Q(W_{1}))\rightarrow (Y_{0},Y_{1})$ such that $%
A=BQA_{2}A_{1}$.

Throughout the paper we denote an operator $BQ$ by $A_{3}$. The
main result of this section is the following theorem, whose
formulation uses the
classes ${\mathbb{F}}_{\theta ,q}^{1}$, ${\mathbb{F}}_{\theta ,q}^{2}$, ${%
\mathbb{F}}_{\theta ,q}^{3}$, defined in introduction.


\begin{theorem}
\label{TN2} Suppose that an operator $A\colon (X_{0},X_{1})\rightarrow
(Y_{0},Y_{1})$ is Fredholm and surjective operator and let $1\leq $ $%
q<\infty $. Then $A\colon (X_{0},X_{1})_{\theta ,q}\rightarrow
(Y_{0},Y_{1})_{\theta ,q}$ is a~Fredholm operator if and only if
there exist Banach couples $(Z_{0},Z_{1})$, $(W_{0},W_{1})$ and
operators $A_1$, $A_2$, $A_3$ such that $A$ factors as
follows{\rm:}
\begin{equation*}
A\colon (X_0, X_1) \stackrel{A_1} \longrightarrow (Z_0, Z_1)
\stackrel{A_2} \longrightarrow (W_0, W_1) \stackrel{A_3}
\longrightarrow (Y_0, Y_1),
\end{equation*}%
where $A_{1}$, $A_{2}$, $A_{3}$ are operators from the classes
${\mathbb{F}}_{\theta ,q}^{1}$, ${\mathbb{F}}_{\theta ,q}^{2}$,
${\mathbb{F}}_{\theta ,q}^{3}$, respectively.
\end{theorem}


Since the classes ${\mathbb{F}}_{\theta ,q}^{1}$\textit{,}${\mathbb{F}}%
_{\theta ,q}^{2}$\textit{, }${\mathbb{F}}_{\theta ,q}^{3}$ consist
of Fredholm operators on $(X_{0},X_{1})_{\theta ,q}$ and
the superposition of Fredholm operators is a~Fredholm operator, the
sufficiency is obvious. Thus we
only need to prove that for the factorization $%
A=A_{3}A_{2}A_{1}$ constructed above we have $A_{i}\in {\mathbb{F}}_{\theta ,q}^{i}$, $%
i=1,2,3 $.

\subsection{The operator $A_{1}$ belongs to the class ${\mathbb{F}}_{\protect%
\theta ,q}^{1}$}

We recall that the class ${\mathbb{F}}_{\theta ,q}^{1}$ consists
of all surjective operators $T\colon (X_{0},X_{1})\rightarrow
(Z_{0},Z_{1})$ with finite dimensional $\ker _{X_{0}+X_{1}}T$ such
that
\begin{itemize}
\item[{\rm{a)}}] $T\colon (X_{0},X_{1})_{\theta ,q}\rightarrow
(Y_{0},Y_{1})_{\theta ,q}$ is surjective.
\item[{\rm{b)}}] $\ker
_{X_{0}+X_{1}}T\subset (X_{0},X_{1})_{\theta ,q}$.
\end{itemize}

By construction, the operator $A_{1}$ is surjective from the
couple $(X_{0},X_{1})$ to the couple $(Z_{0},Z_{1})$ and has
finite dimensional kernel $V_{\theta ,q}^{0}\cap V_{\theta
,q}^{1}$, which belongs to $(X_{0},X_{1})_{\theta ,q}$. Thus in
order to show that $A_{1}\in {\mathbb{F}}_{\theta ,q}^{1}$, it is
enough to prove the following lemma.


\begin{lemma}
\label{NL2} An operator $A_{1}\colon (X_{0},X_{1})\rightarrow
(Z_{0},Z_{1})$ maps $(X_{0},X_{1})_{\theta ,q}$ onto
$(Z_{0},Z_{1})_{\theta ,q}$.
\end{lemma}

\begin{proof} Since $A_{1}$ is a~quotient operator, it follows
from Proposition \ref{KF} that
\begin{equation*}
K(t,A_{1}x;Z_{0},Z_{1})=\func{inf}_{v\in V_{\theta ,q}^{0}\cap V_{\theta
,q}^{1}}K(t,x+v;X_{0},X_{1}).
\end{equation*}%
Hence
\begin{equation}
\left\Vert A_{1}x\right\Vert _{(Z_{0},Z_{1})_{\theta ,q}}=\bigg (%
\int_{0}^{\infty }\bigg (t^{-\theta }\func{inf}_{v\in V_{\theta ,q}^{0}\cap
V_{\theta ,q}^{1}}K(t,x+v;\vec{X})\bigg )^{q}\frac{dt}{t}\bigg )^{1/q}{\text{%
.}}  \label{NM7}
\end{equation}%
Since $A((X_{0},X_{1})_{\theta ,q})$ is a~closed subspace of $%
(Y_{0},Y_{1})_{\theta ,q}$ and $\ker _{X_{0}+X_{1}}A\cap
(X_{0},X_{1})_{\theta ,q}=V_{\theta ,q}^{0}\cap V_{\theta ,q}^{1}$, it
follows from the Banach theorem on inverse operator that
\begin{align}
\left\Vert Ax\right\Vert _{(Y_{0},Y_{1})_{\theta ,q}}& \approx \func{inf}%
_{v\in V_{\theta ,q}^{0}\cap V_{\theta ,q}^{1}}\left\Vert x+v\right\Vert
_{(X_{0},X_{1})_{\theta ,q}} \label{N9} \\
& = \func{inf}_{v\in V_{\theta ,q}^{0}\cap V_{\theta
,q}^{1}}\left(
\int_{0}^{\infty }\left( t^{-\theta }K(t,x+v;\vec{X})\right) ^{q}\frac{dt}{t}%
\right) ^{1/q},  \notag
\end{align}%
with constant of equivalence of norms independent of $x\in
(X_{0},X_{1})_{\theta ,q}$. Combine (\ref{NM7}) and (\ref{N9}) we deduce
that for all $x\in (X_{0},X_{1})_{\theta ,q}$ we have
\begin{equation*}
\left\Vert A_{1}x\right\Vert _{(Z_{0},Z_{1})_{\theta ,q}}\leq \gamma
\left\Vert Ax\right\Vert _{(Y_{0},Y_{1})_{\theta ,q}}
\end{equation*}%
with $\gamma >0$ independent of $x\in (X_{0},X_{1})_{\theta ,q}$. Since $%
A=A_{3}A_{2}A_{1}$, we conclude that for all elements $A_{1}x\in
(Z_{0},Z_{1})_{\theta ,q}$
\begin{equation*}
\left\Vert Ax\right\Vert _{(Y_{0},Y_{1})_{\theta ,q}}\leq \gamma \left\Vert
A_{1}x\right\Vert _{(Z_{0},Z_{1})_{\theta ,q}}
\end{equation*}%
and so we obtain the equivalence
\begin{equation}
\left\Vert Ax\right\Vert _{(Y_{0},Y_{1})_{\theta ,q}}\approx \left\Vert
A_{1}x\right\Vert _{(A_{1}(X_{0}),A_{1}(X_{1}))_{\theta ,q}}  \label{N10}
\end{equation}%
for all $x\in (X_{0},X_{1})_{\theta ,q}$ with the constants of equivalence of
norms independent of $x$.

We claim that $A_{1}((X_{0},X_{1})_{\theta ,q})$ is a closed subspace of $%
(Z_{0},Z_{1})_{\theta ,q}$. To see this assume that sequence
$x_{n}\in (X_{0},X_{1})_{\theta ,q}$ is such that $A_{1}(x_{n})$
converges to some element $z$ in $(Z_{0},Z_{1})_{\theta ,q}$. Then
$\{Ax_{n}\}=\left\{
A_{3}A_{2}(A_{1}x_{n})\right\} $ is a~Cauchy sequence in $%
(Y_{0},Y_{1})_{\theta ,q}$. Since $A((X_{0},X_{1})_{\theta ,q})$
is closed in $(Y_{0},Y_{1})_{\theta ,q}$, it follows that the
sequence $\{Ax_{n}\}$ converges to some element $Ax$ where $x\in
(X_{0},X_{1})_{\theta ,q}$. Then from (\ref{N10}), we conclude
that the sequence $\{A_{1}(x_{n})\}$ converges to $A_{1}(x)$ in
$(Z_{0},Z_{1})_{\theta ,q}$ and this proves the claim. Now observe
that if we show that $Z_{0}\cap Z_{1}=A_{1}(X_{0})\cap
A_{1}(X_{1})\subset A_{1}((X_{0},X_{1})_{\theta ,q})$, then from
$1\leq q<\infty $ it will follow that $A_{1}((X_{0},X_{1})_{\theta
,q})=(Z_{0},Z_{1})_{\theta ,q}$. Thus it is enough to prove that
\begin{equation*}
A_{1}(X_{0})\cap A_{1}(X_{1})\subset A_{1}((X_{0},X_{1})_{\theta ,q}).
\end{equation*}%
Fix $u\in A_{1}(X_{0})\cap A_{1}(X_{1})$. Since an operator $A_{1}$ maps
spaces $X_{0},X_{1}$ respectively onto spaces $A_{1}(X_{0}),$ $A_{1}(X_{1})$ there exist
elements $x_{0}\in X_{0}$ and $x_{1}\in X_{1}$ such that
\begin{equation*}
A_{1}(x_{0})=A_{1}(x_{1})=u.
\end{equation*}%
Hence
\begin{equation*}
v=x_{0}-x_{1}\in \ker _{X_{0}+X_{1}}A_{1}=V_{\theta ,q}^{0}\cap
V_{\theta ,q}^{1},
\end{equation*}%
i.e.,
\begin{equation*}
x_{0}-x_{1}=v\in V_{\theta ,q}^{0}\cap V_{\theta ,q}^{1}\subset
(X_{0},X_{1})_{\theta ,q}{\text{.}}
\end{equation*}%
Let $v=v_{0}+v_{1}$, $v_{0}\in X_{0}$, $v_{1}\in X_{1}$ be any decomposition
of element $v\in (X_{0},X_{1})_{\theta ,q}$, then $v_{0}$, $v_{1}\in
(X_{0},X_{1})_{\theta ,q}$ and
\begin{equation*}
x_{0}-v_{0}=x_{1}+v_{1}\in X_{0}\cap X_{1}{\text{.}}
\end{equation*}%
Denoting this element by $\tilde{x}$ we have
\begin{equation*}
u=A_{1}x_{0}=A_{1}\tilde{x}+A_{1}v_{0}\in
A_{1}((X_{0},X_{1})_{\theta ,q}),
\end{equation*}%
and this completes the proof.
\end{proof}

\subsection{The operator $A_{2}$ belongs to the class ${\mathbb{F}}_{\protect%
\theta ,q}^{2}$}

The class ${\mathbb{F}}_{\theta ,q}^{2}$ consists of all surjective
operators $T\colon (Z_{0},Z_{1})\rightarrow (W_{0},W_{1})$ with finite
dimensional $\ker _{Z_{0}+Z_{1}}T$ \ such that
\begin{itemize}
\item[{\rm{a)}}] operator $T\colon (Z_{0},Z_{1})_{\theta
,q}\rightarrow (W_{0},W_{1})_{\theta ,q}$ is injective.
\item[{\rm{b)}}] $T((Z_{0},Z_{1})_{\theta ,q})$ is a closed finite
codimensional subspace of $(W_{0},W_{1})_{\theta ,q}$ such that
\end{itemize}

\begin{equation*}
{\text{dim}}(W_{0},W_{1})_{\theta ,q}/T((Z_{0},Z_{1})_{\theta ,q})=\func{dim}%
\,(\ker _{Z_{0}+Z_{1}}T).
\end{equation*}

Note that operator $A_{2}$ by construction is a~quotient operator
therefore it is a~surjective operator from the couple $(Z_{0},Z_{1})$ to the couple $%
(W_{0},W_{1})$. Moreover, its kernel $\ker _{Z_{0}+Z_{1}}A_{2}=A_{1}(%
\widetilde{V})$ is a finite dimensional subspace by Theorem
\ref{realfredholm}. So to prove that $A_{2}\in
{\mathbb{F}}_{\theta ,q}^{2} $ we just need to prove the following
lemma.

\begin{lemma}
\label{NL3} Operator $A_{2}\colon (Z_{0},Z_{1})_{\theta ,q}\rightarrow
(W_{0},W_{1})_{\theta ,q}$ is injective and codimension of $%
A_{2}((Z_{0},Z_{1})_{\theta ,q})$ in $(W_{0},W_{1})_{\theta ,q}$ is finite
and is equal to the dimension of $\mathrm{\ker }_{Z_{0}+Z_{1}}A_{2}$.
\end{lemma}

\begin{proof} Let $\left\{ e_{1},...,e_{n}\right\} $ be a basis in $%
\widetilde{V}$, consider decompositions%
\begin{equation*}
e_{i}=e_{i}^{0}+e_{i}^{1},\text{ \ }e_{i}^{0}\in X_{0},\text{ }e_{i}^{1}\in
X_{1,}\text{ \ \ }i=1,...,n
\end{equation*}%
and define linear operators $P_{X_{0}}:\widetilde{V}\rightarrow
X_{0},P_{X_{1}}:\widetilde{V}\rightarrow X_{1}$
\begin{equation*}
P_{X_{0}}(\sum_{i=1}^{n}\lambda _{i}e_{i})=\sum_{i=1}^{n}\lambda
_{i}e_{i}^{0},\text{ \ }P_{X_{1}}(\sum_{i=1}^{n}\lambda
_{i}e_{i})=\sum_{i=1}^{n}\lambda _{i}e_{i}^{1}\text{.}
\end{equation*}%
Since $A_{1}$ maps $(X_{0},X_{1})_{\theta ,q}$ onto $(Z_{0},Z_{1})_{\theta
,q}$ and its kernel is equal to $\ker _{X_{0}+X_{1}}A\cap
(X_{0},X_{1})_{\theta ,q}$, it follows from $A=A_{3}A_{2}A_{1}$ that the
operator $A_{3}A_{2}\colon (Z_{0},Z_{1})_{\theta ,q}\rightarrow
(Y_{0},Y_{1})_{\theta ,q}$ is injective. In consequence, the operator $%
A_{2}\colon (Z_{0},Z_{1})_{\theta ,q}\rightarrow (W_{0},W_{1})_{\theta ,q}$
is also injective. From Theorem \ref{realfredholm} we have
\begin{equation}
(Y_{0},Y_{1})_{\theta ,q}=A((X_{0},X_{1})_{\theta ,q})\oplus A(P_{X_{0}}(%
\widetilde{V}))=A((X_{0},X_{1})_{\theta ,q})\oplus A_{3}A_{2}A_{1}(P_{X_{0}}(%
\widetilde{V}))\text{.}  \label{NM2}
\end{equation}%
Moreover, the dimension of \ $A(P_{X_{0}}(\widetilde{V}))$ is
equal to the dimension of the space $\widetilde{V}$ (see Remark
\ref{REM}), which is equal
to the dimension of the space $A_{1}(\widetilde{V})=\ker _{Z_{0}+Z_{1}}A_{2}$
Thus from (\ref{NM2}) it follows that the space $A_{2}(A_{1}P_{X_{0}}(%
\widetilde{V}))$ has dimension equal to the dimension of $\ker
_{Z_{0}+Z_{1}}A_{2}$ and
\begin{equation*}
A_{2}((Z_{0},Z_{1})_{\theta ,q})\cap A_{2}(A_{1}P_{X_{0}}(\widetilde{V}%
))=A_{2}A_{1}((X_{0},X_{1})_{\theta ,q})\cap A_{2}(A_{1}P_{X_{0}}(\widetilde{V}%
))=\left\{ 0\right\} .
\end{equation*}%
Now we will show that $A_{2}(A_{1}P_{X_{0}}(\widetilde{V}))\subset
W_{0}\cap
W_{1}$. Indeed, from equality $P_{X_{0}}\tilde{v}+P_{X_{1}}\tilde{v}=\tilde{v%
}$ for any element $\tilde{v}\in \widetilde{V}$ we have
\begin{equation*}
0=A_{2}A_{1}\tilde{v}=A_{2}A_{1}P_{X_{0}}\tilde{v}+A_{2}A_{1}P_{X_{1}}\tilde{%
v}
\end{equation*}%
and hence
\begin{equation*}
A_{2}A_{1}P_{X_{0}}\tilde{v}=-A_{2}A_{1}P_{X_{1}}\tilde{v}\in W_{0}\cap
W_{1}.
\end{equation*}%
Thus
\begin{equation*}
A_{2}((Z_{0},Z_{1})_{\theta
,q})+A_{2}(A_{1}P_{X_{0}}(\widetilde{V}))\subset
(W_{0},W_{1})_{\theta ,q}.
\end{equation*}%
To prove lemma we just need to show that $A_{2}((Z_{0},Z_{1})_{\theta ,q})$
is a closed subspace of $(W_{0},W_{1})_{\theta ,q}$ and that
\begin{equation*}
A_{2}((Z_{0},Z_{1})_{\theta ,q})+A_{2}(A_{1}P_{X_{0}}(\widetilde{V}%
))=(W_{0},W_{1})_{\theta ,q}.
\end{equation*}%
From Lemma \ref{NL2} and the Banach theorem of inverse operator,
we deduce that for any $z\in (Z_{0},Z_{1})_{\theta ,q}$ there
exists an $x\in (X_{0},X_{1})_{\theta ,q}$ such that $A_{1}x=z$
with
\begin{equation*}
\left\Vert x\right\Vert _{(X_{0},X_{1})_{\theta ,q}}\leq \gamma \left\Vert
Ax\right\Vert _{(Y_{0},Y_{1})_{\theta ,q}}.
\end{equation*}%
Combining the above estimate with $A=A_{3}A_{2}A_{1}$ and $A_{1}x=z$, we
obtain
\begin{align*}
\left\Vert A_{2}z\right\Vert _{(W_{0},W_{1})_{\theta ,q}}& \leq \gamma
\left\Vert z\right\Vert _{(Z_{0},Z_{1})_{\theta ,q}}\leq \gamma \left\Vert
x\right\Vert _{(X_{0},X_{1})_{\theta ,q}}\leq \\
\gamma \left\Vert Ax\right\Vert _{(Y_{0},Y_{1})_{\theta ,q}}& \leq \gamma
\left\Vert A_{3}A_{2}z\right\Vert _{(Y_{0},Y_{1})_{\theta ,q}}\leq \gamma
\left\Vert A_{2}z\right\Vert _{(W_{0},W_{1})_{\theta ,q}}.
\end{align*}%
If the sequence $\{w_{n}\}=\{A_{2}z_{n}\}$ with $z_{n}\in
(Z_{0},Z_{1})_{\theta ,q}$ for $n\in {\mathbb{N}}$ converges to
some element $w$ in the norm $(W_{0},W_{1})_{\theta ,q}$, then the
sequence $\{z_{n}\}$ is a~Cauchy's sequence and so converges to
some $z\in (Z_{0},Z_{1})_{\theta ,q}$. Hence
\begin{equation*}
A_{2}z=w
\end{equation*}%
and $w\in (W_{0},W_{1})_{\theta ,q}$. This shows that the space $%
A_{2}(A_{1}P_{X_{0}}(\widetilde{V}))\oplus
A_{2}((Z_{0},Z_{1})_{\theta ,q})$ is a~closed subspace of
$(W_{0},W_{1})_{\theta ,q}$ as a sum of a~closed
subspace and a~finite dimensional subspace. As $1\leq q<\infty $ then $%
W_{0}\cap W_{1}$ is dense linear subspace of
$(W_{0},W_{1})_{\theta ,q}$. Thus to show equality
\begin{equation}
A_{2}((Z_{0},Z_{1})_{\theta ,q})\oplus A_{2}(A_{1}P_{X_{0}}(\widetilde{V}%
))=(W_{0},W_{1})_{\theta ,q},  \label{NM3}
\end{equation}%
it is enough to prove that
\begin{equation}
W_{0}\cap W_{1}\subset A_{2}(A_{1}P_{X_{0}}(\widetilde{V}%
))+A_{2}((Z_{0},Z_{1})_{\theta ,q}).  \label{NM8}
\end{equation}%
Let $w\in W_{0}\cap W_{1}$. Since operator, $A_{2}A_{1}$ is
surjective on the spaces $X_{0}$ and $X_{1}$, there exist elements
$x_{0}\in X_{0},x_{1}\in X_{1}$ such that
\begin{equation*}
A_{2}A_{1}x_{0}=A_{2}A_{1}x_{1}=w.
\end{equation*}%
So $x_{0}-x_{1}\in \ker _{X_{0}+X_{1}}A_{2}A_{1}$ and we can find elements $%
v\in V_{\theta ,q}^{0}\cap V_{\theta ,q}^{1}$ and $\tilde{v}\in
\widetilde{V}$ such that \
\begin{equation*}
x_{0}-x_{1}=v+\tilde{v}.
\end{equation*}%
We decompose $v$ as $v=v_{X_{0}}+v_{X_{1}}$ with $v_{X_{0}}\in X_{0}$ and $%
v_{X_{1}}\in X_{1}$. Then
\begin{equation*}
x_{0}-x_{1}=v_{X_{0}}+v_{X_{1}}+P_{X_{0}}(\tilde{v})+P_{X_{1}}(\tilde{v})
\end{equation*}%
and so
\begin{equation*}
x_{0}-v_{X_{0}}-P_{X_{0}}(\tilde{v})=x_{1}+v_{X_{1}}+P_{X_{1}}(\tilde{v}).
\end{equation*}%
Let $u: x_{0}-v_{X_{0}}-P_{X_{0}}(\tilde{v})$; then $u\in
X_{0}\cap X_{1}$ and
\begin{equation*}
x_{0}=u+v_{X_{0}}+P_{X_{0}}(\tilde{v}){\text{.}}
\end{equation*}%
Therefore
\begin{equation*}
w=A_{2}A_{1}x_{0}=A_{2}A_{1}u+A_{2}A_{1}(v_{X_{0}})+A_{2}A_{1}(P_{X_{0}}(%
\tilde{v})).
\end{equation*}%
As $u\in X_{0}\cap X_{1}\in (X_{0},X_{1})_{\theta ,q}$ so $A_{2}A_{1}u\in
A_{2}((Z_{0},Z_{1})_{\theta ,q})$. As $v\in V_{\theta ,q}^{0}\cap V_{\theta
,q}^{1}\in (X_{0},X_{1})_{\theta ,q}$ therefore $v_{X_{0}}\in
(X_{0},X_{1})_{\theta ,q}$ and so $A_{2}A_{1}(v_{X_{0}})\in
A_{2}((Z_{0},Z_{1})_{\theta ,q})$ and we obtain (\ref{NM8}).
\end{proof}

\begin{remark}
\label{DIM} As the dimension of the kernel $A_{2}$ in
$Z_{0}+Z_{1}$ is equal to the dimension of space $\widetilde{V}$,
which is equal to the codimension of the space
$A((X_{0},X_{1})_{\theta ,q})$ in $(Y_{0},Y_{1})_{\theta ,q}$
$($see Theorem $\ref{realfredholm})$, therefore from Lemma
$\ref{NL3}$ we see that the codimension of
$A_{2}((Z_{0},Z_{1})_{\theta ,q})$ in $(W_{0},W_{1})_{\theta ,q}$
is equal to the codimension of the space $A((X_{0},X_{1})_{\theta
,q})$ in $(Y_{0},Y_{1})_{\theta ,q}$.
\end{remark}

\subsection{The operator $A_{3}$ belongs to the class ${\mathbb{F}}_{\protect%
\theta ,q}^{3}$}

Let us recall the class of operators ${\mathbb{F}}%
_{\theta ,q}^{3}$ consists of all $T\colon
(W_{0},W_{1})\rightarrow (Y_{0},Y_{1})$, which are surjective and
Fredholm on endpoint spaces and are such that $T\colon
(W_{0},W_{1})_{\theta ,q}\rightarrow (Y_{0},Y_{1})_{\theta ,q}$ is
invertible.

Since the operator $A=A_{3}A_{2}A_{1}$ is surjective and Fredholm on
endpoint spaces, the operator $A_{3}$ is also surjective on
endpoint spaces. Moreover, since the operators $A_{1}$ and $A_{2}$ are
surjective and $A$ is a Fredholm on endpoint spaces therefore
operator $A_{3}$ is a Fredholm operator on endpoint spaces. So we
just need to prove the following lemma.

\begin{lemma}
\label{NL4} An operator $A_{3}\colon (W_{0},W_{1})_{\theta ,q}\to
(Y_{0},Y_{1})_{\theta ,q}$ is invertible.
\end{lemma}

\begin{proof} In the proof in the Lemma \ref{NL3} we obtain (see (\ref{NM3}%
)) the decomposition $A_{2}((Z_{0},Z_{1})_{\theta ,q})\oplus
A_{2}(A_{1}P_{X_{0}}(\widetilde{V}))=(W_{0},W_{1})_{\theta ,q}$ , so from Lemma %
\ref{NL2} we have%
\begin{equation}
A_{2}(A_{1}(X_{0},X_{1})_{\theta ,q})\oplus A_{2}A_{1}(P_{X_{0}}\widetilde{V}%
)))=(W_{0},W_{1})_{\theta ,q}.  \label{NM4}
\end{equation}%
Hence
\begin{equation*}
A_{3}A_{2}(A_{1}(X_{0},X_{1})_{\theta ,q})+A_{3}A_{2}A_{1}(P_{X_{0}}\widetilde{V}%
)))\subset (Y_{0},Y_{1})_{\theta ,q}.
\end{equation*}%
Notice that the left hand side of the above equality can be written as $%
A((X_{0},X_{1})_{\theta ,q})+A(P_{X_{0}}\widetilde{V})$, which is equal to $%
(Y_{0},Y_{1})_{\theta ,q}$ (by Theorem \ref{realfredholm}), i.e., we have%
\begin{equation*}
A_{3}((W_{0},W_{1})_{\theta ,q})=A_{3}A_{2}(A_{1}(X_{0},X_{1})_{\theta
,q})\oplus A_{3}A_{2}A_{1}(P_{X_{0}}\widetilde{V})))=(Y_{0},Y_{1})_{\theta ,q}%
\text{.}
\end{equation*}
Thus $A_{3}((W_{0},W_{1})_{\theta ,q})=(Y_{0},Y_{1})_{\theta ,q}$.
From the construction we have%
\begin{equation*}
(X_0, X_1)_{\theta, q}\cap \ker_{X_{0}+X_{1}}A=\ker
_{X_{0}+X_{1}}A_{1}
\end{equation*}%
and the operator $AP_{X_{0}}$ is injective on $\widetilde{V}$ (see Remark \ref{DIM}%
)) therefore the kernels of operator $A_{3}$ on $A_{2}(A_{1}(X_{0},X_{1})_{%
\theta ,q})$ and $A_{2}A_{1}(P_{X_{0}}\widetilde{V})))$ are
trivial. Moreover, injectivity of $AP_{X_{0}}$ on $\widetilde{V}$
implies equality of codimensions of
$A_{2}(A_{1}(X_{0},X_{1})_{\theta ,q})$ in $(W_{0},W_{1})_{\theta
,q}$ and $A((X_0, X_1)_{\theta, q} = (Y_0,Y_1)_{\theta, q}$. From
this it follows that $A_{3}$ restricted to $(W_{0},W_{1})_{\theta
,q}$ has also a~trivial kernel. Indeed,
suppose that there exists element $w \in (W_{0},W_{1})_{\theta ,q}$ such that $%
A_{3}w =0$. Then we can decompose (see (\ref{NM4}))%
\begin{equation*}
w = w_{0} + w_{1}, \quad\, w_{0}\in
A_{2}(A_{1}(X_{0},X_{1})_{\theta ,q}), \quad\, w_{1}\in A_{2}A_{1}(P_{X_{0}}%
\widetilde{V}))).
\end{equation*}%
Then%
\begin{equation*}
A_{3}w_{1}=-A_{3} w_{0}\in A(X_{0},X_{1})_{\theta ,q}\cap
A(P_{X_{0}}\widetilde{V})=\left\{ 0\right\} .
\end{equation*}%
Thus $w_{0}$, $w_{1}$ are equal to zero. We have proved that
$A_{3}$ is injective and onto and so $A_{3}$ is invertible by the
Banach theorem on inverse operators.
\end{proof}

\section{Proof of Theorem \protect\ref{TN3} on sufficiency}

Theorem \ref{TN3} consists of three parts a), b) and c) with
rather long proofs. So for convenience we formulate each part as a
theorem.

\subsection{Sufficient condition for the class ${\mathbb{F}}_{\protect\theta %
,q}^{1}$}

Our aim is the following theorem.

\vspace{2mm}

\begin{theorem}
Let $A\colon (X_{0},X_{1})\rightarrow (Y_{0},Y_{1})$ be
a~surjective and Fredholm linear operator on endpoint spaces.
Suppose that
\begin{equation*}
\beta _{\infty }(\ker _{X_{0}+X_{1}}A)<\theta <\alpha _{0}(\ker
_{X_{0}+X_{1}}A).
\end{equation*}%
Then {\rm{a)}} $A$ is surjective from $(X_{0},X_{1})_{\theta ,q}$
to $(Y_{0},Y_{1})_{\theta ,q}$ and {\rm{b)}} $\ker
_{X_{0}+X_{1}}A\subset $ $(X_{0},X_{1})_{\theta ,q}$.
\end{theorem}

\begin{proof} We write proofs for $1\leq q<\infty $. However, the proof for $%
q=\infty $ is similar (it only requires some obvious changes).
First we prove that the property b) holds. Since $\beta _{\infty
}(\ker _{X_{0}+X_{1}}A)<\theta $, it follows that for any $x\in
\ker _{X_{0}+X_{1}}A $ we have the estimate
\begin{equation*}
K(t,x;\vec{X})\leq \gamma t^{\theta -\varepsilon }
\end{equation*}%
for some positive constants $\gamma ,\varepsilon $ and all $t\geq
1$. This implies that
\begin{equation*}
\int_{1}^{\infty }\big (t^{-\theta }K(t,x;\vec{X})\big )^{q}\frac{dt}{t}%
<\infty .
\end{equation*}%
Similarly from $\theta <\alpha _{0}(\ker _{X_{0}+X_{1}}A)$ we have that
\begin{equation*}
K(t,x;\vec{X})\leq \gamma t^{\theta +\varepsilon },\quad \,0<t\leq 1
\end{equation*}%
and so
\begin{equation*}
\int_{0}^{1}\big (t^{-\theta }K(t,x;X_{0},X_{1})\big )^{q}\frac{dt}{t}%
<\infty .
\end{equation*}%
Hence for any $x\in \ker _{X_{0}+X_{1}}A$ we have the inequality
\begin{equation*}
\int_{0}^{\infty }\big (t^{-\theta }K(t,x;\vec{X})\big )^{q}\frac{dt}{t}%
<\infty ,
\end{equation*}%
i.e., $x\in (X_{0},X_{1})_{\theta ,q}$.

Now we show that $A((X_{0},X_{1})_{\theta
,q})=(Y_{0},Y_{1})_{\theta ,q}$. Fix $y\in (Y_{0},Y_{1})_{\theta
,q}$. Then there exists decomposition
\begin{equation*}
y=\sum_{k=-\infty }^{\infty }y_{k},
\end{equation*}%
(the series converges in the sum $Y_{0}+Y_{1}$) such that
\begin{equation*}
\bigg(\sum_{k=-\infty }^{\infty }\big( 2^{-\theta k}J(2^{k},y_{k};\vec{Y})%
\big)^{q}\bigg)^{1/q}\leq \gamma \left\Vert y\right\Vert _{\theta ,q}{%
\text{.}}
\end{equation*}%
Since $A\colon (X_{0},X_{1})\rightarrow (Y_{0},Y_{1})$ is a~surjective
operator, for each $k\in {\mathbb{Z}}$ there exist elements $x_{0}^{k}\in
X_{0}$ and $x_{1}^{k}\in X_{1}$ such that
\begin{equation*}
Ax_{0}^{k}=Ax_{1}^{k}=y_{k}\quad {\text{and\, \thinspace
\thinspace $\Vert x_{0}^{k}\Vert _{X_{0}}\leq \gamma \left\Vert
y_{k}\right\Vert _{Y_{0}},\quad\, \Vert x_{1}^{k}\Vert
_{X_{1}}\leq \gamma \left\Vert y_{k}\right\Vert _{Y_{1}}$}}
\end{equation*}%
with constant $\gamma >0$ dependent only on $A$. So $%
v_{k}=x_{0}^{k}-x_{1}^{k}\in \ker _{X_{0}+X_{1}}A$ and we have
\begin{equation*}
K(2^{k},v_{k};\vec{X})\leq \Vert x_{0}^{k}\Vert _{X_{0}}+2^{k}\Vert
x_{1}^{k}\Vert _{X_{1}}\leq \gamma J(2^{k},y_{k};\vec{Y}){\text{.}}
\end{equation*}%
For each $k\in {\mathbb{Z}}$ we consider the decomposition $%
v_{k}=v_{0}^{k}+v_{1}^{k}$ such that $v_{0}^{k}\in X_{0}$, $v_{1}^{k}\in
X_{1}$ and
\begin{equation*}
\Vert v_{0}^{k}\Vert _{X_{0}}+2^{k}\Vert v_{1}^{k}\Vert _{X_{1}}\leq
2K(2^{k},v_{k};\vec{X}).
\end{equation*}%
We put
\begin{equation*}
\tilde{x}:=\sum_{k\leq 0}x_{0}^{k}+\sum_{k>0}x_{1}^{k}-\sum_{k\leq
0}v_{0}^{k}+\sum_{k>0}v_{1}^{k}.
\end{equation*}%
Note that from the above estimation it follows that the series on the right
hand side is absolutely convergent in $Y_{0}+Y_{1}$. Moreover, it is easy to
check that for each $n\in {\mathbb{Z}}$,
\begin{equation*}
\tilde{x}:=\sum_{k\leq n}x_{0}^{k}+\sum_{k>n}x_{1}^{k}-\sum_{k\leq
n}v_{0}^{k}+\sum_{k>n}v_{1}^{k}.
\end{equation*}%
Thus we have
\begin{align*}
K(2^{n},\tilde{x};\vec{X})& \leq \sum_{k\leq n}\Vert x_{0}^{k}\Vert
_{X_{0}}+2^{n}\sum_{k>n}\Vert x_{1}^{k}\Vert _{X_{1}}+\sum_{k\leq n}\Vert
v_{0}^{k}\Vert _{X_{0}}+2^{n}\sum_{k>n}\Vert v_{1}^{k}\Vert _{X_{1}} \\
& \leq \gamma \Big (\sum_{k\leq n}J(2^{k},y_{k};\vec{Y})+2^{n}\sum_{k>n}%
\frac{1}{2^{k}}J(2^{k},y_{k};\vec{Y})\Big )=\gamma S_{d}\big (\big \{%
J(2^{k},y_{k};\vec{Y})\}\big )(2^{n}),
\end{align*}%
where $S_{d}$ is the ~discrete analog of the well-known Calder\'on
operator (see \cite{AK1}), which is bounded in the weighted sequence space $%
\ell ^{q}\big(\{ 2^{-n\theta}\}\big)$ for all parameters $\theta $ and $%
1\leq q\leq \infty$. In consequence $\tilde{x}\in
(X_{0},X_{1})_{\theta ,q}$.

Now we show that
\begin{equation*}
\bar{x}:=\sum_{k\leq 0}v_{0}^{k}-\sum_{k>0}v_{1}^{k}\in
(X_{0},X_{1})_{\theta ,q}.
\end{equation*}%
Indeed for all $n>0$ we have
\begin{equation*}
\bar{x}=\sum_{k\leq n}v_{0}^{k}-\sum_{k>n}v_{1}^{k}-\sum_{0<k\leq
n}v_{0}^{k}-\sum_{0<k\leq n}v_{1}^{k}=\sum_{k\leq
n}v_{0}^{k}-\sum_{k>n}v_{1}^{k}-\sum_{0<k\leq n}v_{k}.
\end{equation*}%
This implies that
\begin{eqnarray*}
K(2^{n},\bar{x};\vec{X}) &\leq \sum_{k\leq n}J(2^{k},y_{k};\vec{Y}%
)+2^{n}\sum_{k>n}\frac{1}{2^{k}}J(2^{k},y_{k};\vec{Y})+\sum_{0<k\leq
n}K(2^{n},v_{k};\vec{X}) \\
&\leq \gamma S_{d}\left( \left\{ J(2^{k},y_{k};\vec{Y})\right\}
\right) (2^{n})+\sum_{0<k\leq n}K(2^{n},v_{k};\vec{X}).
\end{eqnarray*}%
Since $v_{k}\in \ker _{X_{0}+X_{1}}A$ and $\theta >\beta _{\infty }(\ker
_{X_{0}+X_{1}}A)$, there exist positive constants $\gamma ,\varepsilon $
such that
\begin{equation*}
K(2^{n},v_{k};\vec{X})\leq \gamma \left( \frac{2^{n}}{2^{k}}\right) ^{\theta
-\varepsilon }K(2^{k},v_{k};\vec{X})\leq \gamma \left( \frac{2^{n}}{2^{k}}%
\right) ^{\theta -\varepsilon }J(2^{k},y_{k};\vec{Y})
\end{equation*}%
for all $0\leq k\leq n$. Hence
\begin{equation*}
\sum_{0<k\leq n}K(2^{n},v_{k};X_{0},X_{1})\leq \gamma \sum_{0<k\leq n}\left(
\frac{2^{n}}{2^{k}}\right) ^{\theta -\varepsilon }J(2^{k},y_{k};Y_{0},Y_{1}).
\end{equation*}%
Applying Minkowski's inequality we obtain
\begin{align*}
& \bigg (\sum_{n>0}\bigg (2^{-\theta n}\sum_{0<k\leq n}\bigg (\frac{2^{n}}{%
2^{k}}\bigg )^{\theta -\varepsilon }J(2^{k},y_{k};\vec{Y})\bigg )^{q}\bigg )%
^{1/q}=\bigg (\sum_{n>0}\bigg (\sum_{0<k\leq n}2^{-(n-k)\varepsilon
}2^{-k\theta }J(2^{k},y_{k};\vec{Y})\bigg )^{q}\bigg )^{1/q} \\
& \leq \sum_{j\geq 0}\bigg (\sum_{k>0}\big(2^{-j\varepsilon
}2^{-k\theta}J(2^{k},y_{k};\vec{Y})\big)^{q}\bigg )^{1/q}\leq \gamma \bigg (\sum_{k>0}%
\big(2^{-k\theta }J(2^{k},y_{k};\vec{Y})\big)^{q}\bigg )^{1/q}\leq
\gamma \left\Vert y\right\Vert _{\theta ,q}{\text{.}}
\end{align*}%
Therefore, using boundedness of the discrete analog of the
Calder\'{o}n operator $S_{d}$, we obtain
\begin{equation}
\Big (\sum_{n>0}^{\infty }\big (2^{-\theta n}K(2^{n},\bar{x};\vec{Y})\big )%
^{q}\Big )^{1/q}\leq \gamma \left\Vert y\right\Vert _{\theta ,q}{\text{.}}
\label{NM10}
\end{equation}%
Similarly we consider the case $n<0$. Indeed,
\begin{align*}
\bar{x}& =\sum_{k\leq 0}v_{0}^{k}-\sum_{k>0}v_{1}^{k}=\sum_{k\leq
n}v_{0}^{k}-\sum_{k>n}v_{1}^{k}+\sum_{n<k\leq 0}v_{0}^{k}+\sum_{n<k\leq
0}v_{1}^{k} \\
& =\sum_{k\leq n}v_{0}^{k}-\sum_{k>n}v_{1}^{k}+\sum_{n<k\leq 0}v_{k}.
\end{align*}%
Moreover,
\begin{eqnarray*}
K(2^{n},\bar{x};\vec{X}) &\leq &\sum_{k\leq n}J(2^{k},y_{k};\vec{Y}%
)+2^{n}\sum_{k>n}\frac{1}{2^{k}}J(2^{k},y_{k};\vec{Y})+\sum_{n<k\leq
0}K(2^{n},v_{k};X_{0},X_{1}) \\
&&\gamma S_{d}\left( \left\{ J(2^{k},y_{k};\vec{Y})\right\} \right)
(2^{n})+\sum_{n<k\leq 0}K(2^{n},v_{k};X_{0},X_{1}).
\end{eqnarray*}%
Since $v_{k}\in \ker _{X_{0}+X_{1}}A$ and $\theta <\alpha _{0}(\ker
_{X_{0}+X_{1}}A)$, there exist positive constants $\gamma ,\varepsilon $
such that
\begin{equation*}
K(2^{n},v_{k};\vec{X})\leq \gamma \left( \frac{2^{n}}{2^{k}}\right) ^{\theta
+\varepsilon }K(2^{k},v_{k};\vec{X})\leq \gamma \left( \frac{2^{n}}{2^{k}}%
\right) ^{\theta +\varepsilon }J(2^{k},y_{k};\vec{Y})
\end{equation*}%
for each $n\leq k\leq 0$. Again, applying the Minkowski inequality we get
\begin{align*}
\bigg (& \sum_{n<0}\Big (2^{-\theta n}\sum_{n<k\leq 0}\Big (\frac{2^{n}}{%
2^{k}}\Big )^{\theta +\varepsilon }J(2^{k},y_{k};\vec{Y})\bigg )^{q}\bigg )%
^{1/q}=\bigg (\sum_{n<0}\Big (\sum_{n<k\leq 0}2^{-(k-n)\varepsilon
}2^{-k\theta }J(2^{k},y_{k};\vec{Y})\Big )^{q}\bigg )^{1/q} \\
& \leq \sum_{j>0}\Big (\sum_{k\leq 0}\big (2^{-j\varepsilon }2^{-k\theta
}J(2^{k},y_{k};\vec{Y})\big )^{q}\Big )^{1/q}\leq \gamma \Big (\sum_{k\leq 0}%
\Big (2^{-k\theta }J(2^{k},y_{k};\vec{Y})\Big )^{q}\Big )^{1/q}\leq \gamma
\left\Vert y\right\Vert _{\theta ,q}.
\end{align*}%
This implies
\begin{equation*}
\bigg (\sum_{n<0}^{\infty }\big (2^{-\theta n}K(2^{n},\bar{x};\vec{Y})\big )%
^{q}\bigg )^{1/q}\leq \gamma \left\Vert y\right\Vert _{\theta ,q}\text{ }
\end{equation*}%
and whence by (\ref{NM10})
\begin{equation*}
\bigg (\sum_{n=-\infty }^{\infty }\big (2^{-\theta n}K(2^{n},\bar{x};\vec{Y})%
\big )^{q}\bigg )^{1/q}\leq 2\,\gamma \left\Vert y\right\Vert _{\theta ,q}{%
\text{.}}
\end{equation*}%
The estimates obtained for elements $x$ and $\tilde{x}$ show that
\begin{equation*}
x=\sum_{k\leq 0}x_{0}^{k}+\sum_{k>0}x_{1}^{k}=\bar{x}+\tilde{x}\in
(X_{0},X_{1})_{\theta ,q}.
\end{equation*}%
Since $Ax=A\Big (\sum_{k\leq 0}x_{0}^{k}+\sum_{k>0}x_{1}^{k}\Big )%
=\sum_{k\in {\mathbb{Z}}}y_{k}=y$, thus $A((X_{0},X_{1})_{\theta
,q})=(Y_{0},Y_{1})_{\theta ,q}$ as required.
\end{proof}


\subsection{Sufficient condition for the class ${\mathbb{F}}_{\protect\theta %
,q}^{2}$}

The proof of the next theorem is similar to the previous one, but more complicated.%


\begin{theorem}
\label{ST2}Let $A\colon (X_{0},X_{1})\rightarrow (Y_{0},Y_{1})$ be
a~surjective and Fredholm linear operator on endpoint spaces.
Suppose also that $\dim \left( \ker _{X_{0}+X_{1}}A\right) <\infty
$ and
\begin{equation}
\beta _{0}(\ker _{X_{0}+X_{1}}\,A)<\theta <\alpha _{\infty }(\ker
_{X_{0}+X_{1}}\,A).  \label{K5}
\end{equation}%
Then we have
\begin{itemize}
\item[{\rm{a)}}] The operator $A\colon (X_{0},X_{1})_{\theta
,q}\rightarrow (Y_{0},Y_{1})_{\theta ,q}$ is injective.
\item[{\rm{b)}}] The codimension of $A((X_{0},X_{1})_{\theta ,q})$
in $ (Y_{0},Y_{1})_{\theta ,q}$ is equal to the dimension of $\ker
_{X_{0}+X_{1}}\,A$, i.e., $A\in \mathbb{F}_{\theta
,q}^{2}$\textit{.}
\end{itemize}
\end{theorem}

We write the proof for $q<\infty $. However, the case of $q=\infty
$ is similar. We first prove that \linebreak $A$ is injective on
$(X_{0},X_{1})_{\theta ,q}$. If this would be not true, then there
would exists $x\in (X_{0},X_{1})_{\theta ,q}\linebreak \cap\,\ker
_{X_{0}+X_{1}}A$. From $\theta <\alpha _{\infty }(\ker
_{X_{0}+X_{1}}A)$, it follows that there exist positive constants
$\gamma $, $\varepsilon $ such that
\begin{equation}
K(t,x;\vec{X})\geq \gamma t^{\theta +\varepsilon }K(1,x;\vec{X})  \label{K6}
\end{equation}%
for all $t\geq 1$. So from $x\in (X_{0},X_{1})_{\theta ,q}$ we have $x=0$.

We observe that from inequalities (\ref{K5}), it follows that the operator $%
A $ is invertible on endpoint spaces. Indeed $x\in X_{0}\cap \ker
_{X_{0}+X_{1}}A$ implies $\sup_{t>0}K(t,x;\vec{X})<\infty $ and so
we arrive to a~contradiction to (\ref{K6}). Similarly we explain
that $X_{1}\cap \ker _{X_{0}+X_{1}}\,A=\{0\}$ by $\beta _{0}(\ker
_{X_{0}+X_{1}}A)<\theta $.

The proof that the codimension of $A((X_{0},X_{1})_{\theta ,q})$ in $%
(Y_{0},Y_{1})_{\theta ,q}$ is equal to the dimension of $\ker
_{X_{0}+X_{1}}A $ is more complicated. We first show that
\begin{equation*}
(Y_{0},Y_{1})_{\theta ,q}=A((X_{0},X_{1})_{\theta ,q})\oplus M,
\end{equation*}%
where $M$ is a~finite dimensional space with dimension equal to
the dimension of $\ker _{X_{0}+X_{1}}A$. We start with
construction of the space $M$. Let $g_{1},...,g_{n}$ be a~basis in
$\ker _{X_{0}+X_{1}}A$ and let us consider decompositions
\begin{equation*}
g_{m}=g_{m}^{0}+g_{m}^{1}
\end{equation*}%
such that
\begin{equation*}
\left\Vert g_{m}^{0}\right\Vert _{X_{0}}+2^{0}\left\Vert
g_{m}^{1}\right\Vert _{X_{1}}\leq 2K(2^{0},g_{m};\vec{X})=2\left\Vert
g_{m}\right\Vert _{X_{0}+X_{1}}{\text{.}}
\end{equation*}%
Now we define linear operators $\func{Pr}_{X_{0}}\colon \ker
_{X_{0}+X_{1}}A\rightarrow X_{0}$ and $\func{Pr}_{X_{1}}\colon
\ker _{X_{0}+X_{1}}A\rightarrow X_{1}$ by the formulas
\begin{equation*}
\func{Pr}{}_{X_{0}}\Big (\sum_{m=1}^{n}\alpha _{m}g_{m}\Big )%
=\sum_{m=1}^{n}\alpha_{m}g_{m}^{0},\quad \,{\text{
}}\func{Pr}_{X_{1}}(\sum_{m=1}^{n}\alpha _{m}g_{m})
=\sum_{m=1}^{n}\alpha _{m}g_{m}^{1}{%
\text{.}}
\end{equation*}%
Since $g=\func{Pr}_{X_{0}}(g)+\func{Pr}_{X_{1}}(g)$ and $g\in \ker
_{X_{0}+X_{1}}A$,
\begin{equation*}
A(\func{Pr}_{X_{0}}(g))=-A(\func{Pr}_{X_{1}}(g)),
\end{equation*}%
so
\begin{equation*}
A(\func{Pr}_{X_{0}}(\ker
_{X_{0}+X_{1}}A))=A(\func{Pr}_{X_{1}}(\ker _{X_{0}+X_{1}}A))
\end{equation*}%
and we can define the space $M$ as $A(\func{Pr}{}_{X_{0}}(\ker
_{X_{0}+X_{1}}A))$.

From the invertibility of $A\colon (X_{0},X_{1})\rightarrow
(Y_{0},Y_{1})$ it follows that ${\text{ker}}_{X_{0}+X_{1}}A\cap
X_{i}=\{0\}$ for $i=0,1$. In consequence, the operators
$\func{Pr}_{X_{0}}$, $\func{Pr}_{X_{1}}$ are both injective and
the dimension of the space $A(\func{Pr}{}_{X_{0}}(\ker
_{X_{0}+X_{1}}A))$ is equal to the dimension of $\ker
_{X_{0}+X_{1}}A$.

We claim that
\begin{equation*}
A((X_{0},X_{1})_{\theta ,q})\cap A( \func{Pr}_{X_{0}}(\ker
_{X_{0}+X_{1}}A))=\left \{0\right \}.
\end{equation*}

Suppose to the contrary that there exists an element $v\in \ker
_{X_{0}+X_{1}}A\setminus \left\{ 0\right\}$ and element $x\in
(X_{0},X_{1})_{\theta ,q}$ such that
\begin{equation*}
A(\func{Pr}_{X_{0}}(v))=Ax.
\end{equation*}%
Then $u:=\func{Pr}{}_{X_{0}}(v)-x\in \ker _{X_{0}+X_{1}}A$. There
are two cases: $u=0$ and $u\neq 0$. If $u=0$, then
$\func{Pr}{}_{X_{0}}(v)\in (X_{0},X_{1})_{\theta ,q}$ and
therefore
\begin{equation*}
K(s,v;\vec{X})\leq K(s,\func{Pr}_{X_{0}}(v);\vec{X})+K(s,\func{Pr}{}%
_{X_{1}}(v);\vec{X})\leq cs^{\theta }+cs.
\end{equation*}%
However from $\theta >\beta _{0}(\ker _{X_{0}+X_{1}}A)\geq \beta _{0}(v)$,
it follows that there exists $\varepsilon >0$ such that for all $0<s\leq 1$,
\begin{equation*}
\frac{K(s,v;\vec{X})}{K(1,v;\vec{X})}\geq \gamma s^{\theta -\varepsilon }{%
\text{ ,}}
\end{equation*}%
and we arrive at a~contradiction. If $u\neq 0$, then
\begin{equation}
K(t,u;\vec{X})\leq
K(t,\func{Pr}_{X_{0}}(v);\vec{X})+K(t,x;\vec{X})\leq c+ct^{\theta
}{\text{.}}  \label{K103}
\end{equation}%
To conclude we observe that $\theta <\alpha _{\infty }(\ker
_{X_{0}+X_{1}}A)\leq \alpha _{\infty }(u)$ implies that for some $%
\varepsilon $ and for all $t\geq 1$,
\begin{equation*}
\frac{K(1,u;\vec{X})}{K(t,u;\vec{X})}\leq \gamma \left( \frac{1}{t}\right)
^{\theta +\varepsilon }{\text{ }}{;}
\end{equation*}%
we again arrive at a~contradiction by (\ref{K103}) and this proves the
claim.

Next, we show that $(Y_{0},Y_{1})_{\theta ,q}=A((X_{0},X_{1})_{\theta ,q})+A(%
\func{Pr}_{X_{0}}(\ker _{X_{0}+X_{1}}A))$. To do this, fix $y\in
(Y_{0},Y_{1})_{\theta ,q}$; then there exist $y_{k}\in Y_{0}\cap
Y_{1}$ such that $y=\sum_{k}y_{k}$ (convergence in $Y_{0}+Y_{1}$)
and
\begin{equation*}
\bigg (\sum_{k\in {\mathbb{Z}}}\big(2^{-\theta k}J(2^{k},y_{k};\vec{Y})\big)^{q}%
\bigg )^{1/q}\leq 2\left\Vert y\right\Vert _{\theta ,q}{\text{.}}
\end{equation*}%
Since $A$ is invertible operator on endpoint spaces, for each $k$
we can find elements $x_{0}^{k}\in X_{0}$ and $x_{1}^{k}\in X_{1}$
such that
\begin{equation*}
Ax_{0}^{k}=Ax_{1}^{k}=y_{k}
\end{equation*}%
and
\begin{equation*}
\Vert x_{0}^{k}\Vert _{X_{0}}+2^{k}\Vert x_{1}^{k}\Vert _{X_{1}}\leq \gamma
J(2^{k},y_{k};\vec{Y}).
\end{equation*}%
Now we show that series $\sum_{k\leq
0}x_{0}^{k}+\sum_{k>0}x_{1}^{k}$ absolutely converge in
$X_{0}+X_{1}$ and for $x:=\sum_{k\leq
0}x_{0}^{k}+\sum_{k>0}x_{0}^{k}$ we have $Ax=y$.

Indeed, absolute convergence follows from the estimate
\begin{equation*}
\sum_{k\leq 0}\Vert x_{0}^{k}\Vert _{X_{0}}+\sum_{k>0}\Vert x_{1}^{k}\Vert
_{X_{1}}\leq \gamma \bigg (\sum_{k\leq 0}J(2^{k},y_{k};\vec{Y})+\sum_{k>0}%
\frac{1}{2^{k}}J(2^{k},y_{k};\vec{Y})\bigg )\leq \gamma \left\Vert
y\right\Vert _{\theta ,q}
\end{equation*}%
and equality $Ax=y$ follows from the fact that $Ax_{0}^{k}=Ax_{1}^{k}=y_{k}$
for each $k$.

So to construct the needed decomposition of the element $y\in
(Y_{0},Y_{1})_{\theta ,q}$, it is enough to construct
decomposition of the element $x=\sum_{k\leq
0}x_{0}^{k}+\sum_{k>0}x_{0}^{k}$ into a~sum of two elements, one of
which belongs to $(X_{0},X_{1})_{\theta, q}$ and the other to $\func{Pr}{}%
_{X_{0}}(\ker _{X_{0}+X_{1}}A)+\func{Pr}{}_{X_{1}}(\ker _{X_{0}+X_{1}}A)$.
Let us consider the elements
\begin{equation*}
V_{k}=x_{0}^{k}-x_{1}^{k}\in \ker _{X_{0}+X_{1}}A,\quad \,k\in {\mathbb{Z}}
\end{equation*}%
and let $(V_{k})_{X_{0}}^{i}\in X_{0}$, $(V_{k})_{X_{1}}^{i}\in X_{1}$ for
each $i\in {\mathbb{Z}}$ be such that
\begin{equation*}
V_{k}=(V_{k})_{X_{0}}^{i}+ (V_{k})_{X_1}^{i}, \quad\, i\in
{\mathbb{Z\setminus }}\{0\}
\end{equation*}%
and
\begin{equation}
\big \|\left( V_{k}\right) _{X_{0}}^{i}\big \|_{X_{0}}+2^{i}\big
\|\left( V_{k}\right)_{X_{1}}^{i}\big \|_{X_{1}}\leq \gamma
K(2^{i},V_{k};\vec{X}). \label{K7}
\end{equation}%
For each $k\in {\mathbb{Z}}$ we define elements $\left( V_{k}\right)
_{X_{0}}^{0}$, $\left( V_{k}\right) _{X_{1}}^{0}$ by
\begin{equation*}
\left( V_{k}\right)_{X_{0}}^{0}=\func{Pr}{}_{X_{0}}(V_{k}),{\text{
}}\left( V_{k}\right) _{X_{1}}^{0}=\func{Pr}_{X_{1}}(V_{k}).
\end{equation*}%
Since $\func{dim}(\ker _{X_{0}+X_{1}}A)<\infty $, the norms
$\left\Vert \cdot \right\Vert _{X_{0}+X_{1}}$ , $\left\Vert
\func{Pr}{}_{X_{0}}(\cdot )\right\Vert _{X_{0}{\text{ }}}$and
$\left\Vert \func{Pr}{}_{X_{1}}(\cdot )\right\Vert _{X_{1}{\text{
}}}$ are equivalent on $\ker _{X_{0}+X_{1}}A$. In consequence, the
inequality (\ref{K7}) is also valid for $i=0$.

Moreover from the invertibility of the operator $A$ on endpoint
spaces we have
\begin{equation}
K(2^{k},V_{k};\vec{X})\leq \Vert x_{0}^{k}\Vert _{X_{0}}+2^{k}\Vert
x_{1}^{k}\Vert _{X_{0}}\leq \gamma J(2^{k},y_{k};\vec{Y}).  \label{T5.3}
\end{equation}

We need the following proposition.


\begin{proposition}
\label{Proposition4} The series $\sum _{k\leq 0}\left(V_{k}\right
)_{X_{0}}^{0}$ converges absolutely in $X_{0}$ to some $%
u\in \func{Pr}{}_{X_0}(\ker _{X_{0}+X_{1}}A)$ and the series $\sum
_{k>0}\left (V_{k}\right )_{X_{1}}^{0}$ converges absolutely in
$X_{1}$ to some $v\in \func{Pr}{}_{X_{1}}(\ker _{X_{0}+X_{1}}A).$
\end{proposition}


\begin{proof} Since $\theta >\beta _{0}(\ker _{X_{0}+X_{1}}A)$, there
exists $\varepsilon >0$ such that for all $V\in \ker _{X_{0}+X_{1}}A$ we
have inequality
\begin{equation*}
\frac{K(s,V;\vec{X})}{K(t,V;\vec{X})}\geq \gamma \left( \frac{s}{t}\right)
^{\theta -\varepsilon },0<s<t\leq 1.
\end{equation*}%
Combining it with (\ref{K7}), (\ref{T5.3}) and H\"{o}lder's inequality we
obtain
\begin{align*}
& \sum_{k\leq 0}\big \|\left( V_{k}\right) _{X_{0}}^{0}\big \|_{X_{0}}\leq
\gamma \sum_{k\leq 0}K(2^{0},V_{k};\vec{X})\leq \gamma \sum_{k\leq
0}K(2^{k},V_{k};\vec{X})\left( \frac{2^{0}}{2^{k}}\right) ^{\theta
-\varepsilon } \\
& =\gamma \sum_{k\leq 0}2^{-k\theta }K(2^{k},V_{k};\vec{X})\cdot
2^{k\varepsilon }\leq \gamma \bigg (\sum_{k<0}\left( 2^{-k\theta
}K(2^{k},V_{k};\vec{X})\right) ^{q}\big )^{1/q}\leq \gamma \left\Vert
y\right\Vert _{\theta ,q}{\text{.}}
\end{align*}
Since element $\left( V_{k}\right) _{X_{0}}^{0}$ belongs to a~finite
dimensional subspace $\func{Pr}{}_{X_{0}}(\ker _{X_{0}+X_{1}}A)$
of the space $X_{0}$ therefore the sum $\sum_{k\leq 0}\left(
V_{k}\right) _{X_{0}}^{0}$ converges in $X_{0}$ to some element in
$\func{Pr}{}_{X_{0}}(\ker _{X_{0}+X_{1}}A)$. Similarly $\theta
<\alpha _{\infty }(\ker _{X_{0}+X_{1}}A)$ yields that there exists
$\varepsilon >0$ such that
\begin{equation*}
\frac{K(s,V;\vec{X})}{K(t,V;\vec{X})}\leq \gamma \left( \frac{s}{t}\right)
^{\theta +\varepsilon },\quad \,1\leq s<t.
\end{equation*}%
Hence
\begin{align*}
\sum_{k>0}& \Vert \left( V_{k}\right) _{X_{1}}^{0}\Vert _{X_{1}}\leq \gamma
\sum_{k>0}K(2^{0},V_{k};\vec{X})\leq \gamma \sum_{k>0}K(2^{k},V_{k};\vec{X}%
)\left( \frac{2^{0}}{2^{k}}\right) ^{\theta +\varepsilon } \\
& \leq \gamma \sum_{k>0}2^{-k\theta }K(2^{k},V_{k};\vec{X})\cdot
2^{-k\varepsilon }\leq \gamma \bigg (\sum_{k>0}\big (2^{-k\theta
}K(2^{k},V_{k};\vec{X})\big )^{q}\bigg )^{1/q}\leq \gamma \left\Vert
y\right\Vert _{\theta ,q}{\text{.}}
\end{align*}%
Since element $\left( V_{k}\right) _{X_{1}}^{0}$ belongs to
a~finite dimensional subspace $\func{Pr}{}_{X_{1}}(\ker
_{X_{0}+X_{1}}A)$ of the space $X_{1}$, the sum
$\sum_{k>0}\left( V_{k}\right) _{X_{1}}^{0}$ converges in $%
X_{1}$ to some element in $\func{Pr}{}_{X_{1}}(\ker _{X_{0}+X_{1}}A)$.
\end{proof}

Hence the series
\begin{equation*}
\sum_{k\leq 0}x_{0}^{k}+\sum_{k>0}x_{1}^{k}-\sum_{k\leq 0}\left(
V_{k}\right)_{X_{0}}^{0}+\sum_{k>0}\left( V_{k}\right)
_{X_{1}}^{0}
\end{equation*}%
converges absolutely in $X_{0}+X_{1}$ to some element $\tilde{x}$. In fact $%
\tilde{x}\in $ $(X_{0},X_{1})_{\theta ,q}$. If we prove this then from
Proposition \ref{Proposition4} we have
\begin{equation*}
x=\sum_{k\leq 0}x_{0}^{k}+\sum_{k>0}x_{1}^{k}=\tilde{x}-u+v,
\end{equation*}%
where $u\in \func{Pr}_{X_{1}}(\ker _{X_{0}+X_{1}}A),v\in
\func{Pr}_{X_{1}}(\ker _{X_{0}+X_{1}}A)$ and therefore
\begin{align*}
y& =Ax=A\tilde{x}-Au+Av\in (X_{0},X_{1})_{\theta
,q}+A(\func{Pr}_{X_{0}}(\ker
_{X_{0}+X_{1}}A))+A(\func{Pr}_{X_{1}}(\ker _{X_{0}+X_{1}}A)
\\
& =(X_{0},X_{1})_{\theta ,q}+A(\func{Pr}_{X_{0}}(\ker
_{X_{0}+X_{1}}A)
\end{align*}%
and we are done. The proof of the fact that $\tilde{x}\in $ $%
(X_{0},X_{1})_{\theta ,q}$ is based on two lemmas.

Before we state these lemmas we observe that
\begin{equation*}
\sum _{k\leq 0}\big \|(V_{k})_{X_{0}}^{k}\big \|_{X_0}+\sum _{k>0}\big \|%
(V_{k})_{X_{1}}^{k}\big \|_{X_1}\leq \gamma \bigg (\sum _{k\leq
0}J(2^{k},y_{k};\vec {Y})+\sum _{k>0}\frac{1}{2^{k}}J(2^{k},y_{k};\vec {Y})%
\bigg ).
\end{equation*}
This implies that
\begin{equation*}
\sum _{k\leq 0}x_{0}^{k}+\sum _{k>0}x_{1}^{k}-\sum _{k\leq 0}\left
(V_{k}\right )_{X_{0}}^{k}+\sum _{k>0}\left (V_{k}\right )_{X_{1}}^{k}
\end{equation*}
converges absolutely in $X_{0}+X_{1}$.

\begin{lemma}
\label{Lemma6} The element $\bar {x}=\sum _{k\leq 0}x_{0}^{k}+\sum
_{k>0}x_{1}^{k}-\sum _{k\leq 0}\left (V_{k}\right
)_{X_{0}}^{k}+\sum _{k>0}\left (V_{k}\right )_{X_{1}}^{k}$ belongs
to $(X_{0},X_{1})_{\theta ,q} $ and
\begin{equation*}
\left \Vert \bar {x}\right \Vert _{(X_{0},X_{1})_{\theta ,q}}\leq \gamma
\left \Vert y\right \Vert _{\theta ,q}
\end{equation*}
with $\gamma >0$ independent of $y\in (Y_{0},Y_{1})_{\theta ,q}$.
\end{lemma}

\begin{proof} From the definition of the elements $\left( V_{k}\right)
_{X_{0}}^{k},\left( V_{k}\right) _{X_{1}}^{k}$ we have for any $n\in {%
\mathbb{Z}}$
\begin{equation*}
\bar{x}=\sum_{k\leq n}x_{0}^{k}+\sum_{k>n}x_{1}^{k}-\sum_{k\leq n}\left(
V_{k}\right) _{X_{0}}^{k}+\sum_{k>n}\left( V_{k}\right) _{X_{1}}^{k}{\text{.}%
}
\end{equation*}%
We have the following estimates
\begin{align*}
K(2^{n},\bar{x};\vec{X})& \leq \sum_{k\leq n}\big \|x_{0}^{k}\big \|%
_{X_{0}}+2^{n}\sum_{k>n}\big \|x_{1}^{k}\Vert _{X_{1}}+\sum_{k\leq n}\big \|%
\left( V_{k}\right) _{X_{0}}^{k}\Vert _{X_{0}}+2^{n}\sum_{k>n}\Vert \left(
V_{k}\right) _{X_{1}}^{k}\Vert _{X_{1}} \\
& \leq \gamma \bigg (\sum_{k\leq n}J(2^{k},y_{k};\vec{Y})+\sum_{k>n}\frac{%
2^{n}}{2^{k}}J(2^{k},y_{k};\vec{Y})\bigg )=\gamma S_{d}\big (\{J(2^{k},y_{k};%
\vec{Y})\}\big )(2^{n}),
\end{align*}%
where $S_{d}$ is the discrete analog of the Calder\'{o}n operator. Since $%
S_{d}$ is a bounded operator in the sequence space $\ell
_{q}(\left\{ 2^{-n\theta }\right\} )$ for all $\theta \in (0,1)$
and $1\leq q\leq \infty $, $\bar{x}\in (X_{0},X_{1})_{\theta ,q}$
and
\begin{equation*}
\left\Vert \bar{x}\right\Vert _{(X_{0},X_{1})_{\theta ,q}}\leq \gamma
\left\Vert y\right\Vert _{\theta ,q}{\text{.}}
\end{equation*}%
\end{proof}


It follows from Lemma \ref{Lemma6} that in order to prove that the element
\begin{equation*}
\tilde{x}=\sum_{k\leq 0}x_{0}^{k}+\sum_{k>0}x_{1}^{k}-\sum_{k\leq 0}\left(
V_{k}\right) _{X_{0}}^{0}+\sum_{k>0}\left( V_{k}\right) _{X_{1}}^{0}
\end{equation*}%
is in $(X_{0},X_{1})_{\theta, q}$, it is enough to show that
\begin{align}
w& :=\sum_{k\leq 0}\left( V_{k}\right) _{X_{0}}^{k}-\sum_{k>0}\left(
V_{k}\right) _{X_{1}}^{k}-\sum_{k\leq 0}\left( V_{k}\right)
_{X_{0}}^{0}+\sum_{k>0}\left( V_{k}\right) _{X_{1}}^{0}  \label{T5.4} \\
& =\sum_{k<0}\big (\left( V_{k}\right) _{X_{0}}^{k}-\left( V_{k}\right)
_{X_{0}}^{0}\big )-\sum_{k>0}\big (\left( V_{k}\right) _{X_{1}}^{k}-\left(
V_{k}\right) _{X_{1}}^{0}\big )  \notag
\end{align}%
belongs to $(X_{0},X_{1})_{\theta ,q}$. To show this, we consider
elements
\begin{equation}
(u_{k})_{j}=\left( V_{k}\right) _{X_{0}}^{j+1}-\left( V_{k}\right)
_{X_{0}}^{j}=-\big (\left( V_{k}\right) _{X_{1}}^{j+1}-\left( V_{k}\right)
_{X_{1}}^{j}\big )\in X_{0}\cap X_{1}.  \label{NM11}
\end{equation}%
We observe that
\begin{equation*}
J(2^{j},(u_{k})_{j};\vec{X})\leq \gamma K(2^{j},V_{k};\vec{X}){\text{.}}
\end{equation*}%
Furthermore, if $k<0$ then we have
\begin{equation*}
\left( V_{k}\right) _{X_{0}}^{0}-\left( V_{k}\right)
_{X_{0}}^{k}=\sum_{j:k\leq j<0}(u_{k})_{j}.
\end{equation*}%
and for each $k>0$,
\begin{equation*}
\left( V_{k}\right) _{X_{1}}^{k}-\left( V_{k}\right)
_{X_{1}}^{0}=-\sum_{j:0\leq j<k}(u_{k})_{j}{\text{.}}
\end{equation*}%
So formally for the element $w$ defined in (\ref{T5.4}) we have
\begin{equation}
w=-\sum_{k<0}\sum_{j:k\leq j<0}(u_{k})_{j}+\sum_{k>0}\sum_{j:0\leq
j<k}(u_{k})_{j}{\text{.}}  \label{K20}
\end{equation}%
To work with this series we need the following proposition.


\begin{proposition}
\label{Proposition6} The series%
\begin{equation*}
\sum_{k<0}\sum_{j:k\leq j<0}(u_{k})_{j}, \quad\,\,\,
\sum_{k>0}\sum_{j:0\leq j<k}(u_{k})_{j}\text{\ }
\end{equation*}%
converge absolutely in $X_{0}$ and $X_{1}$, respectively.
\end{proposition}


\begin{proof} We start with the series $\sum_{k<0}\sum_{j:k\leq
j<0}(u_{k})_{j}$. Our hypothesis $\theta >\beta _{0}(\ker _{X_{0}+X_{1}}A)$
implies that there exists $\varepsilon >0$ such that for all $V\in \ker
_{X_{0}+X_{1}}A$,
\begin{equation*}
\frac{K(s,V;\vec{X})}{K(t,V;\vec{X})}\geq \gamma \left( \frac{s}{t}\right)
^{\theta -\varepsilon },\quad \,0<s<t\leq 1.
\end{equation*}%
Combining with estimate (\ref{T5.3}) and H\"older's inequality we
obtain
\begin{align*}
\sum_{k<0}\sum_{j:k\leq j<0}\left\Vert (u_{k})_{j}\right\Vert _{X_{0}}& \leq
\gamma \sum_{k<0}\sum_{j:k\leq j<0}K(2^{j},V_{k};\vec{X})\leq \gamma
\sum_{k<0}\sum_{j:k\leq j<0}K(2^{k},V_{k};\vec{X})\left( \frac{2^{j}}{2^{k}}%
\right) ^{\theta -\varepsilon } \\
& \leq \gamma \sum_{k<0}\sum_{j:k\leq j<0}J(2^{k},y_{k};\vec{Y})\left( \frac{%
2^{j}}{2^{k}}\right) ^{\theta -\varepsilon }\leq \gamma
\sum_{k<0}2^{k\varepsilon }2^{-\theta k}J(2^{k},y_{k};\vec{Y}) \\
& \leq \gamma \left( \sum_{k<0}\left( 2^{-\theta k}J(2^{k},y_{k};\vec{Y}%
\right) ^{q}\right) ^{1/q}\leq \gamma \left\Vert y\right\Vert _{\theta ,q}.
\end{align*}%
Similarly, it follows by $\theta <\alpha _{\infty }(\ker _{X_{0}+X_{1}}A)$,
that there exists $\varepsilon >0$ such that for all $V\in \ker
_{X_{0}+X_{1}}A$,
\begin{equation*}
\frac{K(s,V;\vec{X})}{K(t,V;\vec{X})}\leq \gamma \left( \frac{s}{t}\right)
^{\theta +\varepsilon },\quad \,1\leq s<t.
\end{equation*}%
Thus applying (\ref{T5.3}) and the H\"{o}lder inequality yields%
\begin{align*}
\sum_{k>0}& \sum_{j:0\leq j<k}\left\Vert (u_{k})_{j}\right\Vert _{X_{1}}\leq
\gamma \sum_{k>0}\sum_{j:0\leq j<k}\frac{1}{2^{j}}K(2^{j},V_{k};\vec{X}) \\
& \leq \gamma \sum_{k>0}\sum_{j:0\leq j<k}\frac{1}{2^{j}}K(2^{k},V_{k};\vec{X%
})\left( \frac{2^{j}}{2^{k}}\right) ^{\theta +\varepsilon }\leq \gamma
\sum_{k>0}\sum_{j:0\leq j<k}\frac{1}{2^{j}}J(2^{k},y_{k};\vec{Y})\left(
\frac{2^{j}}{2^{k}}\right) ^{\theta +\varepsilon } \\
& \leq \gamma \sum_{k>0}\frac{1}{2^{k\varepsilon }}2^{-\theta
k}J(2^{k},y_{k};\vec{Y})\leq \bigg (\sum_{k>0}\Big (2^{-\theta
k}J(2^{k},y_{k};\vec{Y}\Big )^{q}\bigg )^{1/q}\leq \gamma \left\Vert
y\right\Vert _{\theta ,q}
\end{align*}%
and this completes the proof.
\end{proof}

We have shown that the $w$ defined by formula (\ref{T5.4}) can be
written in the form (\ref{K20}). Now we are ready to prove that
$w\in (X_{0},X_{1})_{\theta ,q}$.


\begin{lemma}
The element $w=-\sum_{k<0}\sum_{j:k\leq j<0}(u_{k})_{j}+\sum_{k>0}\sum_{j:0\leq
j<k}(u_{k})_{j}$ belongs to $\left( X_{0},X_{1}\right) _{\theta ,q}$ and
\begin{equation*}
\left\Vert w\right\Vert _{(X_{0},X_{1})_{\theta ,q}}\leq \gamma \left\Vert
y\right\Vert _{\theta ,q}.
\end{equation*}
\end{lemma}


\begin{proof} Since the series $-\sum_{k<0}\sum_{j:k\leq j<0}(u_{k})_{j}$ $%
+\sum_{k>0}\sum_{j:0\leq j<k}(u_{k})_{j}$ converges absolutely in $%
X_{0}+X_{1}$, $w$ can be represented as follows
\begin{equation*}
w=-\sum_{j<0}\sum_{k:k\leq j}(u_{k})_{j}+\sum_{j\geq
0}\sum_{k:k>j}(u_{k})_{j}{\text{.}}
\end{equation*}%
As in the proof of Proposition \ref{Proposition6}, we have for each $j<0$,
\begin{align*}
J\Big (2^{j},\sum_{k:k\leq j}(u_{k})_{j};\vec{X}\Big )& \leq \sum_{k:k\leq
j}\left\Vert (u_{k})_{j}\right\Vert _{X_{0}}+2^{j}\sum_{k:k\leq j}\left\Vert
(u_{k})_{j}\right\Vert _{X_{1}} \\
& \leq \gamma \sum_{k:k\leq j}K(2^{j},V_{k};\vec{X})\leq \gamma
\sum_{k:k\leq j}K(2^{k},V_{k};\vec{X})\left( \frac{2^{j}}{2^{k}}\right)
^{\theta -\varepsilon },
\end{align*}%
and for $j\geq 0$,
\begin{align*}
J\Big (2^{j},\sum_{k:k>j}(u_{k})_{j};\vec{X}\Big )& \leq
\sum_{k:k>j}\left\Vert (u_{k})_{j}\right\Vert
_{X_{0}}+2^{j}\sum_{k:k>j}\left\Vert (u_{k})_{j}\right\Vert _{X_{1}} \\
& \leq \gamma \sum_{k:k>j}K(2^{j},V_{k};\vec{X})\leq \gamma
\sum_{k:k>j}K(2^{k},V_{k};\vec{X})\left( \frac{2^{j}}{2^{k}}\right) ^{\theta
+\varepsilon }.
\end{align*}%
In consequence,
\begin{align}
\left\Vert w\right\Vert _{(X_{0},X_{1})_{\theta ,q}}& \leq \bigg (\sum_{j<0}%
\Big (2^{-\theta j}J(2^{j},\sum_{k:k\leq j}(u_{k})_{j};\vec{X}\Big )^{q}%
\bigg )^{1/q}  \label{NM12} \\
& +\bigg (\sum_{j\geq 0}\Big (2^{-\theta j}J(2^{j},\sum_{k:k>j}(u_{k})_{j};%
\vec{X})\Big )^{q}\bigg )^{1/q}.  \label{NM13}
\end{align}%
The first term on the right can be estimated by using the inequalities shown
above and the generalized Minkowski inequality,
\begin{align*}
\bigg (\sum_{j<0}& \Big (2^{-\theta j}J\Big (2^{j},\sum_{k:k\leq
j}(u_{k})_{j};\vec{X}\Big )\Big )^{q}\bigg )^{1/q}\leq \gamma \bigg (%
\sum_{j<0}\Big (2^{-\theta j}\sum_{k:k\leq j}K(2^{k},V_{k};\vec{X})\Big (%
\frac{2^{j}}{2^{k}}\Big )^{\theta -\varepsilon }\Big )^{q}\bigg )^{1/q} \\
& =\gamma \bigg (\sum_{j<0}\Big (\sum_{k:k\leq j}2^{-k\theta }K(2^{k},V_{k};%
\vec{X})\Big (\frac{2^{j}}{2^{k}}\Big )^{-\varepsilon }\Big )^{q}\bigg )%
^{1/q}\leq \gamma \bigg (\sum_{k<0}\Big (2^{-k\theta }K(2^{k},V_{k};\vec{X})%
\Big )^{q}\bigg )^{1/q} \\
& \leq \gamma \bigg (\sum_{k<0}\Big (2^{-k\theta }J(2^{k},y_{k};\vec{Y})\Big
)^{q}\bigg )^{1/q}\leq \gamma \left\Vert y\right\Vert _{\theta ,q}{\text{.}}
\end{align*}%
The second term can be estimated similarly:
\begin{align*}
\bigg (\sum_{j\geq 0}& \Big (2^{-\theta j}J\Big (2^{j},%
\sum_{k:k>j}(u_{k})_{j};\vec{X}\Big )\Big )^{q}\bigg )^{1/q}\leq \gamma %
\bigg (\sum_{j\geq 0}\Big (2^{-\theta j}\sum_{k:k>j}K(2^{k},V_{k};\vec{X})%
\Big (\frac{2^{j}}{2^{k}}\Big )^{\theta +\varepsilon }\Big )^{q}\bigg )^{1/q}
\\
& =\gamma \bigg (\sum_{j\geq 0}\Big (\sum_{k:k>j}2^{-k\theta }K(2^{k},V_{k};%
\vec{X})\Big (\frac{2^{j}}{2^{k}}\Big )^{\varepsilon }\Big )^{q}\bigg )%
^{1/q}\leq \gamma \bigg (\sum_{k>0}\Big (2^{-k\theta }K(2^{k},V_{k};\vec{X})%
\Big )^{q}\bigg )^{1/q} \\
& \leq \gamma \bigg (\sum_{k>0}\Big (2^{-k\theta }J(2^{k},y_{k};\vec{Y})\Big
)^{q}\bigg )^{1/q}\leq \gamma \left\Vert y\right\Vert _{\theta ,q}
\end{align*}%
and this completes the proof of lemma.
\end{proof}

To conclude we note that we have proved that any element $y\in
(Y_{0},Y_{1})_{\theta ,q}$ can be obtained as an image of an
element $x$ which is equal to
\begin{equation*}
x=\tilde{x}+u+v,
\end{equation*}%
where $\tilde{x}\in \left( X_{0},X_{1}\right) _{\theta ,q}$
satisfies estimates
\begin{equation*}
\left\Vert \tilde{x}\right\Vert _{(X_{0},X_{1})_{\theta ,q}}\leq \gamma
\left\Vert y\right\Vert _{\theta ,q}{\text{, }}
\end{equation*}%
with constant $\gamma >0$ independent of $y$, and elements $u\in \func{Pr}%
_{X_{0}}(\ker _{X_{0}+X_{1}}A),v\in \func{Pr}_{X_{1}}(\ker
_{X_{0}+X_{1}}A)$. Thus
\begin{equation*}
y=Ax=A(\tilde{x})+A(u+v)\in A((X_{0},X_{1})_{\theta ,q})+M,
\end{equation*}%
where $M=A(\func{Pr}_{X_{0}}(\ker _{X_{0}+X_{1}}A))=A(\func{Pr}_{X_{1}}(\ker
_{X_{0}+X_{1}}A))$. We have also proved that $M$ does not intersect with
$A((X_{0},X_{1})_{\theta ,q})$, i.e.,
\begin{equation*}
\left( Y_{0},Y_{1}\right) _{\theta ,q}=A((X_{0},X_{1})_{\theta ,q})\oplus M.
\end{equation*}%
and ${\text{dim}}\,M={\text{ker}}_{X_{0}+X_{1}}\,A$. Since $M$ is finite
dimensional, $A((X_{0},X_{1})_{\theta ,q})$ is a closed subspace of $\left(
Y_{0},Y_{1}\right) _{\theta ,q}$. This completes the proof of the theorem.


\subsection{Sufficient condition for the class ${\mathbb{F}}_{\protect%
\theta ,q}^{3}$}

The proof of the next theorem is based on Lemma $2$ in \cite{AK1}.


\begin{theorem}
Let $A\colon (X_{0},X_{1})\rightarrow (Y_{0},Y_{1})$ be
a~surjective and Fredholm operator on endpoint spaces. Suppose
that $\ker _{X_{0}+X_{1}}A=V_{\theta ,q}^{0}+V_{\theta ,q}^{1}$
and
\begin{equation*}
\beta (V_{\theta ,q}^{1})<\theta <\alpha \left( V_{\theta ,q}^{0}\right) {%
\text{.}}
\end{equation*}%
\textit{\ } Then the operator $A\colon (X_{0},X_{1})_{\theta ,q}\rightarrow
(Y_{0},Y_{1})_{\theta ,q}$ is invertible, i.e., $A\in {\mathbb{F}}_{\theta
,q}^{3}$.
\end{theorem}

\begin{proof} First observe that for any element $0\neq x\in \ker
_{X_{0}+X_{1}}A$, $\alpha (x)\leq \beta (x)$ and so $V_{\theta
,q}^{0}\cap V_{\theta ,q}^{1}=\left\{ 0\right\} $. Thus $\ker
_{X_{0}+X_{1}}A\cap (X_{0},X_{1})_{\theta ,q}=\left\{ 0\right\} $,
i.e., the operator $A\colon (X_{0},X_{1})_{\theta ,q}\rightarrow
(Y_{0},Y_{1})_{\theta ,q}$ is injective. Since $\ker
_{X_{0}+X_{1}}A=V_{\theta ,q}^{0}\oplus V_{\theta ,q}^{1}$, it
follows from \cite[Lemma 2]{AK1} (note that in the proof of this
Lemma it is enough to use surjectivity on endpoint spaces instead
of invertibility), that the operator $A$ is an isomorphism from
$(X_{0},X_{1})_{\theta ,q}$ onto $(Y_{0},Y_{1})_{\theta ,q}$.
\end{proof}


\section{Proof of Theorem \protect\ref{TN4} (necessity)}

In order to prove Theorem \ref{TN4} we need some preliminary
results.

\subsection{Preliminary results}

In the proof below we use notation and some results from the paper
\cite{AK1}. For a~given $x\in \ker _{X_{0}+X_{1}}A$, we define the
space of all real sequences $\ell _{\theta ,q}(x)$ equipped with
the norm
\begin{equation}
\Big \|\sum_{i\in {\mathbb{Z}}}\eta _{i}e_{i}\Big \|_{\ell _{\theta ,q}(x)}=%
\bigg (\sum_{i\in {\mathbb{Z}}}\Big (|\eta _{i}|\,\frac{%
K(2^{i},x;X_{0},X_{1})}{2^{\theta i}}\Big )^{q}\bigg )^{1/q},  \label{K9}
\end{equation}%
with the standard modification for $q=\infty$. Here $\{e_{i}\}_{i\in {%
\mathbb{Z}}}$ is the standard unit basis sequence in $c_{0}$.~We denote by $%
S$ the operator of right shift in the sequence space and by $I$ the identity
operator, in particular we will often use the operator
\begin{equation*}
(S-I)\Big (\sum_{i\in {\mathbb{Z}}}\eta _{i}e_{i}\Big )=\sum_{i\in {\mathbb{Z%
}}}(\eta _{i-1}-\eta _{i})e_{n}{\text{.}}
\end{equation*}%
The kernel of the operator $S-I$ on the space $\ell _{\theta
,q}(x)$ is spaned by the element
\begin{equation*}
f_{\ast }=\sum_{i\in {\mathbb{Z}}}e_{i}
\end{equation*}%
if this element belongs to the space $\ell _{\theta ,q}(x)$, if this element
does not belong to the space $\ell _{\theta ,q}(x)$ then the operator $S-I$
is injective. In \cite{AK1} it was shown that if the operator $S-I$ is
invertible on the space $\ell _{\theta ,q}(x)$, then its inverse coincides
with one of the following two operators $T_{0}$, $T_{1}$, which are defined
by the formulas
\begin{equation}
T_{0}\Big (\sum_{i\in {\mathbb{Z}}}\eta _{i}e_{i}\Big )=\sum_{k\in {\mathbb{Z%
}}}\Big (\sum_{i>k}\eta _{i}\Big )e_{k}  \label{K10}
\end{equation}%
and
\begin{equation}
T_{1}\Big (\sum_{i\in {\mathbb{Z}}}\eta _{i}e_{i}\Big )=-\sum_{k\in {\mathbb{%
Z}}}\Big (\sum_{i\leq k}\eta _{i}\Big )e_{k}.  \label{K11}
\end{equation}

In what follows we need some more notation. Throughout the paper we
denote by $U$ the space of all finite sequences $\sum_{i\in {\mathbb{Z}}%
}\eta _{i}e_{i}$ and
\begin{equation}
U^{0}=\Big\{u=\sum_{n\in {\mathbb{Z}}}\eta _{n}e_{n}\in U;\, \mathit{\ }%
\sum_{n\in {\mathbb{Z}}}\eta _{n} = 0 \Big\} \text{.} \label{NM18}
\end{equation}
Clearly we have that
\begin{equation*}
U=U^{0}+\mathrm{span}\left\{ e_{0}\right\} {\text{.}}
\end{equation*}

\begin{remark}
\label{Remark3} Note that from formulas
$(\ref{K10})$-$(\ref{K11})$ it follows that $T_{0}=T_{1}$ on
$U^{0}$. Moreover operator $S-I$ maps $U$ onto $U^{0}$ and has on
$U^{0}$ inverse equal to $T_{0}=T_{1}$.
\end{remark}

We finish this section with a~lemma that characterizes some
properties of the operator $A$ on the space $(X_{0},X_{1})_{\theta
,q}$ in terms of the operator $S-I$ on the space $\ell _{\theta
,q}(x)$.


\begin{lemma}
\label{Lemma1} Let $A\colon (X_0,X_1)\to (Y_0,Y_1)$ be an operator between
Banach couples.

\begin{itemize}
\item[{\rm{a)}}] The operator $A\colon (X_{0},X_{1})_{\theta
,q}\rightarrow (Y_{0},Y_{1})_{\theta ,q}$ is injective if and only
if the
operator $S-I$ is injective on $\ell _{\theta ,q}(x)$ for all $x\in \mathrm{%
ker}_{X_{0}+X_{1}}A\setminus \left\{ 0\right\} $.

\item[{\rm{b)}}] Suppose that the operator $A$ is injective on $%
(X_{0},X_{1})_{\theta ,q}$ and $x\in \mathrm{ker}_{X_{0}+X_{1}}A\setminus
\left\{ 0\right\} $. Then the operator $S-I$ is invertible on the space $%
\ell _{\theta ,q}(x)$ if and only if $\theta \notin \left[\alpha
(x),\beta (x)\right]$.
\end{itemize}
\end{lemma}


\begin{proof} a) If the operator $S-I$ is not injective on $\ell _{\theta
,q}(x)$ for some $x\in \ker _{X_{0}+X_{1}}A\setminus \left\{ 0\right\} $,
then $f_{\ast }=\sum_{i\in {\mathbb{Z}}}e_{i}\in \ell _{\theta ,q}(x)$,
i.e.,
\begin{equation*}
\Vert f_{\ast }\Vert _{\ell _{\theta ,q}(x)}=\bigg (\sum_{i\in {\mathbb{Z}}}%
\Big (\frac{K(2^{i},x;\vec{X})}{2^{\theta i}}\Big )^{q}\bigg )^{1/q}<\infty .
\end{equation*}%
Thus the element $x$ belongs to $(X_{0},X_{1})_{\theta ,q}$. Since
$x\in \ker _{X_{0}+X_{1}}A\setminus \left\{ 0\right\} $, the
operator $A$ is not injective on $(X_{0},X_{1})_{\theta ,q}$.
Moreover, if the operator $A$ is not injective on
$(X_{0},X_{1})_{\theta ,q}$ then there exists a~non-zero $x\in
\ker _{X_{0}+X_{1}}A\cap (X_{0},X_{1})_{\theta ,q}$. It means that
$\left\Vert \sum_{i\in {\mathbb{Z}}}e_{i}\right\Vert_{\ell_{\theta
,q}(x)}<\infty $ and so the operator $S-I$ is not injective on
$\ell _{\theta ,q}(x)$.

b) Suppose that the operator $S-I$ is invertible on $\ell _{\theta ,q}(x)$
for the element $x\in \ker _{X_{0}+X_{1}}A\setminus \left\{ 0\right\} $. We
need to prove that $\theta \notin \left[ \alpha (x),\beta (x)\right] $.
First note that from Lemma 3 in \cite{AK1}, we have that the inverse to the
operator $S-I$ is the operator $T_{0}$ or the operator $T_{1}$.~If its
inverse is equal to $T_{0}$ then from Lemma 5 in \cite{AK1}, it follows that
the operator $T_{0}$ is also inverse to $S-I$ on the space $\ell _{\theta
\pm \varepsilon ,q}(x)$ for $\varepsilon >0$ small enough. Thus
\begin{equation*}
\left\Vert T_{0}e_{k}\right\Vert _{\ell _{\theta +\varepsilon ,q}(x)}\leq
\left\Vert T_{0}\right\Vert _{\ell _{\theta +\varepsilon ,q}(x)\rightarrow
l_{\theta +\varepsilon ,q}(x)}\left\Vert e_{k}\right\Vert _{\ell _{\theta
+\varepsilon ,q}(x)}
\end{equation*}%
and so for $i<k$ we get (see formula (\ref{K10}))
\begin{equation*}
\frac{K(2^{i},x;\vec{X})}{2^{(\theta +\varepsilon )i}}\leq \left\Vert
T_{0}e_{k}\right\Vert _{\ell _{\theta +\varepsilon ,q}(x)}\leq \left\Vert
T_{0}\right\Vert _{\ell _{\theta +\varepsilon ,q}(x)\rightarrow \ell
_{\theta +\varepsilon ,q}(x)}\frac{K(2^{k},x;\vec{X})}{2^{(\theta
+\varepsilon )k}}.
\end{equation*}%
This gives
\begin{equation*}
\frac{K(s,x;\vec{X})}{K(t,x;\vec{X})}\leq \gamma \left( \frac{s}{t}\right)
^{\theta +\varepsilon },\quad \,0<s<t
\end{equation*}%
and hence $\theta <\alpha (x)$. Similarly we show that if the
inverse $S-I$ is the operator $T_{1}$, then $\theta >\beta (x)$.

To prove the "if" part suppose that $\theta \notin \left[ \alpha
(x),\beta (x)\right] $. We need to show that the operator $S-I$ is
invertible on the space $\ell _{\theta ,q}(x)$. First consider the
case when $\theta <\alpha (x)$. Then for all $0<s<t$ and small
enough $\varepsilon >0$ we will have
\begin{equation*}
\frac{K(s,x;\vec{X})}{K(t,x;\vec{X})}\leq \gamma \left( \frac{s}{t}\right)
^{\theta +\varepsilon }{\text{.}}
\end{equation*}%
If $\sum_{i\in {\mathbb{Z}}}\eta _{i}e_{i}\in \ell _{\theta ,\infty }(x)$
then
\begin{equation*}
|\eta _{i}|\frac{K(2^{i},x;\vec{X})}{2^{\theta i}}\leq \Big \|\sum_{i\in {%
\mathbb{Z}}}\eta _{i}e_{i}\Big \|_{\ell _{\theta ,\infty }(x)},
\end{equation*}%
and whence the operator $T_{0}$ is bounded on $\ell _{\tilde{\theta},\infty
}(x)$ for $\tilde{\theta}<\theta +\varepsilon $. Indeed
\begin{align*}
\Big \|T_{0}\Big (\sum_{i\in {\mathbb{Z}}}\eta _{i}e_{i}\Big )\Big \|_{\ell
_{\tilde{\theta},\infty }(x)}& \leq \sup_{k\in {\mathbb{Z}}}\Big (%
\sum_{i>k}\left\vert \eta _{i}\right\vert \Big )\frac{K(2^{k},x;\vec{X})}{2^{%
\tilde{\theta}k}} \\
& \leq \sup_{k\in {\mathbb{Z}}}\Big (\sum_{i>k}\frac{2^{\tilde{\theta}i}}{%
K(2^{i},x;\vec{X})}\Big )\frac{K(2^{k},x;\vec{X})}{2^{\tilde{\theta}k}}\cdot %
\Big \|\sum_{i\in {\mathbb{Z}}}\eta _{i}e_{i}\Big \|_{\ell _{\tilde{\theta}%
,\infty }(x)} \\
& \leq \gamma \sup_{k\in {\mathbb{Z}}}\,\sum_{i>k}\frac{2^{(\theta
+\varepsilon -\tilde{\theta})k}}{2^{(\theta +\varepsilon -\tilde{\theta})i}}%
\cdot \Big \Vert\sum_{i\in {\mathbb{Z}}}\eta _{i}e_{i}\Big \|_{\ell _{\tilde{%
\theta},\infty }(x)}\leq \gamma \Big \|\sum_{i\in {\mathbb{Z}}}\eta _{i}e_{i}%
\Big \|_{\ell _{\tilde{\theta},\infty }(x)}.
\end{align*}%
Since
\begin{equation*}
\ell _{\theta ,q}(x)=\big (\ell _{\theta -\frac{\varepsilon }{2},\infty
}(x),\ell _{\theta +\frac{\varepsilon }{2},\infty }(x)\big )_{1/2,q},
\end{equation*}%
it follows by interpolation that $T_{0}$ is bounded on $\ell
_{\theta ,q}(x)$. Since injectivity of $A$ on
$(X_{0},X_{1})_{\theta ,q}$ implies injectivity of $S-I$ on $\ell
_{\theta ,q}(x)$ (which was shown in a)), it
follows from Lemma 4 in \cite{AK1} that $T_{0}$ is inverse to the operator $%
S-I$ on $\ell _{\theta ,q}(x)$. Similarly, from injectivity of $A$ on $%
(X_{0},X_{1})_{\theta ,q}$ and $\theta >\beta (x)$ we obtain that
the operator $T_{1}$ is inverse to $S-I$ on $\ell _{\theta
,q}(x)$.
\end{proof}

\bigskip

\subsection{Necessity condition for the class ${\mathbb{F}}_{\protect\theta %
,q}^{1}$}

In this subsection we prove the following result.

\begin{theorem}
\label{TTT1}Suppose that \ operator $A\colon
(X_{0},X_{1})\rightarrow
(Y_{0},Y_{1})$ is invertible on endpoint spaces and belongs to the class ${%
\mathbb{F}}_{\theta ,q}^{1}$, i.e., $\ker _{X_{0}+X_{1}}A$
$\subset (X_{0},X_{1})_{\theta ,q}$ and is finite dimensional, and
$A((X_{0},X_{1})_{\theta ,q})=(Y_{0},Y_{1})_{\theta ,q}$. Then
\begin{equation*}
\beta _{\infty }(\ker _{X_{0}+X_{1}}A)<\theta
<\alpha_{0}(\ker_{X_{0}+X_{1}}A).
\end{equation*}
\end{theorem}

\begin{proof} We give below a proof for $\theta <\alpha _{0}(\ker
_{X_{0}+X_{1}}A)$ (the inequality $\beta _{\infty }(\ker
_{X_{0}+X_{1}}A)<\theta $ can be proved similarly by using
operator $T_{1}$ instead of the operator $T_{0}$). We first prove
that for every $x\in \ker _{X_{0}+X_{1}}A$ $\ $we have
\begin{equation}
\Big \|T_{0}\Big (\sum_{k\leq 0}\eta _{k}e_{k}\Big )\Big \|_{\ell
_{\theta ,q}(x)}\leq \gamma \Big \|\sum_{k\leq 0}\eta
_{k}e_{k}\Big \|_{\ell _{\theta ,q}(x)}  \label{K21}
\end{equation}%
with constant $\gamma >0$ independent of $x\in \ker _{X_{0}+X_{1}}A$. Let $%
x\in \ker A$ and consider decompositions
$x=x_{0}^{k}+x_{1}^{k{\text{ }}}$ such that
\begin{equation*}
\big \|x_{0}^{k}\big \|_{X_{0}}+2^{k}\big \|x_{1}^{k}\big
\|_{X_{1}}\leq \gamma K(2^{k},x;\vec{X}){\text{.}}
\end{equation*}%
Then
\begin{equation*}
y_{k}=Ax_{0}^{k}=-Ax_{1}^{k}\in Y_{0}\cap Y_{1},
\end{equation*}%
and from invertibility of $A$ on endpoint spaces we obtain%
\begin{equation}
J(2^{k},y_{k};Y_{0},Y_{1})\leq \gamma K(2^{k},x;\vec{X}){\text{.}}
\label{K100}
\end{equation}%
Let $\sum_{k\leq 0}\eta _{k}e_{k}$ be a finite sequence and
consider the element
\begin{equation*}
y=\sum_{k\leq 0}\eta _{k}y_{k}\in Y_{0}\cap Y_{1}{\text{.}}
\end{equation*}%
Note that the operator $A\colon (X_{0},X_{1})_{\theta ,q}\cap
X_{0}\rightarrow (Y_{0},Y_{1})_{\theta ,q}\cap Y_{0}$ is
invertible.~Indeed if $w\in (Y_{0},Y_{1})_{\theta ,q}\cap Y_{0}$,
then there exists $u\in
(X_{0},X_{1})_{\theta ,q}$ and $v\in X_{0}$ such that $Au=Av=w$. Thus $%
u-v\in \ker A\subset (X_{0},X_{1})_{\theta ,q}$ and so $v\in
(X_{0},X_{1})_{\theta ,q}$. Since $A\colon (X_{0},X_{1})_{\theta
,q}\cap X_{0}\rightarrow (Y_{0},Y_{1})_{\theta ,q}\cap Y_{0}$ is a
surjective operator and $A$ is injective on $X_{0}$ therefore
$A\colon (X_{0},X_{1})_{\theta ,q}\cap X_{0}\rightarrow
(Y_{0},Y_{1})_{\theta ,q}\cap Y_{0}$ is invertible. Combining with
formulas (see Propositions 2 and 4 in \cite{AK1})
\begin{equation*}
(X_{0},X_{1})_{\theta ,q}\cap X_{0}=(X_{0},X_{0}\cap
X_{1})_{\theta ,q}\quad \,{\text{and}}\quad
\,(Y_{0},Y_{1})_{\theta ,q}\cap Y_{0}=(Y_{0},Y_{0}\cap
Y_{1})_{\theta ,q},
\end{equation*}%
we conclude that $A\colon (X_{0},X_{0}\cap X_{1})_{\theta
,q}\rightarrow (Y_{0},Y_{0}\cap Y_{1})_{\theta ,q}$ is invertible.
Since $y=\sum_{k\leq 0}\eta _{k}y_{k}\in Y_{0}\cap Y_{1}\subset
(Y_{0},Y_{0}\cap Y_{1})_{\theta ,q}$, there exists $u\in
(X_{0},X_{0}\cap X_{1})_{\theta ,q}$ such that $Au=y$ and
\begin{equation*}
\left\Vert u\right\Vert _{(X_{0},X_{0}\cap X_{1})_{\theta ,q}}\leq
\gamma \left\Vert y\right\Vert _{(Y_{0},Y_{0}\cap Y_{1})_{\theta
,q}}
\end{equation*}%
and so we can find a~sequence $\{u_{k}\}$ in $X_{0}\cap X_{1}$ with $%
u=\sum_{k\in {\mathbb{Z}}}u_{k}$ (convergent in $X_{0}+X_{0}\cap
X_{1}=X_{0} $) and
\begin{align*}
\bigg (& \sum_{k\in {\mathbb{Z}}}\big (2^{-\theta
k}J(2^{k},u_{k};X_{0},X_{0}\cap X_{1})\big )^{q}\bigg )^{1/q}\leq
\gamma \left\Vert u\right\Vert _{(X_{0},X_{0}\cap X_{1})_{\theta
,q}}\leq \gamma
\left\Vert y\right\Vert _{(Y_{0},Y_{0}\cap Y_{1})_{\theta ,q}} \\
& \leq \bigg (\sum_{k\leq 0}\big (2^{-\theta k}J(2^{k},\eta
_{k}y_{k};Y_{0},Y_{0}\cap Y_{1})\big )^{q}\bigg )^{1/q}\leq \gamma \bigg (%
\sum_{k\leq 0}\big (|\eta _{k}|2^{-\theta k}J(2^{k},y_{k};\vec{Y})\big )^{q}%
\bigg )^{1/q} \\
& \leq \gamma \bigg (\sum_{k\leq 0}\big (|\eta _{k}|2^{-\theta k}K(2^{k},x;%
\vec{X})\big )^{q}\bigg )^{1/q}=\gamma \,\Big \|\sum_{k\leq 0}\eta _{k}e_{k}%
\Big \|_{\ell _{\theta ,q}(x)}.
\end{align*}%
Note that in the second line it is important that $k\leq 0$.

Since $y=\sum_{k\in \mathbb{Z}}Au_{k}=\sum_{k\leq 0}\eta _{k}y_{k}$ and $%
u_{k}\in X_{0}\cap X_{1}$
\begin{equation*}
\sum_{k\in \mathbb{Z}}(\eta _{k}y_{k}-Au_{k})=0.
\end{equation*}%
If we set
\begin{equation*}
z_{k}=\sum_{m>k}(\eta _{m}y_{m}-Au_{m})=-\sum_{m\leq k}(\eta
_{m}y_{m}-Au_{m}),
\end{equation*}%
then
\begin{align*}
\left\Vert z_{k}\right\Vert _{Y_{0}}& =\Big \|-\sum_{m\leq k}(\eta
_{m}y_{m}-Au_{m})\Big \|_{Y_{0}}\leq \sum_{m\leq k}\left\vert \eta
_{m}\right\vert \left\Vert y_{m}\right\Vert _{Y_{0}}+\sum_{m\leq
k}\left\Vert Au_{m}\right\Vert _{Y_{0}}\leq \gamma \sum_{m\leq
k}\left\vert
\eta _{m}\right\vert \left\Vert x_{0}^{m}\right\Vert _{X_{0}} \\
& +\gamma \sum_{m\leq k}\left\Vert u_{m}\right\Vert _{X_{0}}\leq
\gamma \sum_{m\leq k}\left\vert \eta _{m}\right\vert
K(2^{m},x;\vec{X})+\gamma \sum_{m\leq k}J(2^{m},u_{m};\vec{X}),
\end{align*}%
and similarly
\begin{align*}
2^{k}\left\Vert z_{k}\right\Vert _{Y_{1}}& =2^{k}\Big
\|\sum_{m>k}(\eta _{m}y_{m}-Au_{m})\Big \|_{Y_{1}}\leq
2^{k}\sum_{m>k}\left\vert \eta _{m}\right\vert \left\Vert
y_{m}\right\Vert
_{Y_{1}}+2^{k}\sum_{m>k}\left\Vert Au_{m}\right\Vert _{Y_{1}} \\
& \leq \gamma 2^{k}\sum_{m>k}\left\vert \eta _{m}\right\vert
\left\Vert x_{1}^{m}\right\Vert _{X_{1}}+\gamma
2^{k}\sum_{m>k}\left\Vert u_{m}\right\Vert _{X_{1}}\leq \gamma
\sum_{m>k}\left\vert \eta
_{m}\right\vert \frac{2^{k}}{2^{m}}K(2^{m},x;\vec{X}) \\
& +\gamma
\sum_{m>k}\frac{2^{k}}{2^{m}}J(2^{m},u_{m};\vec{X}){\text{.}}
\end{align*}%
Combining the above estimates we obtain (by using boundedness of
$S_{d}$-the discrete analog of Calder\'{o}n operator)
\begin{align*}
\bigg (\sum_{k\in {\mathbb{Z}}}\big (2^{-\theta k}J(2^{k},& z_{k};\vec{Y})%
\big )^{q}\bigg )^{1/q}\leq \gamma \bigg (\sum_{k\in {\mathbb{Z}}}\big (%
2^{-\theta k}\left( S_{d}\left[ \{|\eta _{m}|K(2^{m},x;\vec{X}%
)+J(2^{m},u_{m};\vec{X})\}\right] \right) (2^{k})\big )^{q}\bigg )^{1/q} \\
& \leq \gamma \bigg (\sum_{k\in {\mathbb{Z}}}\big (2^{-\theta
k}|\eta
_{k}|K(2^{k},x;\vec{X})\big )^{q}\bigg )^{1/q}+\gamma \bigg (\sum_{k\in {%
\mathbb{Z}}}\big (2^{-\theta k}J(2^{k},u_{k};\vec{X})(2^{k})\big )^{q}\bigg )%
^{1/q} \\
& \leq \gamma \bigg (\sum_{k\in {\mathbb{Z}}}\big (2^{-\theta
k}|\eta _{k}|K(2^{k},x;\vec{X})\big )^{q}\bigg )^{1/q}=\gamma \Big
\|\sum_{k\leq 0}\eta _{k}e_{k}\Big \|_{\ell _{\theta ,q}(x)}.
\end{align*}%
Thus to prove (\ref{K21}) we need to show
\begin{equation*}
\Big \|T_{0}\Big (\sum_{k\leq 0}\eta _{k}e_{k}\Big )\Big \|_{\ell
_{\theta ,q}(x)}\leq \gamma \bigg (\sum_{k\in {\mathbb{Z}}}\big
(2^{-\theta k}J(2^{k},z_{k};\vec{Y})\Big )^{q}\bigg
)^{1/q}{\text{.}}
\end{equation*}%
Since $y_{m}=-Ax_{1}^{m}$,
\begin{equation*}
\left\Vert z_{k}\right\Vert _{Y_{1}}=\Big \|\sum_{m>k}(\eta _{m}y_{m}-Au_{m})%
\Big \|_{Y_{1}}\geq \gamma \Big \|\sum_{m>k}(\eta _{m}x_{1}^{m}+u_{m})\Big \|%
_{X_{1}}
\end{equation*}%
and from $y_{m}=Ax_{0}^{m}$,
\begin{align*}
\left\Vert z_{k}\right\Vert _{Y_{0}}& =\Big \|\sum_{m>k}(\eta
_{m}y_{m}-Au_{m})\Big \|_{Y_{0}}=\Big \|\sum_{m\leq k}(\eta _{m}y_{m}-Au_{m})%
\Big \|_{Y_{0}} \\
& \geq \gamma \Big \|\sum_{m\leq k}(\eta _{m}x_{0}^{m}-u_{m})\Big
\|_{X_{0}}.
\end{align*}%
Now note that
\begin{equation*}
u=\sum_{k\leq 0}\eta _{k}x_{0}^{k}.
\end{equation*}%
Indeed, elements $u$ and $\sum_{k\leq 0}\eta _{k}x_{0}^{k}$ both belong to $%
X_{0}$ and so their images under the operator $A$ are equal to
$y$. Since $A$ is injective on $X_{0}$, $u=\sum_{k\leq 0}\eta
_{k}x_{0}^{k}$ and so
\begin{equation*}
\sum_{m\leq k}(\eta _{m}x_{0}^{m}-u_{m})=-\sum_{m>k}(\eta
_{m}x_{0}^{m}-u_{m}).
\end{equation*}%
Thus the above estimate implies $\left\Vert z_{k}\right\Vert
_{Y_{0}}\geq \gamma \left\Vert \sum_{m>k}(\eta
_{m}x_{0}^{m}-u_{m})\right\Vert _{X_{0}}$ and {\color{red}w}hence
\begin{align*}
\left\Vert z_{k}\right\Vert _{Y_{0}}+2^{k}\left\Vert
z_{k}\right\Vert
_{Y_{1}}& \geq \gamma \Big \|\sum_{m>k}(\eta _{m}x_{1}^{m}+u_{m})\Big \|%
_{X_{1}}+\gamma 2^{k}\Big \|\sum_{m>k}(\eta _{m}x_{0}^{m}-u_{m})\Big \|%
_{X_{0}} \\
& \geq \gamma K\big (2^{k},\big (\sum_{m>k}\eta _{m}\big
)x;\vec{X}\big ).
\end{align*}%
Combining with already obtained estimate
\begin{equation*}
\bigg (\sum_{k\in {\mathbb{Z}}}\big (2^{-\theta k}J(2^{k},z_{k};Y_{0},Y_{1})%
\big )^{q}\bigg )^{1/q}\leq \gamma \,\Big \|\sum_{k\leq 0}\eta
_{k}e_{k}\Big \|_{\ell _{\theta ,q}(x)},
\end{equation*}%
we have
\begin{equation*}
\bigg (\sum_{k\in {\mathbb{Z}}}\big (2^{-\theta k}K\big (2^{k},\big (%
\sum_{m>k}\eta _{m})x;\vec{X}\big )\big )^{q}\bigg )^{1/q}\leq \gamma \Big \|%
\sum_{k\leq 0}\eta _{k}e_{k}\Big \|_{\ell _{\theta ,q}(x)},
\end{equation*}%
i.e., (\ref{K21}) holds with constant $\gamma >0$ independent of
$x\in \ker _{X_{0}+X_{1}}A$. Since
\begin{equation*}
T_{0}\Big (\sum_{k\leq 0}\eta _{k}e_{k}\Big )=\sum_{k\in {\mathbb{Z}}}\Big (%
\sum_{m>k}\eta _{m}\Big )e_{k}=\eta _{0}e_{-1}+(\eta _{0}+\eta
_{-1})e_{-2}+(\eta _{0}+\eta _{-1}+\eta _{-2})e_{-3}+...,
\end{equation*}%
the image of $\sum_{k\leq 0}\eta _{k}e_{k}$ under the action of
operator $T_{0}$ has the form $\sum_{k\leq -1}\xi _{k}e_{k}$.
Moreover, we have
\begin{equation*}
(S-I)\big (\sum_{k\leq -1}\xi _{k}e_{k}\Big )=\xi _{-1}e_{0}+(\xi
_{-2}-\xi _{-1})e_{-1}+...
\end{equation*}%
and we see that formally
\begin{equation}
T_{0}(S-I)\Big (\sum_{k\leq -1}\xi _{k}e_{k}\Big )=\sum_{k\leq
-1}\xi _{k}e_{k},  \label{NM14}
\end{equation}%
and
\begin{equation}
(S-I)T_{0}\Big (\sum_{k\leq 0}\eta _{k}e_{k}\Big )=\sum_{k\leq
0}\eta _{k}e_{k}.  \label{NM15}
\end{equation}%
Let
\begin{equation*}
\ell _{\theta ,q}(x)_{\left( -\infty ,m\right] }:=\bigg \{\sum_{k\in {%
\mathbb{Z}}}\eta _{k}e_{k}\in \ell _{\theta ,q}(x);\,\eta _{k}=0,\,k\geq m+1%
\bigg \}
\end{equation*}%
be equipped with the norm induced from $\ell _{\theta ,q}(x)$.
Then equalities (\ref{NM14}) and (\ref{NM15}) show that the
operator $T_{0}\colon \ell _{\theta ,q}(x)_{\left( -\infty
,0\right] }\rightarrow \ell _{\theta ,q}(x)_{\left( -\infty
,-1\right] }$ is inverse to the operator $S-I\colon \ell _{\theta
,q}(x)_{\left( -\infty ,-1\right]}\rightarrow \ell _{\theta
,q}(x)_{\left( -\infty ,0\right]}$. Thus from boundedness of these
operators (boundedness of $T_0$ follows from \ref{K21},
boundedness of $S-I$ is evident) and Theorem 12 from \cite{KM}, we
obtain that the operator $T_{0}$ is a bounded operator from $\ell
_{\theta +\varepsilon ,q}(x)_{\left( -\infty ,0\right]
}\rightarrow \ell _{\theta +\varepsilon ,q}(x)_{\left( -\infty
,-1\right] }$ for $\varepsilon
>0$ small enough and its norm does not depend on $x\in \ker _{X_{0}+X_{1}}A$.
In consequence for each $i<j\leq 0$ we have
\begin{align*}
\frac{K(2^{i},x;\vec{X})}{2^{(\theta +\varepsilon )i}}&
=\left\Vert
e_{i}\right\Vert _{\ell _{\theta +\varepsilon ,q}(x)_{\left( -\infty ,-1%
\right] }}\leq \left\Vert T_{0}e_{j}\right\Vert _{\ell _{\theta
+\varepsilon
,q}(x)_{\left( -\infty ,-1\right] }} \\
& \leq \gamma \left\Vert e_{j}\right\Vert _{\ell _{\theta
+\varepsilon
,q}(x)_{\left( -\infty ,-0\right] }}=\gamma \,\frac{K(2^{j},x;\vec{X})}{%
2^{(\theta +\varepsilon )j}}{\text{.}}
\end{align*}%
This implies that
\begin{equation*}
\frac{K(s,x;\vec{X})}{K(t,x;\vec{X})}\leq \gamma \left(
\frac{s}{t}\right) ^{\theta +\varepsilon },\quad\, 0<s\leq t \leq
1
\end{equation*}%
with constant $\gamma >0$ independent of $x\in \ker A$.~This
proves that the required estimate $\theta <\alpha _{0}(\ker
_{X_{0}+X_{1}}A)$ holds.
\end{proof}


\subsection{Necessity condition for the class ${\mathbb{F}}_{\protect\theta %
,q}^{2}$}

The proof of the following theorem is rather involved, later on we
showed another proof of this theorem which used duality. Note that
duality can be used only under some additional assumptions.

\begin{theorem}
\label{TT1} Suppose that the operator $A\colon
(X_{0},X_{1})\rightarrow (Y_{0},Y_{1})$
is invertible on endpoint spaces and belongs to the class ${\mathbb{F}}%
_{\theta ,q}^{2}$, i.e. it is injective, Fredholm on
$(X_{0},X_{1})_{\theta ,q}$ and dimension of $\ker
_{X_{0}+X_{1}}A$ is finite and is equal to the codimension of
$A((X_{0},X_{1})_{\theta ,q})$ in $(Y_{0},Y_{1})_{\theta ,q}$.
Then
\begin{equation}
\beta _{0}(\ker _{X_{0}+X_{1}}A)<\theta <\alpha _{\infty }\left(
\ker _{X_{0}+X_{1}}A\right) {\text{.}}  \label{NM16}
\end{equation}
\end{theorem}

Now let
\begin{equation}
\Omega _{A}=\bigcap_{x\in \ker _{X_{0}+X_{1}}A\setminus \left\{
0\right\} } \left[ \alpha (x),\beta (x)\right] .  \label{NM17}
\end{equation}%

Notice that if $\theta $ satisfies (\ref{NM16}) \ then $\theta \in
\Omega _{A}$. Indeed, for every $x\in \ker
_{X_{0}+X_{1}}A\setminus \left\{ 0\right\} $ we have
\begin{equation*}
\alpha (x)\leq \alpha _{0}(x)\leq \beta _{0}(x)\leq \beta
_{0}(\ker _{X_{0}+X_{1}}A)
\end{equation*}%
and

\begin{equation*}
\alpha _{\infty }\left( \ker _{X_{0}+X_{1}}A\right) \leq \alpha
_{\infty }\left( x\right) \leq \beta _{\infty }(x)\leq \beta (x).
\end{equation*}%
The proof of Theorem \ref{TT1} is done in two steps. On the first
step we prove Theorem \ref{TT1} for $\theta \in \Omega _{A}$
(Theorem \ref{T1}) and on the second step (see Theorem \ref{TT2})
we prove it for $\theta \notin \Omega _{A}$. Note that Theorem
\ref{T1} and Theorem \ref{TT2} provide some additional
information.

\subsubsection{ Theorem \protect\ref{TT1} for the case $\protect\theta \in
\Omega_{A}$}

Theorem \ref{TT1} for $\theta \in \Omega _{A}$ follows immediately
from the following result.

\begin{theorem}
\label{T1} Suppose that the operator $A$ is invertible on endpoint
spaces, has finite dimensional kernel $\ker _{X_{0}+X_{1}}A$ and
that $\theta \in \Omega _{A}$. If the operator $A$ \textit{i}s
injective on $(X_{0},X_{1})_{\theta
,q}$ and $A((X_{0},X_{1})_{\theta ,q})$ is a~closed subspace of $%
(Y_{0},Y_{1})_{\theta ,q}$\textit{\ then}%
\begin{equation*}
\beta _{0}(\ker _{X_{0}+X_{1}}A)<\theta <\alpha _{\infty }\left(
\ker _{X_{0}+X_{1}}A\right) .
\end{equation*}
\end{theorem}

Everywhere below we use the facts, without restating them
explicitly, that $\theta \in \Omega _{A}$ and that $A$ is
invertible on endpoint spaces. We will give a detailed proof of
the left hand side inequality, i.e., $\beta _{0}(\ker
_{X_{0}+X_{1}}A)<\theta $, because the right hand side can be
obtained similarly by using operator $T_1$ instead of operator
$T_0$. We will need to prove several lemmas. The first lemma
stated below shows that from injectivity of $A$ on
$(X_{0},X_{1})_{\theta ,q}$ and closedness of $A\left(
(X_{0},X_{1})_{\theta ,q}\right) $ in $(Y_{0},Y_{1})_{\theta ,q}$,
it follows invertibility of $T_{0}$ on the subspace $U^{0}$
(defined in (\ref{NM18})) of $\ell _{\theta ,q}(x)$ for all $x\in
\ker _{X_{0}+X_{1}}A\setminus \left\{ 0\right\} $.


\begin{lemma}
\label{operatortT0} Suppose that the operator $A$ is injective on $%
(X_{0},X_{1})_{\theta ,q}$ and $A\left ((X_{0},X_{1})_{\theta
,q}\right )$ is closed in $(Y_{0},Y_{1})_{\theta ,q}$. Then the
operator $T_{0}$ is bounded on the subspace $U^{0}$ of $\ell
_{\theta ,q}(x)$ for any $x\in \ker _{X_{0}+X_{1}}A\setminus \left
\{0\right \}$. Moreover, the norm of the operator $T_{0}$ on
$U^{0}$ can be estimated by a constant which does not depend on
$x$ from the $\ker _{X_{0}+X_{1}}A\setminus \left \{0\right \}$.
Since $T_{0}=T_{1}$ on $U^{0}$, the same result holds for $T_{1}$.
\end{lemma}

\begin{proof} The proof is very similar to the proof of Lemma 6
in \cite{AK1}. However, some details are different, and for the
sake of completeness we repeat some essential parts of the proof.
Everywhere below $x\in \ker
_{X_{0}+X_{1}}A\setminus \left\{ 0\right\} $ and $\sum_{i\in {\mathbb{Z}}%
}\eta _{i}e_{i}\in U^{0}$. We need to show that
\begin{equation}
\Big \|T_{0}\Big (\sum_{i\in {\mathbb{Z}}}\eta _{i}e_{i}\Big )\Big
\|_{\ell
_{\theta ,q}(x)}\leq \gamma \Big \|\sum_{i\in {\mathbb{Z}}}\eta _{i}e_{i}%
\Big \|_{\ell _{\theta ,q}(x)}.  \label{K22}
\end{equation}%
For each $k\in {\mathbb{Z}}$ we can find $x_{0}(k)\in X_{0}$ and $%
x_{1}(k)\in X_{1}$ such that $x=x_{0}(k)+x_{1}(k)$ and
\begin{equation*}
\Vert x_{0}(k)\Vert _{X_{0}}+2^{k}\Vert x_{1}(k)\Vert _{X_{1}}\leq
2K(2^{k},x;\vec{X}),{\text{ }}k\in {\mathbb{Z}}
\end{equation*}%
and consider the element
\begin{equation}
y=\sum_{k\in {\mathbb{Z}}}\eta _{k}y_{k},  \label{K24}
\end{equation}%
where $y_{k}=Ax_{0}(k)=-Ax_{1}(k)\in Y_{0}\cap Y_{1}$ , as $x\in
\ker _{X_{0}+X_{1}}A\setminus \left\{ 0\right\} $. Then from the
invertibility of $A$ on the endpoint spaces we have
\begin{equation*}
J(2^{k},y_{k};\vec{Y})\leq \gamma (\left\Vert x_{0}(k)\right\Vert
_{X_{0}}+2^{k}\left\Vert x_{1}(k)\right\Vert _{X_{1}})\leq \gamma K(2^{k},x;%
\vec{X})
\end{equation*}%
with $\gamma >0$ independent of $k\in {\mathbb{Z}}$ and $x\in \ker
_{X_{0}+X_{1}}A$. Since $\sum_{k\in {\mathbb{Z}}}\eta _{k}=0$ and
the sum in (\ref{K24}) consists of finite number of terms we
obtain
\begin{equation*}
y=\sum_{k\in {\mathbb{Z}}}\eta _{k}Ax_{0}(k)=\sum_{k\in
{\mathbb{Z}}}\eta
_{k}Ax_{0}(k)-\Big (\sum_{k\in {\mathbb{Z}}}\eta _{k}\Big )Ax_{0}(0)=A\Big (%
\sum_{k\in {\mathbb{Z}}}\eta _{k}(x_{0}(k)-x_{0}(0)\Big
){\text{.}}
\end{equation*}%
Thus $y=Au$, where
\begin{equation*}
u=\sum_{k\in {\mathbb{Z}}}\eta _{k}(x_{0}(k)-x_{0}(0))\in
X_{0}\cap X_{1}.
\end{equation*}%
From the assumptions that the operator $A$ is injective on $%
(X_{0},X_{1})_{\theta ,q}$ and that the image of
$(X_{0},X_{1})_{\theta ,q}$ under operator $\ A$ is a closed
subspace of $(Y_{0},Y_{1})_{\theta ,q}$, it follows from the
Banach theorem on inverse operator that
\begin{equation*}
\left\Vert u\right\Vert _{(X_{0},X_{1})_{\theta ,q}}\leq \gamma
\left\Vert y\right\Vert _{(Y_{0},Y_{1})_{\theta ,q}}{\text{.}}
\end{equation*}%
Therefore, there exists a decomposition of $u=\sum_{k\in
{\mathbb{Z}}}u_{k}$ with $u_{k}\in X_{0}\cap X_{1}$ such that
\begin{align*}
\bigg (\sum_{k\in {\mathbb{Z}}}\big (2^{-\theta
k}J(2^{k},u_{k};\vec{X})\big )^{q}\bigg )^{1/q}& \leq \gamma
\left\Vert y\right\Vert
_{(Y_{0},Y_{1})_{\theta ,q}}\leq \gamma \bigg (\sum_{k\in {\mathbb{Z}}}\big (%
2^{-\theta k}J(2^{k},\eta _{k}y_{k};\vec{Y})\big )^{q}\bigg )^{1/q} \\
& \leq \gamma \bigg (\sum_{k\in {\mathbb{Z}}}\big (2^{-\theta
k}\left\vert
\eta _{k}\right\vert K(2^{k},x;\vec{X})\big )^{q}\bigg )^{1/q}=\gamma \Big \|%
\sum_{i\in {\mathbb{Z}}}\eta _{i}e_{i}\Big \|_{\ell _{\theta ,q}(x)}{\text{.}%
}
\end{align*}%
Since $y=Au=\sum_{k\in {\mathbb{Z}}}Au_{k}$ (the series absolutely
converges
in $Y_{0}+Y_{1}$), $\sum_{k\in {\mathbb{Z}}}(\eta _{k}y_{k}-Au_{k})=0$ in $%
Y_{0}+Y_{1}$ and we can define the elements
\begin{equation}
z_{k}=\sum_{m>k}(\eta _{m}y_{m}-Au_{m})=-\sum_{m\leq k}(\eta
_{m}y_{m}-Au_{m}){\text{.}}  \label{T4.2}
\end{equation}%
Now repeating the arguments from \cite{AK1} (see, pp.~238--239),
we obtain
\begin{equation*}
\bigg (\sum_{k\in {\mathbb{Z}}}\big (2^{-\theta
k}J(2^{k},z_{k};\vec{Y})\big
)^{q}\bigg )^{1/q}\leq \gamma \Big \|\sum_{i\in {\mathbb{Z}}}\eta _{i}e_{i}%
\Big \|_{\ell _{\theta ,q}(x)}{\text{.}}
\end{equation*}%
Since $u\in X_{0}\cap X_{1}$ and $\sum_{m\leq k}u_{k}\in X_{0}$,
\begin{equation*}
\sum_{m>k}u_{k}=u-\sum_{m\leq k}u_{k}\in X_{0}{\text{.}}
\end{equation*}%
Hence from the invertibility of the operator $A$ on the endpoint spaces and $%
y_{m}=Ax_{0}(m)=-Ax_{1}(m)$ we obtain
\begin{equation*}
\left\Vert z_{k}\right\Vert _{Y_{1}}=\Big \|\sum_{m>k}(\eta _{m}y_{m}-Au_{m})%
\Big \|_{Y_{1}}\geq \gamma \Big \|\sum_{m>k}(\eta _{m}x_{1}(m)+u_{m})\Big \|%
_{X_{1}}
\end{equation*}%
and from (\ref{T4.2})
\begin{align*}
\left\Vert z_{k}\right\Vert _{Y_{0}}& =\Big \|\sum_{m>k}(\eta
_{m}y_{m}-Au_{m})\Big \|_{Y_{0}}=\Big \|-\sum_{m\leq k}(\eta
_{m}y_{m}-Au_{m})\Big \|_{Y_{0}} \\
& =\Big \|A\Big (\sum_{m\leq k}(\eta _{m}x_{0}(m)-u_{m})\Big )\Big \|%
_{Y_{0}}\geq \gamma \Big \|\sum_{m\leq k}(\eta _{m}x_{0}(m)-u_{m}\Big \|%
_{X_{0}} \\
& =\gamma \Big \|(\sum_{m\leq k}(\eta _{m}x_{0}(m)-u\Big )+\Big (%
u-\sum_{m\leq k}u_{m}\Big )\Big \|_{X_{0}}=\gamma \Big
\|-\sum_{m>k}(\eta _{m}x_{0}(m)-u_{m})\Big \|_{X_{0}}{\text{.}}
\end{align*}%
In the last line we used the fact that from $\sum_{k\in
{\mathbb{Z}}}\eta _{k}=0$ it follows
\begin{equation*}
u=\sum_{k\in {\mathbb{Z}}}\eta _{k}(x_{0}(k)-x_{0}(0))=\sum_{k\in {\mathbb{Z}%
}}\eta _{k}x_{0}(k)
\end{equation*}%
and so
\begin{equation*}
\left\Vert z_{k}\right\Vert _{Y_{0}}+2^{k}\left\Vert
z_{k}\right\Vert
_{Y_{1}}\geq \gamma \bigg (\Big \|\sum_{m>k}(\eta _{m}x_{0}(m)-u_{m})\Big \|%
_{X_{0}}+2^{k}\bigg \|\sum_{m>k}(\eta _{m}x_{1}(m)+u_{m})\Big \|_{X_{1}}%
\bigg ){\text{.}}
\end{equation*}%
Since
\begin{equation*}
\sum_{m>k}(\eta _{m}x_{0}(m)-u_{m})+\sum_{m>k}(\eta
_{m}x_{1}(m)+u_{m})=\Big (\sum_{m>k}\eta _{m}\Big )x,
\end{equation*}%
\begin{equation*}
\left\Vert z_{k}\right\Vert _{Y_{0}}+2^{k}\left\Vert
z_{k}\right\Vert
_{Y_{1}}\geq \gamma K(2^{k},x;\vec{X})\Big |\sum_{m>k}\eta _{m}\Big |{\text{.%
}}
\end{equation*}%
This implies
\begin{align*}
\bigg (\sum_{k\in {\mathbb{Z}}}\big (2^{-\theta
k}J(2^{k},z_{k};\vec{X})\big
)^{q}\bigg )^{1/q}& \geq \gamma \bigg (\sum_{k\in {\mathbb{Z}}}\Big (%
2^{-\theta k}K(2^{k},x;\vec{X})\Big |\sum_{m>k}\eta _{m}\Big |\Big )^{q}%
\bigg )^{1/q} \\
& =\gamma \Big \|T_{0}\Big (\sum_{i\in {\mathbb{Z}}}\eta
_{i}e_{i}\Big )\Big \|_{\ell _{\theta ,q}(x)}
\end{align*}%
and we obtain
\begin{equation*}
\Big \|T_{0}\Big (\sum_{i\in {\mathbb{Z}}}\eta _{i}e_{i}\Big )\Big
\|_{\ell _{\theta ,q}(x)}\leq \bigg (\sum_{k\in {\mathbb{Z}}}\big
(2^{-\theta k}J(2^{k},z_{k};\vec{X})\big )^{q}\bigg )^{1/q}\leq
\gamma \Big \|\sum_{i\in {\mathbb{Z}}}\eta _{i}e_{i}\Big \|_{\ell
_{\theta ,q}(x)}{\text{.}}
\end{equation*}%
This means that the operator $T_{0}$ is bounded on the subspace $U^{0}$\ of $%
\ell _{\theta ,q}(x)$ and the above proof shows that its norm does
not depend on $x$\ from the $\ker _{X_{0}+X_{1}}A\setminus \left\{
0\right\} $.
\end{proof}


\begin{corollary}
\label{Corollary1} Let $A$ be an injective operator on $(X_{0},X_{1})_{%
\theta ,q}$ such that $A\left( (X_{0},X_{1})_{\theta ,q}\right) $
is closed in $(Y_{0},Y_{1})_{\theta ,q}$. If $x\in \ker
_{X_{0}+X_{1}}A\setminus \left\{ 0\right\} $ and $\theta \in
\left[ \alpha (x),\beta (x)\right] $ then $U^{0}$ is not dense in
$U$ in the norm of $\ell _{\theta ,q}(x)$ and
\begin{equation*}
\sum_{k<0}e_{k}\notin \ell _{\theta ,q}(x){\text{, }}\sum_{k\geq
0}e_{k}\notin \ell _{\theta ,q}(x).
\end{equation*}
\end{corollary}

\begin{proof} Indeed, if we suppose that $U^{0}$ is dense in $U$, then
the operator $T_{0}$ can be extended to a~bounded operator on $U$
and therefore (see \cite[lemma 4]{AK1}), it will be an inverse
operator to $S-I$ on the whole space $\ell _{\theta ,q}(x)$. But
from Lemma \ref{Lemma1} we know that $S-I$ is not invertible on
$\ell _{\theta ,q}(x)$ for $\theta \in \left[ \alpha (x),\beta
(x)\right] $, so $U^{0}$ is not dense in $U$ in the norm of $\ell
_{\theta ,q}(x)$. Since $U=U^{0}+\mathrm{span}\left\{
e_{0}\right\} $ and $T_{0}$ is bounded on $U^{0}$, unboundedness
of $T_{0}$ on $U$ is equivalent to $T_{0}e_{0}\notin \ell _{\theta
,q}(x)$. Similarly, the unboundedness of $T_{1}$ on $U$ is
equivalent to $T_{1}e_{0}=\sum_{k\geq 0}e_{k}\notin \ell _{\theta
,q}(x).$
\end{proof}

Let us observe that $e_{n}-e_{0}\in U^{0}$ for $n<0$. Then from Lemma \ref%
{operatortT0}, it follows that for every $x\in \ker
_{X_{0}+X_{1}}A$ and each $k$ with $n<k<0$ we have
\begin{align}
2^{-\theta k}K(2^{k},x;\vec{X})& \leq \left\Vert
T_{0}(e_{n}-e_{0})\right\Vert _{\ell _{\theta ,q}(x)}\leq \gamma
\left\Vert
e_{n}-e_{0}\right\Vert _{\ell _{\theta ,q}(x)}  \label{T4.5} \\
& \leq \gamma (2^{-\theta n}K(2^{n},x;\vec{X})+K(1,x;\vec{X})),
\notag
\end{align}%
where $\gamma >0$ does not depend on $n<0$ and $x\in \ker
_{X_{0}+X_{1}}A$.

The next lemma shows that the second term on the right hand side
can be estimated by the first. In the proof we will use assumption
that the kernel of $A$ in $X_{0}+X_{1}$ is a finite dimensional
space (we did not use this assumption above).


\begin{lemma}
\label{Lemma3} Suppose that an operator $A\colon
(X_{0},X_{1})_{\theta ,q}\rightarrow (Y_{0},Y_{1})_{\theta ,q}$ is
injective and $A\left( (X_{0},X_{1})_{\theta ,q}\right) $ is
closed in $(Y_{0},Y_{1})_{\theta ,q}$. Then there exists a
constant $\gamma >0$ such that for all $x\in $ $\ker
_{X_{0}+X_{1}}A$ and all $n\in {\mathbb{Z}}$,
\begin{equation*}
2^{-\theta n}K(2^{n},x;\vec{X})\geq \gamma K(1,x;\vec{X}).
\end{equation*}
\end{lemma}


\begin{proof} We prove the inequality for $n<0$ (for $n>0$ the proof is
similar; in this case we need to use the operator $T_{1}$ instead of $T_{0}$%
). Let $x\in \ker _{X_{0}+X_{1}}A$ and $n<0$. As $e_{n}-e_{0}\in
U^{0}$ then
from boundedness of $T_{0}$ on $U^{0}$ in $\ell _{\theta ,q}(x)$ (see Lemma %
\ref{operatortT0}) we have
\begin{align}
\bigg (\sum_{k:n\leq k<0}\big (2^{-\theta k}K(2^{k},x;\vec{X})\big )^{q}%
\bigg )^{1/q}& =\Vert T_{0}(e_{n}-e_{0})\Vert _{\ell _{\theta
,q}(x)}\leq \gamma \left\Vert e_{n}-e_{0}\right\Vert _{\ell
_{\theta ,q}(x)}
\label{K101} \\
& \leq \gamma \left( 2^{-\theta
n}K(2^{n},x;\vec{X})+K(1,x;\vec{X})\right) . \notag
\end{align}%
From Corollary \ref{Corollary1}, it follows that
$\sum_{k<0}e_{k}\notin \ell _{\theta ,q}(x)$. Thus the left hand
side approaches infinity as $n\rightarrow -\infty$, and so the
right hand side also approaches infinity. Therefore, there exists
a~positive integer $N(x)$ such that for all $n<-N(x)$,
\begin{equation*}
2^{-\theta n}K(2^{n},x;\vec{X})\geq K(1,x;\vec{X}){\text{.}}
\end{equation*}%
Let $N(x)$ be the smallest positive integer for which the above
inequality holds for all $n<-N(x)$ and put
\begin{equation*}
N:=\sup_{x\in \ker _{X_{0}+X_{1}}A}N(x){\text{.}}
\end{equation*}%
Then $N<\infty $. Indeed, if $N$ were not finite, then there would
exists a
sequence $\{x_{k}\}$ in $\ker _{X_{0}+X_{1}}A$ such that $%
N(x_{k})\rightarrow \infty $ as $k\rightarrow \infty $ and
\begin{equation*}
2^{\theta N(x_{k})}K(2^{-N(x_{k})},x_{k};\vec{X})<K(1,x_{k};\vec{X}){\text{.}%
}
\end{equation*}%
By homogeneity we may assume that
\begin{equation*}
\left\Vert x_{k}\right\Vert
_{X_{0}+X_{1}}=K(1,x_{k};\vec{X})=1{\text{.}}
\end{equation*}%
Since the space $\ker _{X_{0}+X_{1}}A$ is finite dimensional,
without loss
of generality we may assume that the sequence $\{x_{k}\}$ converges in $%
\left\Vert \cdot \right\Vert _{X_{0}+X_{1}}$ to some element
$x_{\ast }\in
\ker _{X_{0}+X_{1}}A\setminus \left\{ 0\right\} $. Then for each integer $%
m<0 $ we have (using that $N(x_{k})\rightarrow \infty $ when
$k\rightarrow \infty $)
\begin{align*}
\bigg (& \sum_{l:m\leq l<0}\big (2^{-\theta l}K(2^{l},x_{\ast
};\vec{X})\big
)^{q}\bigg )^{1/q}=\lim_{k\rightarrow \infty }\bigg (\sum_{l:m\leq l<0}\big (%
2^{-\theta l}K(2^{l},x_{k};\vec{X})\big )^{q}\bigg )^{1/q} \\
& \leq \lim_{k\rightarrow \infty }\bigg (\sum_{l:-N(x_{k})\leq l<0}\big (%
2^{-\theta l}K(2^{l},x_{k};\vec{X})\big )^{q}\bigg )^{1/q}=\lim_{k%
\rightarrow \infty }\left\Vert
T_{0}(e_{-N(x_{k})}-e_{0})\right\Vert _{\ell
_{\theta ,q}(x_{k})} \\
& \leq \lim_{k\rightarrow \infty }\gamma (2^{-\theta
(-N(x_{k}))}K(2^{-N(x_{k})},x_{k};\vec{X})+K(1,x_{k};\vec{X}))\leq
\lim_{k\rightarrow \infty }\gamma K(1,x_{k};\vec{X})=\gamma
{\text{.}}
\end{align*}%
In the last inequality we used the fact that the constant $\gamma >0$ does not depend on $%
x_{k}$ (norms of operators $T_{0{\text{ }}}$ are uniformly bounded). Since $%
m<0$ is arbitrary,
\begin{equation*}
\bigg (\sum_{l:l<0}\big (2^{-\theta l}K(2^{l},x_{\ast };\vec{X})\big )^{q}%
\bigg )^{1/q}\leq \gamma ,
\end{equation*}%
and this contradicts the property that $\sum_{k<0}e_{k}\notin \ell
_{\theta ,q}(x_{\ast })$ (see Corollary \ref{Corollary1}). Thus we
have proved that there exists a~natural number $N$ such that for
every $x\in \ker _{X_{0}+X_{1}}A$ and each $n<-N$
\begin{equation*}
2^{-\theta n}K(2^{n},x;\vec{X})\geq K(1,x;\vec{X}){\text{.}}
\end{equation*}%
Combining this with the fact that $x\mapsto 2^{-\theta
n}K(2^{n},x;X_{0},X_{1})$ are equivalent norms on the finite dimensional space $%
\ker _{X_{0}+X_{1}}A$ for each integer $n$, we conclude that there
exists a~constant $\gamma >0$ such that for each $n<0$
\begin{equation*}
2^{-\theta n}K(2^{n},x;\vec{X})\geq \gamma
K(1,x;\vec{X}){\text{{}}}
\end{equation*}%
and this completes the proof.
\end{proof}


\begin{remark}
As the left hand side of $(\ref{K101})$ approaches to infinity
as $n\rightarrow -\infty $, we have%
\begin{equation}
\lim_{n\to - \infty} 2^{-\theta n}K(2^{n},x;\vec{X})= \infty
\text{.} \label{K102}
\end{equation}
\end{remark}

Note that from inequality (\ref{T4.5}) and Lemma \ref{Lemma3}, we
obtain
\begin{equation*}
\frac{K(2^{n},x;\vec{X})}{K(2^{k},x;\vec{X})}\geq \gamma \left( \frac{2^{n}}{%
2^{k}}\right) ^{\theta }
\end{equation*}%
with a~constant $\gamma >0$ independent of $n<k<0$ and any $x\in
\ker
_{X_{0}+X_{1}}A\setminus \left\{ 0\right\} $. This implies that for all $%
0<s\leq t\leq 1$ and all $x\in \ker _{X_{0}+X_{1}}A\setminus
\left\{ 0\right\} $,
\begin{equation*}
\frac{K(s,x;\vec{X})}{K(t,x;\vec{X})}\geq \gamma \left(
\frac{s}{t}\right) ^{\theta }.
\end{equation*}%
This shows that
\begin{equation*}
\beta _{0}(\ker _{X_{0}+X_{1}}A)\leq \theta {\text{.}}
\end{equation*}%
However, to prove the required inequality $\beta _{0}(\ker
_{X_{0}+X_{1}}A)<\theta $ we need some extra considerations. We
start with following key proposition.


\begin{proposition}
\label{Proposition1}Let
\begin{equation*}
\mathrm{dist}_{\ell _{\theta
,q}(x)}(e_{0},U^{0}):=\func{inf}_{g\in U^{0}}\left\Vert
e_{0}-g\right\Vert _{\ell _{\theta ,q}(x)}
\end{equation*}%
be the distance from the element $e_{0}$ to the subspace $U^{0}$
in the metric of $\ell _{\theta ,q}(x)$. Then there exists
constant $\gamma >0$ such that for all $x\in \ker _{X_{0}+X_{1}}A$
we have
\begin{equation*}
\mathrm{dist}_{\ell _{\theta ,q}(x)}(e_{0},U^{0})\geq \gamma
\left\Vert x\right\Vert _{X_{0}+X_{1}}{\text{ .}}
\end{equation*}
\end{proposition}


\begin{proof} Suppose that the statement is not true. Then for any $%
\varepsilon >0$ there exists $x\in \ker _{X_{0}+X_{1}}A$ such that ${\text{%
dist}}_{\ell _{\theta ,q}(x)}(e_{0},U^{0})<\varepsilon \left\Vert
x\right\Vert _{X_{0}+X_{1}}$. Moreover, from the homogeneity of the norm in $%
\ell _{\theta ,q}(x)$, it follows that we can suppose that
$\left\Vert x\right\Vert _{X_{0}+X_{1}}=1$. Then there exists an
element $g\in U^{0}$ such that
\begin{equation*}
\left\Vert e_{0}-g\right\Vert _{\ell _{\theta ,q}(x)}<\varepsilon
.
\end{equation*}%
Thus if $g=\sum_{i\in {\mathbb{Z}}}g_{i}e_{i}$, then
\begin{equation*}
\bigg (\sum_{i\in {\mathbb{Z\setminus }}\left\{ 0\right\} }\Big (\frac{%
\left\vert g_{i}\right\vert K(2^{i},x;\vec{X})}{2^{\theta i}}\Big
)^{q}+\Big
(\frac{\left\vert 1-g_{0}\right\vert K(2^{0},x;\vec{X})}{2^{\theta \cdot 0}}%
\Big )^{q}\bigg )^{1/q}<\varepsilon.
\end{equation*}%
In particular this yields the estimate
\begin{equation}
\left\vert 1-g_{0}\right\vert <\varepsilon \text{.}  \label{K12}
\end{equation}%
As from Lemma \ref{Lemma3} we have
\begin{equation*}
2^{-\theta i}K(2^{i},x;\vec{X})\geq \gamma K(1,x;\vec{X})=\gamma
\left\Vert x\right\Vert _{X_{0}+X_{1}}=\gamma
\end{equation*}%
with a~positive $\gamma $ independent of $x\in \ker _{X_{0}+X_{1}}A{\mathbb{%
\setminus }}\left\{ 0\right\} $, thus for all $i\in {\mathbb{Z\setminus }}%
\left\{ 0\right\} $ we have
\begin{equation}
\left\vert g_{i}\right\vert <\varepsilon /\gamma {\text{.}}
\label{K13}
\end{equation}%
Since $g\in U^{0}$, $\sum_{i\in {\mathbb{Z}}}g_{i}=0$ we have
\begin{equation*}
1-\varepsilon \leq |g_{0}|=\Big |2g_{0}+\sum_{i\in {\mathbb{Z\setminus }}%
\left\{ 0\right\} }g_{i}\Big |\leq \Big |g_{0}+\sum_{i<0}g_{i}\Big |+\Big |%
g_{0}+\sum_{i>0}g_{i}\Big |\,.
\end{equation*}%
In consequence $\varepsilon <1/3$ implies
\begin{equation*}
\Big |g_{0}+\sum_{i<0}g_{i}\Big |\geq 1/3\quad
\,{\text{or\thinspace \thinspace\ $\Big |g_{0}+\sum_{i>0}g_{i}\Big
|\geq 1/3$}}.
\end{equation*}%
Let us consider the case when
\begin{equation}
{\text{ }}\Big |g_{0}+\sum_{i>0}g_{i}\Big |\geq 1/3,  \label{K14}
\end{equation}%
(the case when $\left\vert g_{0}+\sum_{i<0}g_{i}\right\vert \geq
1/3$ can be
considered analogously by using operator $T_{1}$ instead of the operator $%
T_{0}$). Since $\left\Vert g\right\Vert _{\ell _{\theta
,q}(x)}\leq \left\Vert e_{0}\right\Vert _{\ell _{\theta
,q}(x)}+\left\Vert e_{0}-g\right\Vert _{\ell _{\theta
,q}(x)}<1+\varepsilon $, it follows from boundedness of $T_{0}$ on
$U^{0}$ (with constant independent of $x\in \ker
_{X_{0}+X_{1}}A{\mathbb{\setminus }}\left\{ 0\right\} $) that
\begin{equation*}
\left\Vert T_{0}g\right\Vert _{\ell _{\theta ,q}(x)}\leq \gamma
_{1}.
\end{equation*}%
Let us choose a~number $N$ such that $N\frac{\varepsilon }{\gamma
}\leq 1/4$
with constant $\gamma >0$ from (\ref{K13}). Then from estimates (\ref{K13})-(%
\ref{K14}) and formula (\ref{K10}) for $T_{0}$ we obtain
\begin{equation*}
\left\Vert T_{0}g\right\Vert _{\ell _{\theta ,q}(x)}\geq \left\vert \frac{1}{%
3}-\frac{1}{4}\right\vert \frac{K(2^{-N},x;\vec{X})}{2^{-\theta
N}}.
\end{equation*}%
If $\varepsilon \rightarrow 0$, then $N\rightarrow \infty $ and
$2^{\theta N}K(2^{-N},x;\vec{X})\rightarrow \infty $ (see
(\ref{K102})) a contradiction by \linebreak $\left\Vert
T_{0}g\right\Vert _{\ell _{\theta ,q}(x)}\leq \gamma _{1}$.
\end{proof}

The next lemma shows that if $\theta \in \Omega _{A}$, then closedness of $%
A\left( (X_{0},X_{1})_{\theta ,q}\right) $ in
$(Y_{0},Y_{1})_{\theta ,q}$ implies boundedness of $T_{0}$ on
$U^{0}$ not only on $\ell _{\theta ,q}(x)$ but also on $\ell
_{\tilde{\theta},q}(x)$ for $\tilde{\theta}\in \left( \theta
-\varepsilon ,\theta +\varepsilon \right) $ with $\varepsilon >0$
independent of $x\in \ker _{X_{0}+X_{1}}A\setminus \left\{
0\right\} .$


\begin{lemma}
\label{Lemma4} Suppose that $A$ is injective on
$(X_{0},X_{1})_{\theta ,q}$
and $A\left ((X_{0},X_{1})_{\theta ,q}\right )$ is closed in $%
(Y_{0},Y_{1})_{\theta ,q}$. If $\theta \in \Omega _{A}$ then there exists $%
\varepsilon >0$\ such that for all $\tilde {\theta }\in \left
(\theta -\varepsilon ,\theta +\varepsilon \right )$
\begin{equation*}
\left \Vert T_{0}\right \Vert _{U^{0}\cap \,\ell _{\tilde {\theta }%
,q}(x)\rightarrow \ell _{\tilde {\theta },q}(x)}\leq \gamma ,
\end{equation*}
where $\gamma $\ is some positive constant which does not depend
on $x$\ from the $\ker _{X_{0}+X_{1}}A$.
\end{lemma}


\begin{proof} Let $\ell _{\theta ,q}^{U^{0}}(x)$ be the closure of the
space $U^{0}$ in $\ell _{\theta ,q}(x)$ and $\ell _{\theta
,q}^{U}(x)$ the closure of $U$ in $\ell _{\theta ,q}(x)$. Note
that for $1\leq q<\infty $, we have $\ell _{\theta ,q}^{U}(x)=\ell
_{\theta ,q}(x)$ and for $q=\infty $ we have $\ell _{\theta
,\infty }^{U}(x)=\ell _{\theta ,c_{0}}(x)$. Observe
that $\theta \in \Omega _{A}$ implies by Corollary \ref{Corollary1} that $%
U^{0}$ is not dense in $U=U^{0}+ \text{span}\left\{ e_{0}\right\} $. Hence $%
e_{0}\notin \ell _{\theta ,q}^{U^{0}}(x)$ and space $\ell _{\theta
,q}^{U^{0}}(x)$ is a~closed subspace of codimension one in $\ell
_{\theta
,q}^{U}(x)$. It is easy to check that the operator $S-I$ maps $U$ onto $%
U^{0} $ and its inverse is a operator $T_{0}=T_{1}$ (see Remark
5). Note also that the operator $S-I$ is a~bounded operator from
$\ell _{\theta ,q}^{U}(x)$ to $\ell _{\theta ,q}^{U^{0}}(x)$ and
its inverse is a~bounded operator $T_{0}$, which coincides with
$T_{1}$.

Now we wish to apply the result from \cite{KM} which states that if $%
T\colon (C_{0},C_{1})\rightarrow (D_{0},D_{1})$ is an operator
between Banach couples such that the operator $T\colon
(C_{0},C_{1})_{\theta
,q}\rightarrow (D_{0},D_{1})_{\theta ,q}$ is invertible, then there exists $%
\varepsilon >0$ such that it is also invertible on $(C_{0},C_{1})_{\tilde{%
\theta},q}$ for all $\tilde{\theta}\in \left( \theta -\varepsilon
,\theta +\varepsilon \right) $. However, trying to use this
theorem we encounter difficulties connected with the fact that the
spaces $\ell _{\theta ,q}^{U^{0}}(x)$ are not of the form $\left(
D_{0},D_{1}\right) _{\theta ,q}$. To overcome these difficulties
we will use an additional construction. Let us consider a~couple
$(B_{0},B_{1})$, where
\begin{equation*}
B_{0}=\ell_{0,q}(x)\times {\mathbb{R}},\quad \,B_{1}=\ell _{1,q}(x)\times {%
\mathbb{R}}
\end{equation*}%
(norms on $\ell_{0,q}(x)$ and $\ell_{1,q}(x)$ are defined by formulas (\ref{K9}%
)). Let $T\colon B_{0}+B_{1}\rightarrow \ell _{0,q}(x)+\ell
_{1,q}(x)$ be defined by
\begin{equation*}
T(x,\lambda )=(S-I)x+\lambda e_{0}{\text{.}}
\end{equation*}%
Then the operator $T$ is invertible on $\left( B_{0},B_{1}\right)
_{\theta ,q}=\ell _{\theta ,q}(x)\times {\mathbb{R}}$ for $1\leq
q<\infty $ and also on $\left( B_{0},B_{1}\right) _{\theta
,c_{0}}=\ell _{\theta ,c_{0}}(x)\times {\mathbb{R}}$ in the case
when $q=\infty $. Moreover, it follows from Proposition
\ref{Proposition1} that the norm of its inverse is independent of
$x\in \ker _{X_{0}+X_{1}}A\setminus \left\{ 0\right\} $. Now we
apply Theorem 12 \cite{KM} for $q<\infty $ and its analog for $
(B_{0},B_{1})_{\theta ,c_{0}}$ (see Theorem \ref{KM}). We should
make here two important remarks. First, the operator $T$ maps $U$
onto $U^{0}$ and its inverse on $U^{0}$ coincide with
$T_{0}=T_{1}$ and the inverse to $T$ is bounded, we derive that
$T_{0}$ is bounded. Secondly, in Theorem 12 \cite{KM} and also in
Theorem \ref{KM} (see Remark \ref{KMR} after it) the interval
$\left(\theta -\varepsilon ,\theta +\varepsilon \right)$ can be
chosen independent of $x\in \ker A$, because of the uniform
boundedness of norms, i.e.,
\begin{equation*}
\left\Vert T_{0}\right\Vert _{U^{0}\cap \,\ell _{\theta
,q}(x)\rightarrow \ell _{\theta ,q}(x)}\leq \gamma {\text{.}}
\end{equation*}%
Moreover, from this estimate we also obtain the uniform
boundedness of norms,
\begin{equation*}
\left\Vert T_{0}\right\Vert _{U^{0}\cap \,\ell _{\tilde{\theta}%
,q}(x)\rightarrow \ell _{\tilde{\theta},q}(x)}\leq \gamma
\end{equation*}%
for all $\tilde{\theta}\in \left( \theta -\varepsilon ,\theta
+\varepsilon \right) $.
\end{proof}


The next lemma shows that the result of Lemma \ref{Lemma3} holds uniformly for $%
\tilde{\theta}\in \left( \theta -\varepsilon ,\theta +\varepsilon
\right) $.


\begin{lemma}
\label{Lemma5} There exist constants $\varepsilon >0$ and $\gamma
>0$ such that for all $\tilde {\theta }\in \left (\theta
-\varepsilon ,\theta
+\varepsilon \right )$ and all $x\in \ker _{X_{0}+X_{1}}A$\ we have for all $%
n<0$ the inequality
\begin{equation*}
2^{-\tilde {\theta }n}K(2^{n},x;\vec {X})\geq \gamma K(1,x;\vec {X}){\text {.%
}}
\end{equation*}
\end{lemma}

\begin{proof} We first show that there exist $\varepsilon >0$ and a~positive
integer $N$ such that
\begin{equation*}
2^{-\theta n}K(2^{n},x;\vec{X})\geq \frac{1}{2}K(1,x;\vec{X})
\end{equation*}%
for $n<-N$ for all $\tilde{\theta}\in \left( \theta -\varepsilon
,\theta +\varepsilon \right) $ and all $x\in \ker
_{X_{0}+X_{1}}A$. Suppose that this does not hold. Now take
$\varepsilon _{k}=\frac{\varepsilon _{0}}{2^{k}}$, where
$\varepsilon _{0}>0$ is a~constant from Lemma \ref{Lemma4}. Then
there exists $\theta _{k}$ such that $\left\vert \theta
_{k}-\theta \right\vert <\varepsilon _{k}$, elements $x_{k}\in
\ker _{X_{0}+X_{1}}A\setminus \left\{ 0\right\} $ and numbers
$N_{k}\rightarrow +\infty $ such that
\begin{equation*}
2^{-\theta _{k}(-N_{k})}K(2^{-N_{k}},x_{k};\vec{X})<\frac{1}{2}K(1,x_{k};%
\vec{X}).
\end{equation*}%
By the homogeneity of the norm we can assume that each element
$x_{k}$ has norm one in $X_{0}+X_{1}$. Since $\ker
_{X_{0}+X_{1}}A$ is a finite dimensional space, we can find a
subsequence of the sequence $\left\{ x_{k}\right\} $ which
converges in $X_{0}+X_{1}$ to some element $x_{\ast}\in \ker
_{X_{0}+X_{1}}A\setminus \left\{ 0\right\} $. Since
$e_{-N_{k}}-e_{0}\in U^{0}$, it follows from Lemma \ref{Lemma4}
that
\begin{equation*}
\left\Vert T_{0}(e_{-N_{k}}-e_{0})\right\Vert _{\ell _{\theta
_{k},q}(x_{k})}\leq \gamma \left\Vert e_{-N_{k}}-e_{0}\right\Vert
_{\ell _{\theta _{k},q}(x_{k})}.
\end{equation*}%
Fix $n<0$; then for each $k<0$ such that $n>-N_{k}$ we have
\begin{align*}
\bigg (\sum_{m:n\leq m\leq 0}\big (2^{-\theta _{k}m}K(2^{m},x_{k};\vec{X})%
\big )^{q}\bigg )^{1/q}& \leq \gamma \big (2^{-\theta
_{k}(-N_{k})}K(2^{-N_{k}},x_{k};\vec{X})+K(1,x_{k};\vec{X})\big ) \\
& \leq \gamma K(1,x_{k};\vec{X}).
\end{align*}%
As $N_{k}\rightarrow \infty $ and so taking the limit
$k\rightarrow \infty $ we obtain that for all $n<0$ we have
\begin{equation*}
\bigg (\sum_{n\leq m\leq 0}\big (2^{-\theta m}K(2^{m},x_{\ast
};\vec{X})\big )^{q}\bigg )^{1/q}\leq \gamma K(1,x_{\ast
};\vec{X})
\end{equation*}%
with $\gamma >0$ independent of $n<0$. Hence
\begin{equation*}
\bigg (\sum_{m\leq 0}\big (2^{-\theta m}K(2^{m},x_{\ast };\vec{X})\big )^{q}%
\bigg )^{1/q}\leq \gamma K(1,x_{\ast };\vec{X})
\end{equation*}%
gives $T_{0}e_{0}\in \ell _{\theta ,q}(x_{\ast })$. However this
is a~contradiction by Corollary \ref{Corollary1}. Thus we have
shown that there exists $\varepsilon >0$ and natural $N$ such that
for all $n<-N$ and all $\tilde{\theta}\in
\left(\theta-\varepsilon,\theta +\varepsilon \right)$ and all
$x\in \ker_{X_{0}+X_{1}}A$ we have
\begin{equation*}
2^{-\tilde{\theta}n}K(2^{n},x;\vec{X})\geq
\frac{1}{2}K(1,x;\vec{X})
\end{equation*}%
and we just need to notice that for $-N\leq n<0$ we have
\begin{equation*}
2^{-\tilde{\theta}n}K(2^{n},x;\vec{X})\geq K(2^{-N},x;\vec{X})
\end{equation*}%
and norm $K(1,x;\vec{X})$ is equivalent to the norm
$K(2^{-N},x;\vec{X})$ on finite dimensional space \linebreak $\ker
_{X_{0}+X_{1}}A$. This completes the proof.
\end{proof}


Now we are ready to prove Theorem \ref{T1}.

\begin{proof} (of Theorem \ref{T1}) Since $e_{n}-e_{0}\in U^{0}$ for any $%
n<0$, it follows from Lemma \ref{Lemma4} there exists $\varepsilon
>0$ that
for every $x\in \ker _{X_{0}+X_{1}}A$ and each $k$ with $n<k<0$ we have for $%
\tilde{\theta}\in \left( \theta -\varepsilon ,\theta +\varepsilon
\right) $ the following inequality
\begin{align*}
2^{-\tilde{\theta}k}K(2^{k},x;\vec{X})& \leq \left\Vert
T_{0}(e_{n}-e_{0})\right\Vert _{\ell _{\tilde{\theta},q}(x)}\leq
\gamma
\left\Vert e_{n}-e_{0}\right\Vert _{\ell _{\tilde{\theta},q}(x)} \\
& \leq \gamma \big (2^{-\tilde{\theta}n}K(2^{n},x;\vec{X})+K(1,x;\vec{X}X)%
\big )
\end{align*}%
with $\gamma >0$ independent of $n<0$ and $x\in \ker
_{X_{0}+X_{1}}A$. Therefore from Lemma \ref{Lemma5} we obtain that
for any $x\in \ker _{X_{0}+X_{1}}A$ and any $n<k<0$ we have for
$\tilde{\theta}\in \left( \theta -\varepsilon ,\theta +\varepsilon
\right) $
\begin{equation*}
2^{-\tilde{\theta}k}K(2^{k},x;\vec{X})\leq \gamma 2^{-\tilde{\theta}%
n}K(2^{n},x;\vec{X}),
\end{equation*}%
which is equivalent to
\begin{equation*}
\frac{K(s,x;\vec{X})}{K(t,x;\vec{X})}\geq \gamma \left( \frac{s}{t}\right) ^{%
\tilde{\theta}},\quad \,0<s<t\leq 1{\text{.}}
\end{equation*}%
This means that $\theta >\beta _{0}(\ker _{X_{0}+X_{1}}A)$.
\end{proof}

\subsubsection{ Theorem \protect\ref{TT1} for the case $\protect\theta
\protect\notin \Omega _{A}$}

The following corollary from a result in \cite{AK1} will be useful
for us.

\begin{corollary}
\label{Corollary2} Suppose that the operator $A$ is invertible on
the end spaces and has one dimensional kernel $\ker
_{X_{0}+X_{1}}A=span\left\{ x_{\ast }\right\} $. Then operator $A$
is invertible on $(X_{0},X_{1})_{\theta ,q}$
if and only if $\theta \notin \left[ \alpha (x_{\ast }),\beta (x_{\ast })%
\right]$.
\end{corollary}

\begin{proof} In the case when $V_{\theta ,q}^{0}=span\left\{ x_{\ast
}\right\} $ and $V_{\theta ,q}^{1}=\left\{ 0\right\} $ condition
(3.4) of the Theorem 11 in \cite{AK1} coincide with $0<\theta
<\alpha (x_{\ast })$. If $V_{\theta ,q}^{0}=\left\{ 0\right\} $
and $V_{\theta ,q}^{1}=span\left\{ x_{\ast }\right\} $ then
condition is $\beta (x_{\ast })<\theta <1$.
\end{proof}

Theorem \ref{TT1} for $\theta \notin \Omega _{A}$ follows
immediately from the following result.

\begin{theorem}
\label{TT2}Suppose that the operator $A$ is invertible on end
spaces, injective
on $(X_{0},X_{1})_{\theta ,q}$, has finite dimensional kernel $%
\ker _{X_{0}+X_{1}}A$ and $A((X_{0},X_{1})_{\theta ,q})$ is
a~closed subspace of $(Y_{0},Y_{1})_{\theta ,q}$. Then
\begin{itemize}
\item[{\rm{a)}}] The codimension of $A((X_{0},X_{1})_{\theta ,q})$
in $(Y_{0},Y_{1})_{\theta ,q}$\textit{\ is not more than dimension
of \ }$\ker _{X_{0}+X_{1}}A$. \item[{\rm{b)}}] If $\theta \notin
\Omega _{A}$ than codimension of $
A((X_{0},X_{1})_{\theta ,q})$ in $(Y_{0},Y_{1})_{\theta ,q}$%
is strictly less than dimension of $\ker _{X_{0}+X_{1}}A$.
\end{itemize}
\end{theorem}

\begin{proof} We prove a) by induction. Let us first consider the case when
dimension of $\ker _{X_{0}+X_{1}}A$ is equal to one, i.e. $\ker
_{X_{0}+X_{1}}A= \text{span}\left\{ x_{\ast }\right\}$. In this
case from Corollary \ref{Corollary2} it follows that operator $A$
is invertible on $(X_{0},X_{1})_{\theta ,q}$ if $\theta \notin
\Omega _{A}=\left[ \alpha (x_{\ast }),\beta (x_{\ast })\right]$.
So we have property b). Moreover, from Theorem \ref{ST2} on
sufficiency if follows that if $\theta $ satisfies
\begin{equation*}
\beta_{0}(x_{\ast })<\theta <\alpha _{\infty }(x_{\ast })
\end{equation*}%
then $A((X_{0},X_{1})_{\theta ,q})$ is a~closed subspace of $
(Y_{0},Y_{1})_{\theta ,q}$ with codimension equal to one, and from
Theorem \ref{T1} we see that $A((X_{0},X_{1})_{\theta
,q})$\textit{\ }is not a~closed subspace of $(Y_{0},Y_{1})_{\theta
,q}$ if $\theta \in \left[ \alpha (x_{\ast }),\beta (x_{\ast
})\right] $ and do not satisfy inequality $\beta _{0}(x_{\ast
})<\theta <\alpha _{\infty }(x_{\ast })$. So the theorem is true
when $\dim (\ker _{X_{0}+X_{1}}A)=1$ .

Suppose that properties a) and b) are satisfied for $\dim (\ker
_{X_{0}+X_{1}}A)=n$. We would like to prove that these properties
are satisfied when $\dim (\ker _{X_{0}+X_{1}}A)=n+1$.

If $\theta \in \Omega _{A}$ then from Theorem \ref{T1} and Theorem
\ref{ST2} we see that properties a) and b) are correct. So we need
just to consider the case when $\theta \notin \Omega _{A}$. In
this case it is possible to
find element $x_{\ast }\in \ker _{X_{0}+X_{1}}A$ such that $\theta \notin %
\left[ \alpha (x_{\ast }),\beta (x_{\ast })\right] $. Then the
operator
\begin{equation*}
A_{1}\colon X_{0}+X_{1}\rightarrow
(X_{0}+X_{1})/{\rm{span}}\left\{ x_{\ast }\right\}
\end{equation*}%
can be considered as an operator from the couple $(X_{0},X_{1})$
to the couple $(A_{1}(X_{0}),A_{1}(X_{1}))$. Note that the
operator $A$ can be written as a superposition of the operator
$A_{1}$ and some operator
$A_{2}$, i.e., $A=A_{2}A_{1}$, where $A_{2}$ is an operator from $%
(A_{1}(X_{0}),A_{1}(X_{1}))$ to $(Y_{0},Y_{1})$ with kernel equal to $%
A_{1}(\ker _{X_{0}+X_{1}}A)$. As an operator $A$ is invertible on
end spaces therefore the operators $A_{1}$ and $A_{2}$ are also
invertible on end spaces. Indeed, from invertibility of $A$ it
follows that $A_{1}$ is injective on end spaces. As it maps
$X_{i}$ onto $A_{1}(X_{i})$ for $i=0,1$, from the Banach theorem
of inverse operator it is invertible on end spaces. Similarly we
obtain that operator $A_{2}$ is invertible on end spaces.

As $\theta \notin \left[ \alpha (x_{\ast }),\beta (x_{\ast
})\right] $ therefore the operator $A_{1}$ is invertible on
$(X_{0},X_{1})_{\theta ,q}$. So from the injectivity $A$ on
$(X_{0},X_{1})_{\theta ,q}$ it follows that $A_{2}$ is injective
on $(A_{1}(X_{0}),A_{1}(X_{1}))_{\theta ,q}$ and
\begin{equation*}
A_{2}((A_{1}(X_{0}),A_{1}(X_{1}))_{\theta
,q})=A((X_{0},X_{1})_{\theta ,q}).
\end{equation*}%
So applying induction to the operator $A_{2}$ we obtain that codimension in $%
(Y_{0},Y_{1})_{\theta ,q}$ is not more than dimension of the $\ker
_{X_{0}+X_{1}}A_{2}$, which is equal to $n$ and is strictly less
than $\dim (\ker _{X_{0}+X_{1}}A)$.
\end{proof}


\subsection{Necessity condition for the class ${\mathbb{F}}_{\protect\theta %
,q}^{3}$}

The next result follows immediately from Theorem 11 in \cite{AK1}.

\begin{theorem}
Suppose that an operator $A\colon (X_{0},X_{1})\rightarrow
(Y_{0},Y_{1})$ is invertible on endpoint spaces, $\ker
_{X_{0}+X_{1}}A=V_{\theta ,q}^{0}+V_{\theta ,q}^{1}$ and the
restriction of the operator $A\colon (X_{0},X_{1})_{\theta
,q}\rightarrow (Y_{0},Y_{1})_{\theta ,q}$ is invertible. Then
\begin{equation*}
\beta (V_{\theta ,q}^{1})<\theta <\alpha \left( V_{\theta ,q}^{0}\right) {%
\text{.}}
\end{equation*}
\end{theorem}

\begin{proof} Since the operator $A\colon (X_{0},X_{1})_{\theta
,q}\rightarrow (Y_{0},Y_{1})_{\theta ,q}$ is invertible we have
$V_{\theta ,q}^{0}\cap V_{\theta ,q}^{1}=\left\{ 0\right\} $. Then
$\ker _{X_{0}+X_{1}}A=V_{\theta ,q}^{0}\oplus V_{\theta ,q}^{1}$,
so from Theorem 11 in \cite{AK1} it follows required inequality.
\end{proof}


\section{Duality}

Below we will give a~short proof of Theorem \ref{TT1} (necessity condition for
the class ${\mathbb{F}}_{\theta ,q}^{2}$) based on duality and Theorem \ref%
{TTT1} (necessity condition for the class ${\mathbb{F}}_{\theta ,q}^{1}$).
As we said before, duality requires some additional restrictions.

Recall that if $\vec{X}=(X_{0},X_{1})$ is a~Banach couple then the dual
couple $\vec{X^{\prime }}:=(X_{0}^{\prime },X_{1}^{\prime })$ is defined by
\begin{equation*}
X_{j}^{\prime }:=(X_{0}\cap X_{1},\Vert \cdot \Vert _{X_{j}})^{\ast },\quad
\,j=0,1.
\end{equation*}%
It is well known that $X_{0}^{\prime }+X_{1}^{\prime }=(X_{0}\cap
X_{1})^{\ast }$ isometrically. If $\vec{X}$ is regular, we furthermore have $%
X_{j}^{\prime }\cong X_{j}^{\ast }$ and $X_{0}^{\prime }\cap X_{1}^{\prime
}\cong (X_{0}+X_{1})^{\ast }$ isometrical isomorphisms.

If $A\colon (X_{0},X_{1})\rightarrow (Y_{0},Y_{1})$, then the dual operator $%
(A|_{X_{0}\cap X_{1}})^{\ast }\colon (Y_{0}\cap Y_{1})^{\ast
}\rightarrow (X_{0}\cap X_{1})^{\ast }$ is denoted by $A^{\prime
}$. It is well known that $A^{\prime }\colon (Y_{0}^{\prime
},Y_{1}^{\prime })\rightarrow (X_{0}^{\prime },X_{1}^{\prime })$.

We are now ready to prove the following.

\begin{theorem}
\label{N1T} Let $(X_0, X_1)$ and $(Y_0, Y_1)$ be regular Banach
couples. Suppose that an operator $A\colon (X_0,X_1)\to (Y_0,Y_1)$
is invertible and that the kernel $\mathrm{ker}_{X_0+X_1}A$ is
finite dimensional. Then

\begin{itemize}
\item[{\rm(i)}] $Y_{0}\cap Y_{1}=A(X_{0}\cap X_{1})\oplus A(P_{X_{0}}(%
\mathrm{ker}_{X_{0}+X_{1}}\,A))$ and dimension of $A(P_{X_{0}}(\mathrm{ker}%
_{X_{0}+X_{1}}A))$ is equal to the dimension of $\mathrm{ker}%
_{X_{0}+X_{1}}\,A$.

\item[{\rm(ii)}] If $A^{\prime }\colon (Y_{0}^{\prime
},Y_{1}^{\prime
})\rightarrow (X_{0}^{\prime },X_{1}^{\prime })$ is a dual operator, then $%
\mathrm{\dim }(\mathrm{ker}_{Y_{0}^{\prime }+Y_{1}^{\prime }}\,A^{\prime })=%
\mathrm{\dim }(\mathrm{ker}_{X_{0}+X_{1}}\,A))$ and the following formulas
hold\textrm{:}%
\begin{equation*}
\beta _{0}(\ker _{Y_{0}^{\prime } + Y_{1}^{\prime }}A^{\prime
})=\beta _{\infty }(\ker _{X_{0}+X_{1}}A),\quad \,\alpha _{\infty
}(\ker _{Y_{0}^{\prime} + Y_{1}^{\prime}}A^{\prime })=\alpha _{0}(\ker
_{X_{0}+X_{1}}A).
\end{equation*}
\end{itemize}
\end{theorem}

\begin{proof} (i). Our hypothesis that $A\colon (X_{0},X_{1})\rightarrow
(Y_{0},Y_{1})$ is invertible on endpoint spaces easily implies that $%
A(X_{0}\cap X_{1})$ is closed in $Y_{0}\cap Y_{1}$. Clearly, the
functor $\Delta (\vec{X}):=X_{0}\cap X_{1}$ has a~decomposition
property for any Banach couple $\vec{X}$. We have $V_{\Delta
}^{0}(A)=V_{\Delta }^{1}(A)=\{0\} $ by $X_{j}\cap \ker
_{X_{0}+X_{1}}A=\{0\}$ for $j=0,1$. Then the Theorem
\ref{lowerFredholm} with $\widetilde{V}:=\ker _{X_{0}+X_{1}}A$
gives
\begin{equation*}
Y_{0} \cap Y_{1}=A(X_{0}\cap X_{1})\oplus
A(P_{X_{0}}\widetilde{V}).
\end{equation*}%
Since $A(X_{0}\cap X_{1})$ is closed in $Y_{0}\cap Y_{1}$, we obtain the
required statement by the fact that $AP_{X_{0}}$ is an injective map (see
Remark \ref{REM}).

(ii). By (i) it follows that the range $A(X_{0}\cap X_{1})$ is closed and $%
(Y_{0}\cap Y_{1})/A(X_{0}\cap X_{1})$ is isomorphic to $\ker _{X_{0}+X_{1}}A$
and so
\begin{align*}
\mathrm{dim\,(ker}_{Y_{0}^{\prime }+Y_{1}^{\prime }}A^{\prime })& =\mathrm{%
dim\,(ker}_{(Y_{0}\cap Y_{1})^{\ast }}\,A^{\ast })=\mathrm{dim}\,A(X_{0}\cap
X_{1})^{\perp } \\
& =\mathrm{dim}\,\big ((Y_{0}\cap Y_{1})/A(X_{0}\cap X_{1})\big )^{\ast } \\
& =\mathrm{dim}\,(Y_{0}\cap Y_{1})/A(X_{0}\cap X_{1})=\mathrm{{\ \func{dim}%
\,(ker}}_{X_{0}+X_{1}}A).
\end{align*}%
Now we prove formulas for indices. By the definition of the index $\beta
_{0}(\ker _{Y_{0}^{\prime }+Y_{1}^{\prime }}A^{\prime })$ is the infimum of $%
\theta \in \left[ 0,1\right] $ such that there exists $\gamma >0$ such that
for all $\psi \in \ker _{Y_{0}^{\prime }+Y_{1}^{\prime }}A^{\prime }$ and
all $0<s<t\leq 1$ we have%
\begin{equation*}
\frac{K(s,\psi ;Y_{0}^{\prime },Y_{1}^{\prime })}{K(t,\psi ;Y_{0}^{\prime
},Y_{1}^{\prime })}\geq \gamma \left( \frac{s}{t}\right) ^{\theta }
\end{equation*}%
Combining this with the well-known duality formula
\begin{equation*}
K(s,\psi ;Y_{0}^{\prime },Y_{1}^{\prime })=\sup_{y\in Y_{0}\cap Y_{1}}\frac{%
|\left\langle \psi ,y\right\rangle |}{J(s^{-1},y;Y_{0},Y_{1})}
\end{equation*}%
with $Y_{0}\cap Y_{1}=A(X_{0}\cap X_{1})\oplus A(P_{X_{0}}(\ker
_{X_{0}+X_{1}}A))$ and $\psi \in \ker _{Y_{0}^{\prime }+Y_{1}^{\prime
}}A^{\prime }$, we conclude by $\left\langle \psi ,Au\right\rangle =0$ for
all $u\in X_{0}\cap X_{1}$,
\begin{align*}
K(s,\psi ;& Y_{0}^{\prime },Y_{1}^{\prime })=\sup_{v\in \ker
_{X_{0}+X_{1}}A}\left( \sup_{u\in X_{0}\cap X_{1}}\frac{|\left\langle \psi
,A(P_{X_{0}}v)\right\rangle |}{J(s^{-1},Au+A(P_{X_{0}}v);Y_{0},Y_{1})}\right)
\\
& \approx \sup_{v\in \ker _{X_{0}+X_{1}}A}\left( \sup_{u\in X_{0}\cap X_{1}}%
\frac{|\left\langle \psi ,A(P_{X_{0}}v)\right\rangle |}{\left\Vert
Au+A(P_{X_{0}}v)\right\Vert _{Y_{0}}+\frac{1}{s}\left\Vert
Au+A(P_{X_{0}}v)\right\Vert _{Y_{1}}}\right) \\
& \approx \sup_{v\in \ker _{X_{0}+X_{1}}A}\left( \frac{|\left\langle \psi
,A(P_{X_{0}}v)\right\rangle |}{\func{inf}_{u\in X_{0}\cap X_{1}}(\left\Vert
u+(P_{X_{0}}v)\right\Vert _{X_{0}}+s^{-1}\left\Vert
-u+(P_{X_{1}}v)\right\Vert _{X_{1}})}\right) \\
& =\sup_{v\in \ker _{X_{0}+X_{1}}A}\frac{|\left\langle \psi
,A(P_{X_{0}}v)\right\rangle |}{K(s^{-1},v;X_{0},X_{1}))}{\text{,}}
\end{align*}%
where the constants of equivalence depend only on the norms of the
inverses of the operator
$A$ on endpoint spaces. Let us consider the duality given by%
\begin{equation*}
\left\langle \psi ,v\right\rangle _{A}=\left\langle \psi
,A(P_{X_{0}}v)\right\rangle
\end{equation*}%
between $\psi \in \ker _{Y_{0}^{\prime }+Y_{1}^{\prime }}A^{\prime }$ and $%
v\in \ker _{X_{0}+X_{1}}$, then the equivalence obtained,
\begin{equation*}
K(s,\psi ;Y_{0}^{\prime },Y_{1}^{\prime })\approx \sup_{v\in \ker
_{X_{0}+X_{1}}A}\frac{\left\langle \psi ,A(P_{X_{0}}v)\right\rangle }{%
K(1/s,v;X_{0},X_{1}))},
\end{equation*}%
can be interpreted as the duality between $\ker _{X_{0}+X_{1}}A$ with norm $%
K(1/s,v;X_{0},X_{1})$ and $\ker _{Y_{0}^{\prime }+Y_{1}^{\prime
}}A^{\prime } $ with norm $K(s,\psi ;Y_{0}^{\prime },Y_{1}^{\prime
})$. So the inequality
\begin{equation*}
\frac{K(s,\psi ;Y_{0}^{\prime },Y_{1}^{\prime })}{K(t,\psi ;Y_{0}^{\prime
},Y_{1}^{\prime })}\geq \gamma \left( \frac{s}{t}\right)^{\theta},
\end{equation*}
which we rewrite as an inequality between equivalent norms%
\begin{equation*}
K(t,\psi ;Y_{0}^{\prime },Y_{1}^{\prime })\leq \gamma \left( \frac{t}{s}%
\right) ^{\theta }K(s,\psi ;Y_{0}^{\prime },Y_{1}^{\prime }),
\end{equation*}%
is equivalent to dual inequalities
\begin{equation*}
K(s^{-1},v;X_{0},X_{1})\leq \gamma \left( \frac{t}{s}\right) ^{\theta
}K(t^{-1},v;X_{0},X_{1}).
\end{equation*}%
Put $t^{\prime }=1/s$ and $s^{\prime }=1/t$; then we obtain
\begin{equation*}
\frac{K(s^{\prime },v;X_{0},X_{1})}{K(t^{\prime },v;X_{0},X_{1})}\geq \gamma
\left( \frac{s}{t}\right) ^{\theta }=\gamma \left( \frac{s^{\prime }}{%
t^{\prime }}\right) ^{\theta },
\end{equation*}%
where $1\leq s^{\prime }\leq t^{\prime }$. This means that we have
\begin{equation*}
\beta _{0}(\ker _{Y_{0}^{\prime}+Y_{1}^{\prime}}A^{\prime })=\beta
_{\infty }(\ker _{X_{0}+X_{1}}A).
\end{equation*}%
Similarly it can be proved that $\alpha _{\infty }(\ker _{Y_{0}^{^{\prime
}}+Y_{1}^{^{\prime }}}A^{\prime })=\alpha _{0}(\ker _{X_{0}+X_{1}}A)$.
\end{proof}

\subsection{A short proof of Theorem \protect\ref{TT1}}

Suppose that the couples $(X_{0},X_{1})$, $(Y_{0},Y_{1})$ are
regular and $1\leq q<\infty $. Let $A:(X_{0},X_{1})\rightarrow
(Y_{0},Y_{1})$ be invertible on endpoint spaces and let it belong
to the class ${\mathbb{F}}_{\theta ,q}^{2}$, i.e., $A\colon
(X_{0},X_{1})_{\theta ,q}\rightarrow (Y_{0},Y_{1})_{\theta ,q}$ be
an injective operator and $A\left( (X_{0},X_{1})_{\theta
,q}\right)$ a~closed subspace of finite codimension equal to the
dimension of $\ker _{X_{0}+X_{1}}A$. Therefore the dual operator
$A'\colon (Y_{0}',Y_{1}')\rightarrow (X_{0}', X_{1}')$ is
invertible on endpoint spaces and is a surjective operator as
operator from $(Y_{0}', Y_{1}')_{\theta ,q^{\prime }}$ to
$(X_{0}', X_{1}')_{\theta ,q^{\prime }}$, where $1/q +
1/q^{\prime}=1$. Since (see Theorem \ref{N1T})
\begin{equation*}
Y_{0} \cap Y_{1}=A(X_{0}\cap X_{1})\oplus A(P_{X_{0}}(\ker
_{X_{0}+X_{1}}A))
\end{equation*}%
then the kernel of $A'$ consists of all $\psi \in Y_{0}' + Y_{1}'$
such that $\psi (A(P_{X_{0}}(\ker _{X_{0}+X_{1}}A)))=0$.

Since (see Theorem \ref{lowerFredholm} and Remark \ref{REM})
\begin{equation*}
(Y_{0},Y_{1})_{\theta ,q}=A\left( (X_{0},X_{1})_{\theta ,q}\right) \oplus
A(P_{X_{0}}(\widetilde{V})),
\end{equation*}
then the dimension of $\widetilde{V}$ is equal to the codimension
of $A\left( (X_{0},X_{1})_{\theta ,q}\right) $ in
$(Y_{0},Y_{1})_{\theta ,q}$. Since an operator $A$ is in the class
${\mathbb{F}}_{\theta ,q}^{2}$, the codimension of its image is equal
to the dimension of its kernel. Combining this with
$\widetilde{V}\subset \ker _{X_{0}+X_{1}}A$ yields
\begin{equation*}
\widetilde{V}=\ker _{X_{0}+X_{1}}A
\end{equation*}%
and%
\begin{equation*}
(Y_{0},Y_{1})_{\theta ,q}=A\left( (X_{0},X_{1})_{\theta ,q}\right) \oplus
A(P_{X_{0}}(\ker _{X_{0}+X_{1}}A))\text{.}
\end{equation*}%
Moreover, since dimension of $A(P_{X_{0}}(\ker _{X_{0}+X_{1}}A))$ is
equal to the dimension of $\ker _{X_{0}+X_{1}}A$, which (see
Theorem \ref{N1T}) is equal to the dimension of $\ker _{Y_{0}' +
Y_{1}'}A'$ therefore kernel of $A'$ belongs to the space
$(Y_{0}', Y_{1}')_{\theta ,q^{\prime }}$. So we have for the operator $%
A'$ the following properties: \ it is invertible on endpoint spaces, it is
surjective on $(Y_{0}', Y_{1}')_{\theta ,q^{\prime }}$ and its
kernel is finite dimensional and contained in
$(Y_{0}', Y_{1}')_{\theta ,q^{\prime }}$. So from Theorem \ref{TTT1} we have%
\begin{equation*}
\beta _{\infty }(\ker _{Y_{0}' + Y_{1}'}A')<\theta <\alpha
_{0}(\ker _{Y_{0}' + Y_{1}'}A')\text{.}
\end{equation*}%
Then from Theorem \ref{N1T} we have%
\begin{equation*}
\beta _{0}(\ker _{X_{0}+X_{1}}A)<\theta <\alpha _{\infty }(\ker
_{X_{0}+X_{1}}A).
\end{equation*}%
This is exactly the statement of Theorem \ref{TT1}. Notice that we used
conditions: $1\leq q<\infty $ and couples $(X_{0},X_{1})$, $(Y_{0},Y_{1})$
are regular. We would like also mention that duality is not working for
quasi-Banach case.

\section{Concluding remarks and applications}

In the last section of the paper we discuss possible applications
to several areas which are connected with interpolation of
Fredholm operators.

\subsection{Lions-Magenes problem of interpolation of closed subspaces}

One of the most difficult and most important from the point of
view of applications in interpolation theory is the problem
related to interpolation of closed subspaces. It was stated as
early as in the monograph \cite{LM} by Lions and Magenes in which
the authors studied a~Hilbert space of Sobolev type connected with
elliptic boundary data. Let us first formulate exactly the problem
that is naturally to call the weak Lions-Magenes problem.

\noindent \textbf{Problem} \label{LMW} \emph{Suppose that $X_{i}$ are closed
complemented subspaces of $Y_{i}$ $(i=0,1)$. Find necessary and sufficient
conditions on parameters $\theta \in \left( 0,1\right) $, $q\in \left[
1,\infty \right] $ such that $(X_{0},X_{1})_{\theta ,q}$ is a closed
subspace of $(Y_{0},Y_{1})_{\theta ,q}$ of finite codimension.}

As it was mentioned in Introduction this problem was repeated by
J.-\,L.~Lions in Notices Amer.~Math.~Soc.~22~(1975), pp.~124-126
in the section related to Problems in Section $2$ of operators and
applications. We notice that various variants of this problem were
considered and special cases were investigated in by many authors
(see, e.g., \cite{AS, IK, KX, L1, L2, W}). We also refer to a
recent paper \cite{ACK} where interpolation of subspaces is
studied.

Now we shall explain how to apply results on interpolation of Fredholm
operators in order to get an affirmative solution of the mentioned problem:
Let us denote by $A$ an embedding operator $X_{0}+X_{1}$ in $Y_{0}+Y_{1}$.
Then by Lemma \ref{reduction} applied to the functor $(\cdot )_{\theta ,q}$,
we reduce the problem to the case when operator $\widetilde{A}\colon (%
\widetilde{X}_{0},\widetilde{X}_{1})\rightarrow (Y_{0},Y_{1})$ is
invertible. Thus we can use the factorization Theorem \ref{TN2} for $%
\widetilde{A}$ and Corollary 1, which allow us to describe all
parameters $\theta \in \left( 0,1\right) $, $q\in \left[ 1,\infty
\right) $ for which the mentioned weak Lions-Magenes problem
\ref{LMW} has a~positive solution. We note that it follows from
the proof of Lemma \ref{reduction} that an operator $\tilde{A}$ is
injective on all spaces $(X_{0},X_{1})_{\theta ,q}$ spaces and so
an operator $A_{1}$ does not appear in the factorization Theorem
\ref{TN2}.

It should be pointed out that there are known interesting
applications of results on interpolation on subspaces to various
problems, e.g., in differential equations and modern analysis. We
mention here that Kalton and Ivanov in \cite{IK} show remarkable
applications to the study of exponential Riesz bases in Sobolev
spaces.

\subsection{Spectral analysis}

Let $T\colon X\to X$ be an operator on a~complex Banach space $X$ and let $%
\sigma (T)$ (resp., $\rho (T)$) denote the~spectrum (resp., the resolvent)
of $T$. The~essential spectrum $\sigma _{{\mathit{ess}}}(T)$ is the~set of
all $\lambda \in {\mathbb{C}}$ such that $\lambda I-T$ is not Fredholm.
The~essential spectral radius is given by
\begin{equation*}
r_{{\mathit{ess}}}(T):=\sup \{|\lambda |;\,\lambda \in \sigma _{{\mathit{ess}%
}}(T)\}.
\end{equation*}
The classical Fredholm theory gives that the~set
\begin{align*}
\Lambda (T)=\Lambda (T\colon X\to X):=\{\lambda \in \sigma (T);\,|\lambda
|>r_{{\mathit{ess}}}(T)\}.
\end{align*}
is at most countable and consists of isolated eigenvalues of finite
algebraic multiplicity. Another way of expressing the~essential spectral
radius is
\begin{equation*}
r_{{\mathit{ess}}}(T)=\lim _{n\to \infty }\|T^n\|_{ess}^{1/n},
\end{equation*}
where $\|\cdot \|_{{\mathit{ess}}}$ denotes the~essential norm of $T$, i.e.,
the~distance in $L(X)$ from the~space $K(X)$ of compact operators on $X$, $%
\|T\|_{{\mathit{ess}}}= \func{inf}\{\|T-A\|;\,A\in K(X)\}$. We recall
formula due to Nussbaum \cite{Nussbaum}
\begin{align*}
r_{{\mathit{ess}}}(T)=\lim _{n\to \infty }\beta (T^n)^{1/n}.
\end{align*}

Now if $A\colon (X_{0},X_{1})\rightarrow (X_{0},X_1)$ is an
operator between complex Banach spaces and let
$A_{j}:=A|_{X_{j}}\colon X_{j}\rightarrow X_{j} $ for $j=0,1$.
Assume that $\lambda \in \Lambda (A_{0})\cap \Lambda (A_{1})$
(resp., $\lambda \in \rho (A_{0})\cap \rho (A_{1})$). In this case
an operator $A_{\lambda }:=A-\lambda I\colon
(X_{0},X_{1})\rightarrow (X_{0},X_{1})$ is Fredholm (resp.,
invertible). Thus our result could be applied to verify when
$A_{\lambda }\colon (X_{0},X_{1})_{\theta ,q}\rightarrow
(X_{0},X_{1})_{\theta ,q}$ is a~Fredholm operator.

\subsection{An example connected with integral equation of the second kind}

We show applications to the operator identity minus the Hardy
operator. On the space $L_{\mathrm{loc}}^{1}$ of locally
integrable on $(0,\infty )$ equipped with the Lebesgue measure we
define the Hardy operator $H$ by the formula:
\begin{equation*}
H\!f(t)=\frac{1}{t}\int_{0}^{t}f(s)\,ds,\quad \,f\in L_{\mathrm{loc}%
}^{1},\quad \,t>0.
\end{equation*}%
We will consider the operator $I-H$, which corresponds to the
integral equation of the second kind: for a~given function $g$
find function $f$ such that%
\begin{equation*}
f- H\!f = g.
\end{equation*}%
Note that the operator $I-H$ appears in many problems in analysis.
For example for important nonlinear map $f\mapsto f^{\ast \ast
}(t)-f^{\ast }(t)$ which is used to the study of symmetric
envelope of $B\!M\!O$ (see \cite{BS} for details) we have
\begin{equation*}
f^{\ast \ast }-f^{\ast }=(H-I)f^{\ast }{\text{.}}
\end{equation*}%
where $f^{\ast }$ if the decreasing rearrangement of $|f|$. We are
interested in of Fredholm property of the operator $I-H$ on the real
interpolation space $(L^{p}(\omega _{0}),L^{p}(\omega _{1}))_{\theta
,p}=L^{p}(\omega _{0}^{1-\theta }\omega _{1}^{\theta })$ of the couple $%
(L^{p}(\omega_{0}),L^{p}(\omega _{1}))$, where
\begin{equation*}
\left\Vert f\right\Vert _{L^{p}(\omega )}=\left( \int_{0}^{\infty
}\left\vert f(t)\,\omega (t)\right\vert ^{p}\frac{dt}{t}\right)
^{1/p}
\end{equation*}%
and $1\leq p<\infty $. For simplicity of presentation we will restrict
ourselves only to the special case
\begin{equation*}
\omega _{0}(t)=\left\{
\begin{array}{c}
t^{a_{0}}{\text{ , }}0<t\leq 1, \\
t^{a_{\infty }}{\text{, }}1<t<\infty ,%
\end{array}%
\right. {\text{, \ }}\omega _{1}(t)=\left\{
\begin{array}{c}
t^{b_{0}}{\text{ , }}0<t\leq 1, \\
t^{b_{\infty }}{\text{, }}1<t<\infty ,%
\end{array}%
\right.
\end{equation*}%
where $0<a_{0}$, $a_{\infty }<1,$ $b_{0}$, $b_{\infty }<0$ some
numbers. In this case operator $I-H$ has a one-dimensional kernel
in the sum
\begin{equation*}
\ker _{L^{p}(\omega _{0})+L^{p}(\omega _{1})}(I-H)=\mathrm{span}\left\{
f_{\ast }\right\} ,
\end{equation*}%
where $f_{\ast }(t)=1$ for all $t>0$. Then by using Hardy's inequalities
\begin{equation*}
\left( \int_{0}^{\infty }\Big |\frac{1}{t}\int_{0}^{t}f(s)\,ds\Big |%
^{p}s^{a}ds\right) ^{1/p}\leq \frac{p}{\left\vert p-a-1\right\vert }\left(
\int_{0}^{\infty }\left\vert f(t)\right\vert ^{p}t^{a}\,dt\right)
^{1/p},\quad \,{\text{if $a+1<p$}},
\end{equation*}

\begin{equation*}
\left( \int_{0}^{\infty }\Big |\frac{1}{t}\int_{t}^{\infty }f(s)\,ds\Big |%
^{p}s^{a}dt\right) ^{1/p}\leq \frac{p}{\left\vert p-a-1\right\vert }\left(
\int_{0}^{\infty }\left\vert f(t)\right\vert ^{p}t^{a}\,dt\right)
^{1/p},\quad \,{\text{if $a+1>p$}}
\end{equation*}%
it is possible to prove that the operator $I-H$ is bounded (see
\cite{KMP}), and even invertible on endpoint spaces $L^{p}(\omega
_{0})$ and $L^{p}(\omega _{1})$ and its
inverse is equals to $I-K_{0}$ on $L^{p}(\omega _{0})$ and $I+K_{1}$ on $%
L^{p}(\omega _{1})$, where
\begin{equation*}
(K_{0}f)(t)=\int_{t}^{\infty }f(s)\frac{ds}{s}{\text{, \ }}(K_{1}f)(t)=\frac{%
1}{t}\int_{0}^{t}f(s)\frac{ds}{s}{\text{.}}
\end{equation*}%
Then it is easily to get the following equivalence of the
$K$-functional for $f_{\ast}$:
\begin{equation*}
K(t,f_{\ast };L^{p}(\omega _{0}),L^{p}(\omega _{1}))\approx \left\{
\begin{array}{c}
t^{\frac{a_{0}}{a_{0}-b_{0}}}{\text{ , }}0<t\leq 1, \\
t^{\frac{a_{\infty }}{a_{\infty }-b_{\infty }}}{\text{, }}1<t<\infty .%
\end{array}%
\right.
\end{equation*}%
This implies that the dilation indices of the element $f_{\ast }$ are
\begin{equation*}
\alpha _{0}(f_{\ast })=\beta _{0}(f_{\ast
})=\frac{a_{0}}{a_{0}-b_{0}}{\text{ and \ }}\alpha_{\infty
}(f_{\ast}) = \beta_{\infty} (f_{\ast}) = \frac{a_{\infty
}}{a_{\infty }-b_{\infty }}.
\end{equation*}

Note that numbers $\frac{a_{0}}{a_{0}-b_{0}}$, $\frac{a_{\infty }}{a_{\infty
}-b_{\infty }}$ could be any numbers between $0$ and $1$. Thus combining the
above facts with our results on interpolation of Fredholm operators by real
method we obtain the following:


\begin{proposition}
\label{PE1} \ Suppose that $1\leq p<\infty $ and $0<a_{0},a_{\infty }<1$, $%
b_{0},b_{\infty }<0.$
\begin{itemize}
\item[\rm{(i)}] If $\frac{a_{0}}{a_{0}-b_{0}}<\frac{a_{\infty
}}{a_{\infty }-b_{\infty }} $ then the operator $I-H$ on the spaces
$(L^{p}(\omega _{0}),L^{p}(\omega
_{1}))_{\theta ,p}$ is {\rm{a)}} invertible if $0<\theta <\frac{a_{0}}{%
a_{0}-b_{0}}$ or $\frac{a_{\infty }}{a_{\infty }-b_{\infty }}%
<\theta <1$; {\rm{b)}} injective and its image has codimension one if $%
\frac{a_{0}}{a_{0}-b_{0}}<\theta <\frac{a_{\infty }}{a_{\infty
}-b_{\infty }} $; c) not Fredholm if $\theta
=\frac{a_{0}}{a_{0}-b_{0}}$ or $\theta
=\frac{a_{\infty}}{a_{\infty }-b_{\infty }}$.
\item[\rm{(ii)}] If $\frac{a_{\infty }}{a_{\infty }-b_{\infty }}<\frac{a_{0}}{%
a_{0}-b_{0}}$ then the operator $I-H$ on the spaces $(L^{p}(\omega
_{0}),L^{p}(\omega _{1}))_{\theta, p}$ is {\rm{a)}} invertible if $%
0<\theta <\frac{a_{\infty }}{a_{\infty }-b_{\infty }}$ or $\frac{%
a_{0}}{a_{0}-b_{0}}<\theta <1$; {\rm{b)}} surjective and has one
dimensional kernel if $\frac{a_{\infty }}{a_{\infty }-b_{\infty }}<\theta <\frac{a_{0}}{%
a_{0}-b_{0}}$; {\rm{c)}} not Fredholm if $\theta
=\frac{a_{0}}{a_{0}-b_{0}}$ or $\theta =\frac{a_{\infty
}}{a_{\infty }-b_{\infty }}$. \item[\rm{(iii)}] If
$\frac{a_{0}}{a_{0}-b_{0}}=\frac{a_{\infty }}{a_{\infty
}-b_{\infty }}$, then the operator $I-H$ is invertible on the spaces $%
(L^{p}(\omega _{0}),L^{p}(\omega _{1}))_{\theta ,p}$ if \ $\theta \neq $ $%
\frac{a_{0}}{a_{0}-b_{0}}$. It is not Fredholm for $\theta =\frac{a_{0}}{%
a_{0}-b_{0}}$.
\end{itemize}
\end{proposition}

\begin{proof} Since the operator $I-H$ is invertible on endpoint spaces and its
kernel in the sum $L^{p}(\omega _{0})+L^{p}(\omega _{1})$ is one
dimensional and spanned by element $f_{\ast }$ therefore operator
$I-H$ is invertible if
\begin{equation*}
\theta \notin \Omega _{A}=\left[ \alpha (f_{\ast }),\beta (f_{\ast })\right]
=\left[ \min \left( \frac{a_{0}}{a_{0}-b_{0}},\frac{a_{\infty }}{a_{\infty
}-b_{\infty }}\right) ,\max \left( \frac{a_{0}}{a_{0}-b_{0}},\frac{a_{\infty
}}{a_{\infty }-b_{\infty }}\right) \right]
\end{equation*}
(see Corollary \ref{Corollary2}). If $\frac{a_{0}}{a_{0}-b_{0}}<\theta <%
\frac{a_{\infty }}{a_{\infty }-b_{\infty }}$ then we have $\beta
_{0}(f_{\ast })<\theta <\alpha _{\infty }(f_{\ast })$ and result follows
from Theorem \ref{TN3} part b), if $\frac{a_{\infty }}{a_{\infty }-b_{\infty
}}<\theta <\frac{a_{0}}{a_{0}-b_{0}}$ then $\beta _{\infty }(f_{\ast
})<\theta <\alpha _{0}(f_{\ast })$ result follows from Theorem \ref{TN3}
part a). If $\theta =\frac{a_{0}}{a_{0}-b_{0}}$ or $\theta =\frac{a_{\infty }%
}{a_{\infty }-b_{\infty }}$ then from factorization Theorem
\ref{TN2} (here we use condition $p<\infty $) it follows that the
operator $I-H$ can be written in the form $I-H=A_{3}A_{2}A_{1}$,
where the operator $A_{i}$ belongs to the class
${\mathbb{F}}_{\theta ,q}^{i}$, $i=1,2,3$. From Theorem \ref{TN3}
it follows that such decomposition is impossible: when $\theta $
does not satisfy $\beta _{\infty }(f_{\ast })<\theta <\alpha
_{0}(f_{\ast })$, the operator $A_{1}$ does not exist; when
$\theta $ does not satisfy $\beta _{0}(f_{\ast })<\theta <\alpha
_{\infty }(f_{\ast })$, the operator $A_{2}$ does not exist; when
$\theta \in \left[ \alpha (f_{\ast }),\beta (f_{\ast })\right]$,
then from Corollary \ref{Corollary2} we see that the operator
$A_{3}$ also does not exist.
\end{proof}

Using the above result we can construct the following example.~Let us
consider the Banach couple $\left( X_{0},X_{1}\right) $, where
\begin{equation*}
X_{0}=L^{p}(\omega _{0}^{1})\times L^{p}(\omega _{0}^{1})\times
\cdot \cdot \cdot \times L^{p}(\omega _{0}^{n}), \quad\,
X_{1}=L^{p}(\omega _{1}^{1})\times L^{p}(\omega _{1}^{1})\times
\cdot \cdot \cdot \times L^{p}(\omega _{1}^{n}),
\end{equation*}%
and
\begin{equation*}
\omega _{0}^{i}(t)=\left\{
\begin{array}{c}
t^{a_{0}^{i}}{\text{ , }}0<t\leq 1, \\
t^{a_{\infty }^{i}}{\text{, }}1<t<\infty ,%
\end{array}%
\right. {\text{ \ , }}\omega _{1}^{i}(t)=\left\{
\begin{array}{c}
t^{b_{0}^{i}}{\text{ , }}0<t\leq 1, \\
t^{b_{\infty }^{i}}{\text{, }}1<t<\infty ,%
\end{array}%
\right. {\text{\ \ , }} 1\leq i \leq n.
\end{equation*}%
Let $A\colon \left( X_{0},X_{1}\right) \rightarrow \left( X_{0},X_{1}\right)
$ be a linear operator defined as
\begin{equation*}
A\left( f_{1},...,f_{n}\right) =\left( \left( I-H\right)
f_{1},...,(I-H)f_{n}\right) ,\quad \,(f_{1},...,f_{n})\in X_{0}+X_{1}.
\end{equation*}%


As an application of Proposition \ref{PE1} for such defined couple $\left(
X_{0},X_{1}\right) $ and operator $A$ we have the following result.


\begin{corollary}
\label{PE2}Suppose that $1\leq p<\infty $ and $0<a_{0}^{i},a_{\infty }^{i}<1$,
$b_{0}^{i},b_{\infty }^{i}<0$ for each $1\leq i\leq n$. Then operator $%
A\colon (X_{0},X_{1})\rightarrow (X_{0},X_{1})$ is invertible and has $n$%
-dimensional kernel in the sum $X_{0}+X_{1}$. Moreover $A\colon
(X_{0},X_{1})_{\theta ,p}$ has the following properties{\rm:}
\begin{itemize}
\item[{\rm(i)}] $A$ is a~Fredholm operator if and only if $\theta
\neq \frac{a_{0}^{i}}{a_{0}^{i}-b_{0}^{i}}$ and $\theta \neq \frac{a_{\infty }^{i}%
}{a_{\infty }^{i}-b_{\infty }^{i}}$for each $1\leq i\leq n$.
\item[{\rm{(ii)}}] $A$ is invertible on the space
$(X_{0},X_{1})_{\theta ,p}$ if and only if
\begin{equation*}
0<\theta <\min_{1\leq i\leq n}\,\min \Big \{\frac{a_{0}^{i}}{%
a_{0}^{i}-b_{0}^{i}},\frac{a_{\infty }^{i}}{a_{\infty }^{i}-b_{\infty }^{i}}%
\Big \} \quad\,\text{or}\,\,\,\,\max_{1\leq i\leq n}\,\max \Big \{\frac{a_{0}^{i}}{%
a_{0}^{i}-b_{0}^{i}},\frac{a_{\infty }^{i}}{a_{\infty }^{i}-b_{\infty }^{i}}%
\Big \}<\theta <1.
\end{equation*}
\item[{\rm(iii)}] Suppose that $\theta \neq \frac{a_{0}^{i}}{%
a_{0}^{i}-b_{0}^{i}}$ and $\theta \neq \frac{a_{\infty }^{i}}{a_{\infty
}^{i}-b_{\infty }^{i}}$for each $1\leq i\leq n$. Denote by $k$ number of $i$
such that $\frac{a_{\infty }^{i}}{a_{\infty }^{i}-b_{\infty }^{i}}<\theta <%
\frac{a_{0}^{i}}{a_{0}^{i}-b_{0}^{i}}$ and by $l$ number of $i$ such that $%
\frac{a_{0}^{i}}{a_{0}^{i}-b_{0}^{i}}<\theta <\frac{a_{\infty }^{i}}{%
a_{\infty }^{i}-b_{\infty }^{i}}$. Then the dimension of the
kernel of $A$ in $\left( X_{0},X_{1}\right)_{\theta ,p}$ equals
$k$ and the codimension of $A(\left( X_{0},X_{1}\right) _{\theta
,p})$ in $(X_{0},X_{1})_{\theta ,p}$ equals $m${\text{.}}
\end{itemize}
\end{corollary}

We finish this subsection with the following remark: the operator
$A$ defined above corresponds to the integral equation of the
second kind and if the integers $0\leq k,0\leq \ell $ are such
that $k+\ell \leq n$, then we can find numbers $0<a_{0}^{i},a_{\infty }^{i}<1$ and $%
b_{0}^{i},b_{\infty }^{i}<0$ for each $1\leq i\leq n$ such that for some $%
\theta \in \left( 0,1\right) $ the dimension of the kernel of $A$ in the
space $\left( X_{0},X_{1}\right) _{\theta ,p}$ is equal to $k$ and
codimension of $A(\left( X_{0},X_{1}\right) _{\theta ,p})$ in $\left(
X_{0},X_{1}\right) _{\theta ,p}$ is equal to $\ell $.

\subsection{An example connected to PDE's}

We conclude the paper with an example connected to PDE's in
domains with piecewise smooth boundaries that is interesting on
its own. Let consider the strip $\Pi=\left\{(x, y);\,
x\in\mathbb{R},y\in(0,\alpha)\right\} $, where $\alpha \in\left(
0,\pi\right)$ and a~family of weighted Sobolev spaces $W_{\beta
}^{l}$, which consists of the functions $f$ on the strip $\Pi$
with zero traces $f(\cdot,0)=f(\cdot,\alpha)=0$ and norm
\[
\left\Vert f\right\Vert _{W_{\beta }^{l}}=\max_{i+j\leq l}\bigg(
\int_{\Pi }\bigg( \frac{\partial ^{i}}{\partial
x^{i}}\frac{\partial ^{j}}{\partial y^{j}}\Big( e^{\beta
x}f(x,y)\Big)\bigg)^{2} dx\,dy\bigg)^{1/2},
\]
where the parameter $\beta$ is any real number and $l\in
\mathbb{Z}_{+}$. It is not difficult to show that%
\[
\big [W_{\beta_{0}}^{l},W_{\beta_{1}}^{l}\big] _{\theta} = \big(
W_{\beta_{0}}^{l},W_{\beta_{1}}^{l}\big)_{\theta,2}=
W_{(1-\theta)\beta_{0}+\theta\beta_{1}}^{l}\text{.}%
\]
It is known (see, for example, \cite{NP}) that the Laplace
operator maps $W_{\beta}^{l+2}$ to $W_{\beta}^{l}$ and it is
Fredholm and even invertible for all $\beta$ except
$\beta=\frac{k\pi}{\alpha}$, $k\in\mathbb{Z}\setminus \{0\}$, for
which the Laplace operator is not Fredholm. This example also
demonstrates that the conditions that the operator $A$ is Fredholm
on the end spaces is not sufficient to determine the Fredholm
property on interpolation spaces.

To explain connections with our results let us take
$X_{0}=W_{\beta_{0}}^{2}$, $X_{1}=W_{\beta_{1}}^{2}$, where
$\beta_{0}<\beta_{1}$ and $\beta_{0}$,
$\beta_{1}$ are not equal to $\frac{k\pi}{\alpha}$, $k\in\mathbb{Z}%
\setminus \{0\}  $, so the Laplace operator is invertible on
endpoint spaces, i.e., from $W_{\beta_{i}}^{2}$ to
$W_{\beta_{i}}^{0}$, $i=0,1$. In this concrete case the kernel of
the Laplace operator on the sum $X_{0}+X_{1}$ is finite
dimensional and has a basis that consists of functions
\[
f_{k}(x, y)=e^{-\frac{k\pi}{\alpha}x}\sin\frac{k\pi}{\alpha}y,
\quad\, (x, y)\in \Pi,
\]
where $k\in\mathbb{Z}\setminus \{0\} $ are such that $\beta
_{0}<\frac{k\pi}{\alpha}<\beta_{1}$ (see \cite{NP}). Moreover it
is possible to show that
\[
K(t,f_{k}; W_{\beta_{0}}^{2},W_{\beta_{1}}^{2})\approx
t^{\theta_{k}},
\]
where $\theta_{k}$ can be found from the equations%
\[
(1-\theta_{k})\beta_{0}+\theta_{k}\beta_{1}=\frac{k\pi}{\alpha}\text{.}%
\]

If $\theta=\theta_{k_{0}}$, then
\[
V_{\theta,2}^{0}= \text{span}\left\{f_{k};\,
\theta_{k}>\theta\right\}, \text{
}V_{\theta,2}^{1}=\text{span}\left\{ f_{k};\,
\theta_{k}<\theta\right\}\,\,\,\text{and} \,\,\, \widetilde{V}=
\text{span}\left\{ f_{k_{0}}\right\}
\]
and from Corollary \ref{TN5}, it follows that the Laplace operator is not Fredholm on%
\[
\big(W_{\beta_{0}}^{2},W_{\beta_{1}}^{2}\big)_{\theta_{k_{0}}%
,2}=W_{(1-\theta_{k_{0}})\beta_{0}+\theta_{k_{0}}\beta_{1}}^{2}=W_{\frac
{k_{0}\pi}{\alpha}}^{2}.
\]
Indeed, from $V_{\theta,2}^{0}\cap V_{\theta,2}^{1}=\left\{
0\right\}  $ it
follows that operator $A_{1}$ is an identity operator and we have%
\[
\beta_{0}(A_{1}(\widetilde{V}))=\beta_{0}(\widetilde{V})=\theta_{k_{0}}\,\,\,
\text{and}\,\,\, \alpha_{\infty}(A_{1}(\widetilde{V}))=\alpha_{\infty}(\widetilde{V})=\theta_{k_{0}%
}\text{.}%
\]
It means that condition (2.10) of Corollary \ref{TN5} is not
satisfied and Laplace operator is not Fredholm on
$W_{\frac{k_{0}\pi}{\alpha}}^{2}$. If $\theta\neq \theta_{k}$ for
each integer $k\neq 0$, by results in \cite{AK}, we conclude that
the Laplace operator is invertible on
$\big(W_{\beta_{0}}^{2},W_{\beta_{1}}^{2}\big)_{\theta,2}
=W_{(1-\theta)\beta_{0}+\theta\beta_{1}}^{2}$.

\vspace{5 mm}

{\bf Acknowledgements.} The authors thank to Dr.~Margaret
Stawiska-Friedland for corrections which improve the presentation.


\vspace{5 mm}

\noindent Department of Mathematics (MAI)\newline
Link\"{o}ping University, Sweden \newline
E-mail: \texttt{irina.asekritova@liu.se} \newline

\vspace{1.5 mm}

\noindent Department of Mathematics (MAI)\newline
Link\"{o}ping University, Sweden \newline
E-mail: \texttt{natan.kruglyak@liu.se}\newline

\vspace{1.5 mm}

\noindent Faculty of Mathematics and Computer Science\newline
Adam~Mickiewicz University; and Institute of Mathematics,\newline
Polish Academy of Science (Pozna\'{n} branch) \newline Umultowska
87, 61-614 Pozna{\'{n}}, Poland \newline E-mail:
\texttt{mastylo$@$math.amu.edu.pl}
\end{document}